\documentclass[12pt,reqno]{amsart}
\usepackage{graphicx}
\usepackage{amssymb,amsmath}
\usepackage{amsthm}
\usepackage{color,graphicx}
\usepackage{eucal}
\usepackage{hyperref}
\usepackage{color}
\usepackage{epstopdf}
\usepackage{mathabx}

 \usepackage[normalem]{ulem}
 \usepackage{float}

\newtheorem{prop}{Proposition}[section]

\newtheorem{obs}{Remark}[section]

\newcommand{\ve}{{\varepsilon}}

\newcommand{\Hper}{{H_{per}^{\frac{\alpha}{2}}(\mathbb{T})}}
\newcommand{\Hn}{{H_{per}^{\frac{\alpha}{2}}}}
\newcommand{\Hpers}{{H_{per}^{s}(\mathbb{T})}}
\newcommand{\Hns}{{H_{per}^{s}}}

\numberwithin{equation}{section}
\numberwithin{figure}{section}

\newtheorem{theorem}{Theorem}[section]
\newtheorem{proposition}[theorem]{Proposition}
\newtheorem{remark}[theorem]{Remark}
\newtheorem{lemma}[theorem]{Lemma}

\newtheorem{definition}[theorem]{Definition}

\begin{document}


\title[ Periodic waves for the fractional BBM]{On the existence, uniqueness and stability of Periodic Waves for the fractional Benjamin-Bona-Mahony equation}

\subjclass[2000]{76B25, 35Q51, 35Q53.}

\keywords{BBM
	type-equations, existence and uniqueness of minimizers, spectral stability, orbital stability, Petviashvili's method.}
\thanks{$^*${\it Corresponding author}\\
	\indent {\it Date}: June, 2021.}

\author{Sabrina Amaral}

\address[S. Amaral]{Department of Mathematics, State University of
	Maring\'a, Maring\'a, PR, Brazil. }
\email{sabrinasuelen20@gmail.com}

\author{Handan Borluk}

\address[H. Borluk]{Ozyegin University, Department of Natural and Mathematical Sciences, Cekmekoy, Istanbul, Turkey. }
\email{handan.borluk@ozyegin.edu.tr }

\author{Gulcin M. Muslu}

\address[G. M. Muslu]{Istanbul Technical University, Department of Mathematics, Maslak,
	Istanbul,  Turkey.}
\email{gulcin@itu.edu.tr }

\author{F\'abio Natali$^*$}

\address[F. Natali]{Department of Mathematics, State University of
	Maring\'a, Maring\'a, PR, Brazil. }
\email{fmanatali@uem.br }

\author{Goksu Oruc}

\address[G. Oruc]{Istanbul Technical University, Department of Mathematics, Maslak,
         Istanbul,  Turkey.}
\email{topkarci@itu.edu.tr }

\maketitle

\begin{abstract}
The  existence, uniqueness and stability of periodic traveling waves for the fractional Benjamin-Bona-Mahony equation is considered. In our approach, we give sufficient conditions to prove a uniqueness result for the single-lobe solution obtained by a constrained minimization problem. The spectral stability is then shown by determining that the associated linearized operator around the wave restricted to the orthogonal of the tangent space related to the momentum and mass at the periodic wave has no negative eigenvalues. We propose the Petviashvili's method to investigate the spectral stability of the  periodic waves for the fractional Benjamin-Bona-Mahony equation, numerically. Some remarks concerning the orbital stability of periodic traveling waves are also presented.
\end{abstract}

\section{Introduction.}
In this paper, we show the spectral/orbital stability of zero mean periodic traveling wave solutions associated to the fractional Benjamin-Bona-Mahony (fBBM) type equation
	\begin{equation}\label{rbbm}
	u_t +u_x +uu_x +(D^{\alpha}u)_t=0.
	\end{equation}
Here $u=u(x,t):\mathbb{R \times \mathbb{R} \rightarrow \mathbb{R}}$ is a $2\pi$-periodic function  at the $x-$variable and $D^{\alpha}$ represents the fractional differential operator defined as a Fourier multiplier by
	\begin{equation}\label{fracderivative}
	\widehat{D^{\alpha}g}(\xi)= |\xi|^{\alpha}\widehat{g}(\xi), \ \ \xi \in \mathbb{Z},
	\end{equation}
where $0< \alpha \leq 2$.\\
\indent For the case $\alpha=2$, we have the well known Benjamin-Bona-Mahony (BBM) equation arising as an improvement of the Korteweg-de Vries equation for modeling long surface gravity waves of small amplitude which propagate unidirectionally.
It also describes the propagation of long waves with a balance of nonlinear
and dissipative effects. In addition, the model appears  in the analysis of the surface waves of long wavelength in liquids, hydromagnetic
waves in cold plasma, acoustic-gravity waves in compressible
fluids, and acoustic waves in harmonic crystals (see \cite{BBM} and related references).\\
\indent If $\alpha=1$, we obtain the regularized Benjamin Ono equation (rBO) which is a regularized version of the standard Benjamin-Ono equation. The rBO equation is a model for the time evolution of long-crested waves at the interface between two immiscible fluids. Some situations in which the equation is useful are the pycnocline in the deep ocean, and the two-layer system created by the inflow of fresh water from ariver into the sea \cite{kalisch}.

The fBBM equation \eqref{rbbm} admits the following conserved quantities, smooth in their domains, as

	\begin{equation}\label{quantconserved}
	\displaystyle E(u)= \frac{1}{2}\int_{-\pi}^{\pi}\big[ (D^{\frac{\alpha}{2}}u)^2 - \frac{1}{3}u^3\big] dx
	\end{equation}
	
		\begin{equation}\label{momentum}
	\displaystyle P(u) = \frac{1}{2}\int_{-\pi}^{\pi} \big[ (D^{\frac{\alpha}{2}}u)^2 +u^2\big]dx,
	\end{equation}
and
	\begin{equation}\label{mass}
	\displaystyle M(u)= \int_{-\pi}^{\pi} u dx.
	\end{equation}
\indent A traveling wave solution for \eqref{rbbm} is of the form $u(x,t)= \phi(x- ct)$, where $\phi:\mathbb{T} \rightarrow \mathbb{R}$, with $\mathbb{T}= [-\pi, \pi]$, is a smooth $2\pi$-periodic function and $c$ represents the wave speed. Substituting this form into the fBBM equation \eqref{rbbm}, we obtain
	\begin{equation}\label{ode1}
	cD^{\alpha}\phi + (c-1)\phi -\frac{1}{2}\phi^2 +A =0,
	\end{equation}
where $A$ is a constant of integration. If we suppose that $\phi:\mathbb{T} \rightarrow \mathbb{R}$ is a periodic function with the zero mean property, then $A = A(c)$ is defined by
	\begin{equation}\label{integconst}
	\displaystyle A(c) :=  \frac{1}{4\pi}\int_{-\pi}^{\pi} \phi^2(x) dx.
	\end{equation}
Combining  \eqref{ode1} and \eqref{integconst}, we see that the equation $(\ref{ode1})$ can be expressed by the following boundary-value problem
	\begin{equation}\label{ode2}
	cD^{\alpha}\phi + (c-1)\phi =\frac{1}{2}\Pi_0\phi^2, \ \ \phi \in H^{\alpha}_{per}(\mathbb{T}),
	\end{equation}
where $\Pi_0 f := f- \frac{1}{2\pi}\int_{-\pi}^{\pi} f(x) dx$ is the projection operator restricted to the mean zero $2\pi$-periodic functions.

The conserved quantities \eqref{quantconserved}, \eqref{momentum} and \eqref{mass} play an important role in our spectral stability analysis. In fact, they allow us to consider the augmented Lyapunov functional
	\begin{equation}\label{lyapfunct}
	G(u)= E(u) + (c-1)P(u) + AM(u),
	\end{equation}
where $G'(\phi)=0$, that is, $\phi$ is a critical point of $G$. Moreover, from \eqref{lyapfunct} we obtain the Hessian operator around the wave $\phi$, which is commonly called in the current literature, as linearized operator around $\phi$ and given by
	\begin{equation}\label{operator}
	\mathcal{L}:= G''(\phi) = cD^{\alpha} + c - 1 - \phi.
	\end{equation}
It is easy to see that $\mathcal{L}$ is a self-adjoint operator defined in $L^2_{per}(\mathbb{T})$ with dense domain $H_{per}^{\alpha}(\mathbb{T})$.

Next, we establish our linearized spectral problem for the fBBM equation. Indeed, substituting $u(x,t)= \phi(x-ct)+v(x-ct, t)$ into \eqref{rbbm} and using that \eqref{ode2} is satisfied by $\phi$, we obtain that $v$ is a solution of the nonlinear equation
	\begin{equation}\label{eqnonlinear}
	(\partial_t- c\partial_x)(v + D^{\alpha}v) + \partial_x(v+ \phi v) + v \partial_x v=0.
	\end{equation}
Replacing \eqref{eqnonlinear} by its linearization about $\phi$ we obtain the basic spectral stability problem
	\begin{equation}\label{linearproblem}
	\partial_t(v + D^{\alpha}v) = \partial_x \mathcal{L}v
	\end{equation}
where $\mathcal{L}$ is given by \eqref{operator}. Since $\phi$ depends on $x$ and but not $t$, the equation \eqref{linearproblem} has a separation of variables of the form $v(x,t)= e^{\lambda t}\eta(x)$ with some $\lambda \in \mathbb{C}$ and $\eta: \mathbb{T} \rightarrow \mathbb{C}$. These arguments allow us to consider the spectral problem
	\begin{equation}\label{spectproblem}
	\partial_x \mathcal{L}\eta = \lambda (1 + D^{\alpha})\eta.
	\end{equation}
\indent We can rewrite the spectral problem in $(\ref{spectproblem})$ as
\begin{equation}
	J \mathcal{L}\eta= \lambda \eta,
	\label{hamilt1}\end{equation}
where $J:=(1+D^{\alpha})^{-1}\partial_x$. Let us denote the spectrum of $J\mathcal{L}$ by $\sigma(J\mathcal{L})$. In a general setting, the periodic wave $\phi\in H^{\alpha}_{per}(\mathbb{T})\cap X_0$ is said to be spectrally stable if $\sigma(J\mathcal{L}) \subset i\mathbb{R}$ in  $L^2_{per}(\mathbb{T})$. Otherwise, that is, if $\sigma(J\mathcal{L})$ in $L^2_{per}(\mathbb{T})$ contains a point $\lambda$ with $Re(\lambda)>0$, the periodic wave $\phi$ is said to be spectrally unstable.\\
 \indent According to the classical spectral stability results as in \cite{grillakis1}, the problem given by $(\ref{hamilt1})$ can not be handled in periodic context since $J= (1+D^{\alpha})^{-1}\partial_x$ is not a one-to-one operator. To overcome this difficulty, we can consider the linearization of $(\ref{rbbm})$ around $\phi$ restricted to the space of zero mean periodic function to obtain the modified problem (see \cite{DK})
	\begin{equation}\label{hamilt2}
	J\mathcal{L}|_{X_0}\chi = \lambda \chi,
	\end{equation}
where $\mathcal{L}|_{X_0}=\mathcal{L}+\frac{1}{2\pi}\langle\phi,\cdot\rangle$ is a restriction of $\mathcal{L}$ on the closed subspace $X_0$ of periodic functions with zero mean,
	\begin{equation}
	\displaystyle X_0 = \left\{f \in L^2_{per}(\mathbb{T}): \int_{-\pi}^{\pi} f(x)dx =0 \right\}.
	\end{equation}
Clearly, operator $\mathcal{L}|_{X_0}$ is defined in the space $X_0$ with domain $H_{per}^{\alpha}(\mathbb{T})\cap X_0$. The new spectral problem $(\ref{hamilt2})$ allows to consider a new definition for the spectral stability restricted to the periodic space $X_0$ as:
\begin{definition}\label{defistab1}
	The periodic wave $\phi \in H^{\alpha}_{per}(\mathbb{T})\cap X_0$ is said to be spectrally stable if $\sigma(J\mathcal{L}|_{X_0}) \subset i\mathbb{R}$ in  $X_0$. Otherwise, that is, if $\sigma(J\mathcal{L}|_{X_0})$ in $X_0$ contains a point $\lambda$ with $Re(\lambda)>0$, the periodic wave $\phi$ is said to be spectrally unstable.
\end{definition}

\begin{remark}
	In Definition $\ref{defistab1}$, we consider a smooth zero mean periodic wave $\phi$. However, it is possible to obtain the spectral stability/instability in the space $X_0$ using the same definition for non-zero mean periodic waves as reported, for instance, in \cite{DK} and \cite{hur}.
\end{remark}
\indent In  \cite{hur}, the existence of minimizers for the energy functional $V:=P-E$ with fixed momentum $P$ and mass $M$ has been established. The periodic wave obtained by this minimization problem is smooth in terms of the independent parameters $A$ and $c$ in $(\ref{ode1})$ and it has been determined that they are spectrally stable in the sense of Definition $\ref{defistab1}$. The orbital stability of the periodic minimizers is then established by assuming that the 2-by-2 determinant $\{P,M\}_{A,c}$ is nonzero (see the precise definition of orbital stability in the last section of this paper). We also refer the reader to see \cite{CNP} for a different approach. However, in the referred work, a different minimization problem compared to \cite{hur} has been established and the authors obtained a different criterium for the orbital stability in the energy space.\\
\indent For the cases $\alpha=1$ and $\alpha=2$, precise results of orbital stability for the equation $(\ref{ode1})$ with $A=0$ were determined in \cite{an-ban-scia1} combined that $\mathcal{L}$ has only one negative eigenvalue which is simple and zero is a simple eigenvalue whose associated eigenfunction is $\phi'$. In order to show that $\mathcal{L}$ has only one negative eigenvalue which is simple and zero is a simple eigenvalue with associated eigenfunction $\phi'$ they have used the Fourier expansion of the explicit solutions together with the main theorem in \cite{natali}.
In addition, the explicit solutions were also used to calculate the positiveness of $\frac{d}{dc}P(\phi)$.\\
\indent Concerning the case $\alpha=2$ in $(\ref{rbbm})$ with a general power nonlinearity as
\begin{equation}\label{pBBM}
	u_t+u_x+(p+1)u^pu_x-u_{xxt}=0,
\end{equation}
where $p\geq1$ is an integer and traveling $\phi$ waves satisfying the equation
\begin{equation}\label{pBBM1}
	-c\phi''+(c-1)\phi-\phi^{p+1}=0,
\end{equation}
existence and stability results of small amplitude periodic waves near of the equilibrium solution $(c-1)^{1/p}$ for $c>1$ have been considered in \cite{haragus}. These waves are spectrally stable for all speeds $c>1$ when $1\leq p \leq 2$ in both $L^2(\mathbb{R})$ and $C_{b}(\mathbb{R})$. Here, $C_{b}(\mathbb{R})$ stands the space of bounded uniformly continuous functions. For $p\geq3$, there exists a critical speed $c_p=\frac{p(p+4)}{p^2+p+6}$, $1<c_p<\frac{p}{p-3}$, where the waves are stable for $c\in \left(c_p,\frac{p}{p-3}\right)$ and spectrally unstable for $c\in (1,c_p)\cup \left(\frac{p}{p-3},\infty\right)$ by considering the same perturbations. For perturbations which have the same period $L$ as the wave, all small amplitude periodic waves are spectrally stable for all values of $p$.\\
\indent For the model $(\ref{rbbm})$ posed on unbounded domains, the author in \cite{angulo1} showed the existence of solitary waves employing the arguments in \cite{FL}. He used the abstract approach in \cite{grillakis1} (\cite{lin}) and determined the orbital (spectral) stability of solitary waves if:\\
i) $c>1$ and $\alpha\in \left[\frac{1}{2},1\right)$,\\
ii) $c>c_0:=\frac{2+\sqrt{2(3\alpha-1)}}{6\alpha}>1$ and $\alpha\in \left(\frac{1}{3},\frac{1}{2}\right)$.\\
If $\alpha\in\left(\frac{1}{3},\frac{1}{2}\right)$ and $c\in (1,c_0)$, the solitary wave is said to be spectrally unstable using the arguments in \cite{lin}. His proof relies on the scaling argument for the solitary wave $\phi$ which gives  good features about the  Hessian matrix $\frac{d}{dc}P(\phi)$. The same method in \cite{angulo1} can be used to establish the orbital (spectral) stability of solitary waves for the case $\alpha\in[ 1,2]$ and $c>1$. \\
\indent We can cite some additional contributions concerning the spectral/orbital stability of periodic waves for related equations to $(\ref{rbbm})$ and other dispersive equations (see \cite{ANP}, \cite{ACN}, \cite{natali}, \cite{CJ2019}, \cite{DK}, \cite{gallay}, \cite{gavage}, \cite{hakkaev12}, \cite{johnson13}, \cite{NPU1}, \cite{NPL}, and references therein). In some cases, the study of the stability was considered in relation to perturbations of the same period of the wave and other types of perturbations (in the latter case, we can consider: anti-periodic, bounded and so on). \\
\indent  The main results of our work can be presented according to the following theorem:
\begin{theorem}\label{maintheorem}
Let $\alpha\in\left(\frac{1}{3},2\right]$ and $\tau>0$ be fixed. For every $c>\frac{1}{2}$, there exists an even and periodic single-lobe profile $\phi\in H_{per}^{\alpha}(\mathbb{T})$ which is solution of the constrained minimization problem,
\begin{equation}\label{minp}
	\inf_{u\in \Hper}\left\{\int_{-\pi}^{\pi}c(D^{\frac{\alpha}{2}}u)^2+(c-1)u^2dx;\ \int_{-\pi}^{\pi}u^3dx=\tau,\ \int_{-\pi}^{\pi}udx=0\right\},
\end{equation}
and a solution of $(\ref{ode2})$. If $\ker\left(\mathcal{L}|_{X_0}\right)=[\phi']$ for all $c>\frac{1}{2}$, we obtain a smooth curve $c\in\left(\frac{1}{2},+\infty\right)\mapsto\phi\in H_{per}^{\alpha}(\mathbb{T})\cap X_0$ of even periodic waves with fixed period and furthermore, denoting $d:=1+2A(c)-c-cA'(c)$, we have,\\

\noindent i) for the linearized operator $\mathcal{L}$: a simple negative eigenvalue if $d<0$ and two negative eigenvalues if $d\geq0$,\\
ii) zero is a simple eigenvalue of $\mathcal{L}$ if and only if $d\neq0$, \\
iii) $\phi$ is the unique solution of the problem $(\ref{minp})$ for the case $\alpha\in\left(\frac{1}{2},2\right],$\\
iv) the periodic wave $\phi$ is spectrally stable if:
\begin{itemize}
	\item $d\neq0$ and $A'(c)+ \frac{1}{\pi}\mathcal{B}_c(\phi)>0$,
	\item $\alpha\in\left(\frac{1}{2},2\right]$ and $d\neq0$ for all $c>\frac{1}{2}$,
\end{itemize}
v) the periodic wave $\phi$ is orbitally stable if $\alpha\in\left(1,2\right]$ and $d\neq0$ for all $c>\frac{1}{2}$.

\end{theorem}
\indent The first two parts in Theorem $\ref{maintheorem}$ concern the existence of periodic minimizers for the problem $(\ref{minp})$ and the existence of smooth curve of solutions with fixed period. To prove the first part, we need to use the Poincar\'e-Wirtinger inequality to find the bifurcation point $c=\frac{1}{2}$ and compactness tools to establish that the periodic waves exist for all $c>\frac{1}{2}$. The second part is similar to \cite[Lemma 3.8]{NPL} and it can be determined by using the crucial hypothesis $\ker\left(\mathcal{L}|_{X_0}\right)=[\phi']$ and the implicit function theorem.\\
\indent  Assumption $\ker\left(\mathcal{L}|_{X_0}\right)=[\phi']$ is also important to give precise informations about items i)-v) in Theorem $\ref{maintheorem}$. In fact, item i) is determined since we show that the our periodic waves solve the minimization problem $(\ref{minp})$. Thus, we combine this information with convenient index formulas  (see \cite[Theorem 4.1]{pel-book}) to obtain that $n(\mathcal{L})=1$, if $d<0$ and $n(\mathcal{L})=2$, if $d\geq0$. Here, $n(\mathcal{L})$ stands for the number of negative eigenvalues of $\mathcal{L}$. The value of $d$ is crucial in our analysis since it determines the existence of \textit{fold points}, that is, values of $c$ (and depending on $\alpha$) such that $d=0$. Folds points are related with the existence of additional elements in $\ker(\mathcal{L})$, besides $\phi'$, and they give the exact value of $c$ where the number of negative eigenvalues change. It is worth mentioning that they were first studied in \cite{NPL} for the case of fractional Korteweg-de Vries (fKdV) equation. \\
\indent Second item can be established since the operator $\mathcal{L}$ in $(\ref{operator})$ satisfies an \textit{Oscillation Theorem} for fractional linear operators with smooth periodic potentials according to \cite{FL}, \cite{hur} and \cite{NPL}. In fact, we obtain that $\dim(\ker(\mathcal{L}))\in\{1,2\}$ and $\dim(\ker(\mathcal{L}))=1$ if and only if $d\neq0$.\\
\indent Regarding item iii). Using the arguments in \cite{FL}, our work establishes sufficient conditions for the \textit{uniqueness of periodic minimizers} associated to the problem $(\ref{minp})$. In what follows, let $\alpha_0\in(\frac{1}{2}, 2]$ be fixed. Assume that $(\phi_{0}, c_{0})\in V\times \left(\frac{1}{2},+\infty\right)$ with $\phi_0$ being a non-zero solution of $(\ref{ode2})$ and satisfying assumption $(\ref{hyp-uniq})$ with $\alpha=\alpha_0$ and $c=c_{0}$. The implicit function theorem guarantees the existence of a $\delta>0$ and an $C^1-$map $\alpha\in I=[\alpha_0,\alpha_0+\delta) \mapsto (\phi_{\alpha},c_{\alpha})$ such that $(\phi_{\phi},c_{\alpha})$ is the unique local solution of $(\ref{ode2})$ in a convenient neighbourhood around $(\phi_0,c_0)$. In addition, it is possible to obtain for all $\alpha\in I$ that
\begin{equation}\label{L3}
	\int_{-\pi}^{\pi}\phi_{\alpha}^3dx=\int_{-\pi}^{\pi}\phi_{0}^3dx.
\end{equation}
Using compactness arguments, we prove that the maximal branch $(\phi_{\alpha},c_{\alpha})$ and equality $(\ref{L3})$ extends to the interval $[\alpha_0,2)$. Next, since it is well know that for $\alpha=2$ and a fixed $c>\frac{1}{2}$ the solution with dnoidal profile (see $(\ref{dnsol})$) is unique, we obtain the uniqueness of minimizers for the problem $(\ref{minp})$ using the similar arguments as in \cite[Theorem 2.4]{FL}. \\
\indent We determine item iv) employing again the index formula to get
$$n(\mathcal{L}|_{\{1,(D^{\alpha}+1)\phi\}^{\bot}})=n(\mathcal{L})-n_0-z_0,$$
where $n_0$ is the number of negative eigenvalues determined by the Hessian symmetric matrix formed by momentum and mass and $z_0$ indicates the dimension of the kernel of the same Hessian matrix. To simplify the analysis, $n_0$ and $z_0$ can be both precisely determined by analysing the following quantity for the case $d\neq0$ (see Section 5).
\begin{equation}\label{detS0}\det S(0):= \frac{4\pi^2d^{-1}}{c} \left(A'(c)+ \frac{1}{\pi}\mathcal{B}_c(\phi) \right).\end{equation}
 \indent To obtain the precise statements for this item, we need to discuss the existence (and spectral stability) of periodic small amplitude waves associated with the equation $(\ref{rbbm})$ for the case $\alpha\in (0,2]$. For a small amplitude parameter $a$, we put forward the explicit solutions of  $(\ref{ode2})$ as
\begin{equation}\label{exp-sol1}
	\phi(x)= a\cos(x) + \frac{a^2}{2(2^{\alpha} -1)} \cos(2x) + O(a^3).
\end{equation}
The wave speed $c$ and parameter $A$ are given respectively by
\begin{equation}\label{exp-speed1}
	c= \frac{1}{2} + \frac{a^2}{4(2^\alpha +1)(2^\alpha -1)} + O(a^4),\ \ \ \mbox{and}\ \ \
	A(c) = \frac{a^2}{4} + O(a^4).
\end{equation}
Orbital/spectral stability and related topics of small amplitude periodic waves associated to several evolution models have been exhaustively studied  (see \cite{haragus1}, \cite{HP}, \cite{johnson13}, \cite{LP}, \cite{NPL}, and references therein). In our paper, we show the existence of two fold points: $\alpha=\tilde{\alpha}\approx 0.2924$ and $\alpha=\frac{1}{2}$. We prove that $n(\mathcal{L})=1$, if $\alpha\in(0,\tilde{\alpha})\cup (\frac{1}{2},2]$ and $n(\mathcal{L})=2$, if $\alpha\in(\tilde{\alpha},\frac{1}{2})$. To determine $A'(c)$ in $(\ref{detS0})$, we can use both expressions in $(\ref{exp-speed1})$ to get
\begin{equation}\label{deriv-ci}
	A'(c)= (2^\alpha +1)(2^\alpha -1) + O(a^3).
\end{equation}
Since $A'(c)>0$ for $\alpha\in (0,2]$, it follows that $n_0$ and $z_0$ are determined only by the sign of $d$. In fact, if $d<0$, one has $n_0=1$ and $z_0=0$, while $d>0$ gives us that $n_0=2$ and $z_0=0$. In both cases, we obtain the spectral stability of the small amplitude periodic wave $\phi$ $(\ref{exp-sol1})$. The degenerate case $d=0$ is also established and the wave is said to be spectrally stable. For the case $\alpha=2$ in $(\ref{rbbm})$, our results are compatible with those ones in \cite{haragus} when $p=1$ in equation $(\ref{pBBM})$ but, in our case, the small amplitude periodic waves are near of the equilibrium solution $\phi=0$.\\
\indent Using the spectral information obtained by the small amplitude periodic waves, we establish that the number of negative eigenvalues of the linearized operator $\mathcal{L}$ is always equal to one if and only if $\dim(\ker({\mathcal{L}}))=1$ for all $c>\frac{1}{2}$. In addition, since $d<0$ for the periodic wave obtained by the problem $(\ref{minp})$, we conclude $A(c)>c-\frac{1}{2}$ and this bound allows us to deduce that $\det S(0)$ in $(\ref{detS0})$ is always negative. This fact gives us the spectral stability and moreover, it recovers the same result as in \cite{angulo1} without using scaling argument. In addition, we obtain periodic travelling wave solutions of the equation \eqref{rbbm} numerically by using a Petviashvili's iteration method. The numerical results also confirm the analytical results stating that  the periodic wave is spectrally stable for $\alpha\in\left(\frac{1}{2},2\right]$.\\
\indent Next, we discuss the case $\alpha\in (\frac{1}{3},\frac{1}{2}]$. In fact, numerical experiments will be used to decide the exact sign of the quantities $d$ and $A'(c)+ \frac{1}{\pi}\mathcal{B}_c(\phi)$ in $(\ref{detS0})$ in terms of $c>\frac{1}{2}$. Our numerical results point out that there exists a critical wave speed $c^*\in (\frac{1}{2}, +\infty)$ at which the sign of $d$ changes.
For $c\in(c^*,+\infty)$, we observe numerically that $d<0$ and  $A'(c)+ \frac{1}{\pi}\mathcal{B}_c(\phi) >0$, therefore the periodic wave is spectrally stable. On the other hand, for $c\in(\frac{1}{2}, c^*)$, we have $d>0$ and  $A'(c)+ \frac{1}{\pi}\mathcal{B}_c(\phi)>0$.  Thus $\det S(0)>0$ which implies $n_0=0$ or $n_0=2$. Since the first entry of the matrix $S(0)$ is negative, we obtain $n_0=2$ and the spectral stability of periodic wave is observed numerically  for $c\in(\frac{1}{2}, c^*)$.\\
\indent In addition to the spectral stability, we prove the orbital stability for $\alpha\in(1,2]$ by employing an adaptation of the recent arguments in \cite{CNP} and new ingredients. Indeed, in our paper we are not considering, for the orbital stability, the old method which minimizes the energy $E$ in $(\ref{quantconserved})$ with fixed momentum $P$ in $(\ref{momentum})$ and mass $M$ in $(\ref{mass})$. When this kind of approach is considered, we need to prove, besides good spectral properties for the linearized operator $\mathcal{L}$ in $(\ref{operator})$,  the positiveness  of the  Hessian matrix $\frac{d}{dc}P(\phi)$ (see \cite{ANP}, \cite{an-ban-scia1}, \cite{natali}, \cite{CNP}, \cite{grillakis1}, and references therein). As far as we can see, even in the case $\alpha=2$, when explicit solutions are known in terms of the Jacobi elliptic functions, the calculation of the derivative of $P(\phi)$ in terms of $c$ becomes a hard task (see \cite{an-ban-scia1}). Instead of this, we use the conservation law $V(u)=P(u)-E(u)$ as a constrained manifold associated to the augmented Lyapunov functional $G$ in $(\ref{lyapfunct})$. This consideration sheds new light in the stability theory since we obtain our results without using any additional information of the wave, only the basic bound $A(c)>c-\frac{1}{2}$ determined to prove the spectral stability. This fact enables us to conclude item v) in Theorem $\ref{maintheorem}$.\\
\indent Our paper is organized as follows. In Section 2, we establish local and global well posedness results for the Cauchy problem  associated to the equation $(\ref{rbbm})$. In Section 3 we show the existence of periodic minimizers for the problem $(\ref{minp})$. and a precise information about the number of negative eigenvalues of $(\ref{operator})$. Also in this section, we show the uniqueness of periodic minimizers for the problem $(\ref{minp})$ in the interval $\alpha\in\left(\frac{1}{2},2\right]$. In Section 4, we present our spectral stability result for the waves obtained in Section 3. Section 5 is devoted to the numerical investigation of the spectral stability of periodic waves. Finally, in Section 6 we give some important remarks concerning the orbital stability of periodic waves.

\section{Well-posedness results} In this section, we present a brief comment concerning the global well-posedness for the Cauchy problem associated to the equation $(\ref{rbbm})$

\begin{equation}\label{cprbbm}
\left\{\begin{array}{llll}
u_t +u_x +uu_x +(D^{\alpha}u)_t=0,\\\\
u(x,0)=u_0(x).
\end{array}\right.
\end{equation}
In the whole real line, the best result of local well posedness for the Cauchy problem $(\ref{cprbbm})$ has been determined in \cite{LPS}  for initial data  $u_0\in H^s(\mathbb{R})$, $s>\frac{3}{2}-\frac{3}{8}\alpha$
and $\alpha\in (0, 1)$. For  initial data $u_0\in H_{per}^{\alpha/2}(\mathbb{T})$ and $\alpha>1$, we are going to use fixed point arguments applied to the representation of $(\ref{cprbbm})$ in an integral form as

\begin{equation}\label{intform}
u(x,t)=u_0(x)+\int_0^t\left[\mathcal{K}\ast \left(u+\frac{u^2}{2}\right)\right](x,s)ds,
\end{equation}
where $\mathcal{K}$ is defined using the periodic Fourier transform
$$
\widehat{\mathcal{K}v}(k)=\frac{-ik}{1+|k|^{\alpha}}\widehat{v}(k).
$$

\begin{prop}\label{propwp}
	Let $\alpha>1$ be fixed. For each $s\geq\frac{\alpha}{2}$ and $u_0\in H_{per}^{s}(\mathbb{T})$, there exist $T>0$ and a unique solution $u$ of $(\ref{cprbbm})$ such that $u\in C([-T,T];H_{per}^{s}(\mathbb{T}))$. Moreover, for all $T^{\star}<T$, there exists a neighborhood $V$ of $u_0$ in $H_{per}^{s}(\mathbb{T})$ such that the data-solution map
	$$w_0\in V\subset H_{per}^{s}(\mathbb{T})\mapsto w\in C([-T^{\star},T^{\star}];H_{per}^{s}(\mathbb{T})),$$
	is continuous.
\end{prop}

\begin{proof}
	The proof of this result is classical because of the integral representation in $(\ref{intform})$ (for the cases $\alpha=1,2$, see \cite{an-ban-scia1}). Define, $Y=C([-T,T];H_{per}^{s}(\mathbb{T}))$ with the norm $||w||_{Y}=\sup_{t\in[-T,T]}||w(t)||_{H_{per}^{s}}$. Consider the map $\Upsilon: B_{r}\rightarrow B_r$ given by \begin{equation}\label{duh}\Upsilon (u)=\displaystyle u_0(x)+\int_0^t\mathcal{K}\ast \left(u+\frac{u^2}{2}\right)ds,\end{equation}
	where $r>0$ will be chosen later and $B_r=\{w\in Y;\  ||w||_{Y}\leq r\}$. We show that $\Upsilon$ is well defined, in the sense that $\Upsilon(u)\in B_r$ for $u\in B_r$, and $\Upsilon$ is a strict contraction.\\
	\indent In fact, consider $v\in H_{per}^{\frac{\alpha}{2}}(\mathbb{T})$. Since $\alpha>1$, we have
	\begin{equation}\label{est1}\begin{array}{llll}
	||\mathcal{K}\ast v||_{H_{per}^{s}}^2\leq \displaystyle 2\pi\sum_{k\in\mathbb{Z}}(1+k^2)^{s}|\hat{v}(k)|^2=
	||v||_{H_{per}^{s}}^2.
	\end{array}\end{equation}
	In addition, if $v\in L_{per}^2(\mathbb{T})$ we obtain the basic smoothing effect given by
	\begin{equation}\label{est2}\begin{array}{llll}
	||\mathcal{K}\ast v||_{H_{per}^{\frac{\alpha}{2}}}^2\leq \displaystyle 2\pi C_1\sum_{k\in\mathbb{Z}}|\hat{v}(k)|^2=C_1||v||_{L_{per}^{2}}^2,
	\end{array}\end{equation}	
	where $C_1>0$ does not depend on $v$.\\
	\indent On the other hand $\alpha>1$ implies that $\Hpers$ is a Banach algebra for all $s\geq\frac{\alpha}{2}$, and thus
	\begin{equation}\label{est3}
	||u^2||_{\Hns}\leq C_2 ||u||_{\Hns}^2,
	\end{equation}
	where $C_2>0$ does not depend on $u$.
	
	Let $u,v\in Y$ be fixed. Gathering $(\ref{est1})$, $(\ref{est3})$ and the definition of $\Upsilon$, we obtain
	\begin{equation}\label{est5}
	||\Upsilon u||_{Y}\leq ||u_0||_{\Hns}+T\left(||u||_{Y}+C_3||u||_{Y}^2\right),
	\end{equation}
	where $C_3>0$ is positive constant depending on $C_2>0$.\\
	\indent Next,
	\begin{equation}\label{est6}\begin{array}{lllll}
	||\mathcal{K}\ast(u^2-v^2)||_{\Hns}\leq \left(||u||_Y+||v||_{Y}\right)||u-v||_{Y}.
	\end{array}
	\end{equation}
	Therefore, by using a similar procedure as in $(\ref{est5})$ we obtain that
	\begin{equation}\label{est7}
	||\Upsilon u-\Upsilon v||_{Y}\leq T\left(1+||u||_Y+||v||_Y\right)||u-v||_Y.
	\end{equation}
	
	\indent By $(\ref{est5})$, $(\ref{est7})$ and the fact that $u,v\in Y$, we have
	
	\begin{equation}\label{est8}
	||\Upsilon u||_{Y}\leq ||u_0||_{\Hns}+T\left(r+C_3r^2\right)\ \mbox{and}\  ||\Upsilon u-\Upsilon v||_{Y}\leq T\left(1+2r\right)||u-v||_Y.
	\end{equation}
	Using $(\ref{est8})$, we can choose $r=2||u_0||_{\Hns}$ and $T=\frac{1}{2(1+C_4r)}$, where $C_4=\max \{2,C_3\}>0$, to obtain
	
	\begin{equation}\label{est9}
	||\Upsilon u||_{Y}\leq r\ \mbox{and}\  ||\Upsilon u-\Upsilon v||_{Y}\leq \frac{1}{2}||u-v||_Y.
	\end{equation}
	Inequalities in $(\ref{est9})$ give us that $\Upsilon:B_r\rightarrow B_r$ is well defined and a strict contraction. From the Banach Fixed Point Theorem, we obtain the existence of a unique $u\in B_r$ such that $\Upsilon u(t)=u(t)$ for all $t\in [-T,T]$. The uniqueness in the whole space and the continuous dependence are determined by a direct application of standard arguments.
\end{proof}

\begin{prop}\label{conserved1}
	Let $\alpha>1$ be fixed. The quantities $E$, $P$ and $M$ defined in $(\ref{quantconserved})$, $(\ref{momentum})$ and $(\ref{mass})$, respectively  are conservation laws.
\end{prop}

\begin{proof}
The proof of this result is standard and we skip the details.

\end{proof}

\indent Propositions $\ref{propwp}$ and $\ref{conserved1}$ give us the following result.

\begin{proposition}\label{teogwp}
	Let $\alpha>1$ be fixed. For each $u_0\in \Hn(\mathbb{T})$, the Cauchy problem $(\ref{cprbbm})$ is globally well posed in $\Hn(\mathbb{T})$ with $u\in C(\mathbb{R};\Hn(\mathbb{T}))$.	
\end{proposition}
\begin{flushright}
	$\square$
\end{flushright}

\section{Existence and Uniqueness of Minimizers - Spectral Properties.}
\subsection{Existence of Periodic Minimizers} In this subsection, we first give a sufficient condition for the existence of periodic waves for the equation \eqref{ode1} by showing the existence of minimizers associated to a convenient variational problem. After that, we use the minimizers to obtain good spectral properties for the linearized operator $\mathcal{L}$ in $(\ref{operator})$. \\
\indent Before starting, we need an important definition which characterizes the solutions of the equation $(\ref{ode2})$.

\begin{definition}\label{defiSL}
	We say that the periodic traveling wave solutions satisfying the equation $(\ref{ode2})$ has a single-lobe profile $\phi$ if there exist only one maximum and minimum of $\phi$ on $\mathbb{T}$. Without the loss of generality, the maximum of $\phi$ is placed at $x=0$.
	
\end{definition}

For a fixed $\tau>0$, let us consider the set
\begin{equation}\displaystyle	Y_{0} = \left\{ u \in  H^{\frac{\alpha}{2}}_{per}(\mathbb{T}); \int_{-\pi}^{\pi}u^3 =\tau,  \int_{-\pi}^{\pi} u =0 \right\}.\label{Ycond} \end{equation}
Our goal is to find a minimizer of the constrained minimization problem
\begin{equation}\label{infB}
\displaystyle q_c= \inf_{u \in Y_{0}} \mathcal{B}_c(u),
\end{equation}
where
\begin{equation}\label{Bfunctional}
\displaystyle \mathcal{B}_c(u) = \frac{1}{2}\int_{-\pi}^{\pi}c(D^{\frac{\alpha}{2}}u)^2 + (c -1)u^2dx.
\end{equation}
\indent The next lemma establishes the first part of Theorem $\ref{maintheorem}$.
\begin{lemma}\label{minlema}
	Let $\alpha> 1/3$ and $c>\frac{1}{2}$ be fixed. The minimization problem \eqref{infB} has at least one solution, that is, there exists a $\phi \in Y_{0}$ satisfying
	\begin{equation}
	\displaystyle \mathcal{B}_c(\phi)= \inf_{u \in Y_{0}} \mathcal{B}_c(u).
	\end{equation}\label{minB}
	If $\alpha\in (\frac{1}{3},2]$, the periodic minimizer of $(\ref{minB})$ has an even, single-lobe profile in the sense of Definition $\ref{defiSL}$.
\end{lemma}
\begin{proof}
	Let $c>\frac{1}{2}$ be fixed. By Poincar\'e-Wirtinger inequality, we see that
\begin{equation}\label{poinc}\displaystyle 2\mathcal{B}_c(u) =\int_{-\pi}^{\pi}c(D^{\frac{\alpha}{2}}u)^2 + (c -1)u^2dx\geq (2c-1)\int_{-\pi}^{\pi}u^2dx\geq0.
\end{equation}	
Combining $(\ref{poinc})$ with the Garding inequality, we obtain that $\mathcal{B}_c$ is an equivalent norm in $H_{per}^{\frac{\alpha}{2}}(\mathbb{T})$ yielding
	$q_c \geq 0$. Since $\mathcal{B}_c$ is a smooth functional in $H_{per}^{\frac{\alpha}{2}}(\mathbb{T})$, let $\{u_n\}$ be a minimizing sequence for \eqref{infB}, that is, a sequence in $Y_{0}$ satisfying
	$$\displaystyle \mathcal{B}_c(u_n) \rightarrow \inf_{u \in 	Y_{0}}\mathcal{B}_c(u), \ \ \textrm{as} \ n \rightarrow +\infty.$$
	There exist positive constants $M_0$ and $M_1$ depending on $c$ satisfying
	$$M_0 ||u_n||_{H_{per}^{\frac{\alpha}{2}}}^2 \leq \mathcal{B}_c(u_n) \leq M_1 || u_n||_{H_{per}^{\frac{\alpha}{2}}}^2, \ \ \forall n \in \mathbb{N},$$
	that is, $\{u_n\}$ is bounded in $H_{per}^{\frac{\alpha}{2}}(\mathbb{T})$. Thus, there is $\phi \in H_{per}^{\frac{\alpha}{2}}(\mathbb{T})$ such that, up to a subsequence, $$u_n \rightharpoonup \phi \ \textrm{weakly in}  \ H_{per}^{\frac{\alpha}{2}}(\mathbb{T}), \ \textrm{as} \ n \rightarrow +\infty.$$
	
	On other hand, since $\alpha> 1/3$ one sees that the energy space $H_{per}^{\frac{\alpha}{2}}(\mathbb{T})$ is compactly embedded in $L_{per}^3(\mathbb{T})$ (see \cite[Theorem 4.2]{amb}), and
	$u_n \rightarrow \phi \ \textrm{in} \ L_{per}^3(\mathbb{T}), \ \textrm{as} \ n \rightarrow +\infty.$
	In addition, using the fact
	\begin{eqnarray*}
		\displaystyle \left|\int_{-\pi}^{\pi}(u_n^3 -\phi^3)dx \right|  \leq ||u_n - \phi ||_{L_{per}^3} + 3||u_n -\phi ||_{L_{per}^3}||\phi||_{L_{per}^3}|| u_n||_{L_{per}^3},
	\end{eqnarray*}
	one has that $\int_{-\pi}^{\pi}\phi^3 dx = \tau$. A similar argument as above and using the fact that
	$H_{per}^{\frac{\alpha}{2}}(\mathbb{T})$ is compactly embedded in $L_{per}^1(\mathbb{T})$, give us $\int_{-\pi}^{\pi} \phi \ dx = 0.$
	
	Moreover, since the weak lower semi-continuity of $\mathcal{B}_c$, we have
	$\displaystyle \mathcal{B}_c(\phi) \leq  \liminf_{n \rightarrow +\infty} \mathcal{B}_c(u_n) = q_c.$\\
\indent The   symmetric rearrangements $\phi^{\#}\in Y_0$ associated to the solution $\phi$\footnote{We can take $\phi^{\#}$ as a symmetric rearrangement in $Y_0$ because it leaves $\int_{-\pi}^{\pi}u^3dx$ and $\int_{-\pi}^{\pi}udx$ invariant (see \cite[Proposition 2.1]{hur})} leave the $L_{per}^2-$norm of $\phi^{\#}$ and $\phi$ invariants and $\int_{-\pi}^{\pi}(D^{\alpha/2}\phi^{\#})^2dx$ does not increase in comparison with $\int_{-\pi}^{\pi}(D^{\alpha/2}\phi)^2dx$ thanks to the  fractional Polya--Szeg\"{o} inequality (for further details and similar applications, see \cite[Lemma A.1]{CJ2019}, \cite[Proposition 2.1]{hur} and \cite[Theorem 2.1]{NPL}). Therefore, for $c>\frac{1}{2}$ we have $\mathcal{B}_c(\phi^{\#})=q_c$. Invoking the original notation $\phi$ instead of $\phi^{\#}$ to simplify the comprehension of the reader, we see that the minimizer $\phi \in Y_0$ of $\mathcal{B}_c(u)$ must decrease away symmetrically from
the maximum point. Using the translational invariance, the maximum point can be placed at $x = 0$,
which yields an even single-lobe profile for $\phi$.
\end{proof}

From Lemma \eqref{minlema} and Lagranges's Multiplier Theorem, there exists $C_1$ and $C_2$ such that
\begin{equation}\label{lagrange}
\displaystyle cD^{\alpha}\phi +(c -1)\phi = C_1\phi^2 +C_2.
\end{equation}

We see that $\phi$ is a nontrivial single-lobe because $\phi\in Y_0$. Since $\int_{-\pi}^{\pi}\phi dx =0$, we deduce from \eqref{lagrange} that $C_2=-\frac{C_1}{2\pi}\int_{-\pi}^{\pi}\phi^2dx$. In addition, multiplying equation $(\ref{lagrange})$ by $\phi$ and integrating the result over $[0,L]$, we obtain by the Poincar\'e-Wirtinger inequality, the fact $c>\frac{1}{2}$ and since $\tau>0$ that $C_1>0$. By the homogeneity of $C_1\left(\phi^2-\frac{1}{2\pi}\int_{-\pi}^{\pi}\phi^2dx\right)$, we see that $C_1$ can be chosen as $C_1=\frac{1}{2}$. Indeed, for all $s>0$, we obtain
\begin{equation}\label{minsB}\begin{array}{lllll}
	\mathcal{B}_c(s\phi)&=&\displaystyle s^2\mathcal{B}_c(\phi)\\\\
	&=&\displaystyle\inf\left\{\mathcal{B}_c(su);\ u\in H_{per}^{\frac{\alpha}{2}}(\mathbb{T}),\ \ \int_{-\pi}^{\pi}u^3dx=\tau,\ \ \ \int_{-\pi}^{\pi}udx=0\right\}\\\\
	&=&\displaystyle\inf\left\{\mathcal{B}_c(u);\ u\in H_{per}^{\frac{\alpha}{2}}(\mathbb{T}),\ \ \int_{-\pi}^{\pi}u^3dx=s^3\tau,\ \ \ \int_{-\pi}^{\pi}udx=0\right\}.
\end{array}\end{equation}
Thus, if $\phi$ solves $(\ref{minB})$ we have that $\psi=s\phi$ is a solution of the minimization problem $(\ref{minsB})$. As above and by Lagrange's Multiplier Theorem, there exists $C_3>0$ and $C_4=-\frac{C_3}{2\pi}\int_{-\pi}^{\pi}\psi^2dx$ such that
\begin{equation}\label{lagrange1}
	s\left(\displaystyle cD^{\alpha}\phi +(c -1)\phi\right) = s^2C_3\left(\phi^2 -\frac{1}{2\pi}\int_{-\pi}^{\pi}\phi^2dx\right).
\end{equation}
Since $s>0$ is arbitrary and $\phi$ is non-trivial, we can choose $s=\frac{1}{2C_3}$ to obtain that $\phi$ solves the equation
\begin{equation}\label{ode23}
	cD^{\alpha}\phi +(c -1)\phi - \frac{1}{2}\phi^2 +A =0,
\end{equation}
so that $C_1$ can be chosen as $C_1=\frac{1}{2}$. Moreover, by a standard bootstrap argument (see \cite{CNP}), we obtain that $\phi$ is smooth.\\

\begin{remark}
	Let $\alpha\in \left(\frac{1}{3},2\right]$ and $L>0$ be fixed. The periodic wave $\phi$ obtained in Lemma $\ref{minlema}$ can be considered with a general period $L$ instead of the normalized period $L=2\pi$. In the general case, the wave speed needs to satisfy (by using Poincar\'e-Wirtinger inequality for general periods) the basic bound $c>\frac{1}{1+\left(\frac{2\pi}{L}\right)^{\alpha}}$ for the existence of periodic waves in the space $X_{0,L}=\{f\in L_{per}^2(\mathbb{T}_L);\ \int_{-L/2}^{L/2}fdx=0\}$, where $\mathbb{T}_{L}=[-L/2,L/2]$. All results in this section, Sections 4 and 6 can be obtained with a general period by replacing the normalized bifurcation point $c=\frac{1}{2}$ by  $c=\frac{1}{1+\left(\frac{2\pi}{L}\right)^{\alpha}}$ and slight modifications in the arguments when necessary. In Section 5, where we present our numerical experiments, it makes necessary to fix the period in order to obtain the plots which give us, for example, the behaviour of $\phi$, $d$, $A'(c)$ and etc. However, the restriction to consider normalized periodic solutions does not affect the generality of the proposed results in a general context.
\end{remark}
\indent Next, let us denote $\mathcal{\tilde{L}}:=\frac{1}{c}\mathcal{L}=D^{\alpha}+ (1 - \frac{1}{c}) - \frac{1}{c}\phi$
the linearized and self-adjoint operator around the periodic wave $\phi$ and suppose that $\ker(\mathcal{\tilde{L}}|_{X_0})=[\phi']$. Thus, a direct application of the implicit function theorem as in \cite[Lemma 3.8]{NPL} gives us that $c\in (\frac{1}{2},+\infty)\mapsto\phi\in H_{per}^{\infty}(\mathbb{T})$ defines a smooth curve of periodic waves all of them with the same period $2\pi$. Thus, one has that $\frac{d}{dc}\phi\in D(\tilde{\mathcal{L}})=H_{per}^{\alpha}(\mathbb{T})$,
\begin{equation}\label{eq15}
\displaystyle \mathcal{\tilde{L}}\left(\frac{d}{dc}\phi\right) = \frac{1}{c^2}\left(A(c)- \phi - \frac{1}{2}\phi^2 - cA'(c) \right), \ \ \ \mathcal{\tilde{L}}\left(\frac{1}{c}\right)= \frac{1}{c^2} \left(c-1- \phi \right)
\end{equation}
and
\begin{equation}\label{eq16}
\displaystyle \mathcal{\tilde{L}}\left( \frac{1}{c}\phi\right) = -\frac{1}{c^2}\left(\frac{1}{2}\phi^2+ A(c)\right).
\end{equation}
Equations in \eqref{eq15} and \eqref{eq16} give us an important relation
\begin{equation}\label{vd1}
c^2\mathcal{\tilde{L}}\left(\frac{d}{dc} \phi - \frac{1}{c}-\frac{1}{c}\phi\right) =\left(1 +2A(c) -c -cA'(c)\right)=d.
\end{equation}

Since $\mathcal{\tilde{L}}\phi= -\frac{1}{c}\left(\frac{1}{2}\phi^2 + A(c) \right)$, we have $\langle \mathcal{\tilde{L}}\phi, \phi \rangle <0$ and we deduce that $\mathcal{\tilde{L}}$ has at least one negative eigenvalue. Next result establishes items i) and ii) of Theorem $\ref{maintheorem}$ by giving the precise behaviour of the first eigenvalues associated with the linearized operator $\mathcal{\tilde{L}}$.

\begin{prop}\label{propL}
	Let $\alpha\in\left(\frac{1}{3},2\right]$ be fixed and assume that $\ker(\mathcal{\tilde{L}}|_{X_0})=[\phi']$. For $c>\frac{1}{2}$ consider $d=1 +2A(c) -c -cA'(c)$. The linearized operator $\mathcal{\tilde{L}}= D^{\alpha}+ (1 - \frac{1}{c}) - \frac{1}{c}\phi$ around the periodic solution $\phi$ satisfies
	\begin{equation}\label{propr1}
	\displaystyle n(\mathcal{\tilde{L}}) = \left\{ \begin{array}{ll} 1, &
	\textrm{if} \  d < 0,\\
	2, & \textrm{if} \ d\geq0
	\end{array}\right.
	\end{equation}
	and
		\begin{equation}\label{propr2}
	\displaystyle z(\mathcal{\tilde{L}}) = \left\{ \begin{array}{ll} 1, &
	\textrm{if} \  d \neq 0,\\
	2, & \textrm{if} \ d=0.
	\end{array}\right.
	\end{equation}
\end{prop}
\begin{proof}
	In view of \eqref{infB}, since $\phi$ locally minimizes $\mathcal{B}_c$, we have
	$
	\mathcal{\tilde{L}}|_{\{1, \phi^2\}^\bot} \geq 0,
	$
	that is, $n(\mathcal{\tilde{L}}|_{\{1, \phi^2\}^\bot}) = 0$ and $n(\mathcal{\tilde{L}})\in\{1,2\}$, since $\langle\mathcal{\tilde{L}}\phi,\phi\rangle=-2\mathcal{B}_c(\phi)<0$.
	
	For $\lambda\notin \sigma(\mathcal{\tilde{L}})$, let $R(\lambda)$ be the following matrix given by
	\begin{equation} R(\lambda) :=
		\left[
		\begin{array}{c c }
			\langle (\mathcal{\tilde{L}}-\lambda)^{-1}1, 1 \rangle &  \langle (\mathcal{\tilde{L}}-\lambda)^{-1}\phi^2, 1 \rangle \\
			& \\
			\langle (\mathcal{\tilde{L}}-\lambda)^{-1}\phi^2, 1 \rangle & \langle (\mathcal{\tilde{L}}-\lambda)^{-1}\phi^2, \phi^2 \rangle
		\end{array}\right].
	\end{equation}
For $d\neq0$, we compute $R(\lambda)$ at $\lambda=0$. In fact,

\begin{equation}\label{op1}
\langle \mathcal{\tilde{L}}^{-1}\phi^2, \phi^2 \rangle	=	\displaystyle - 4c^2 \int_{-\pi}^{\pi} (D^{\frac{\alpha}{2}}\phi)^2 dx -8\pi A(c)\left( \frac{A(c)}{c}d^{-1} + 2(c^2-c) \right)
\end{equation}

	\begin{equation}\label{op2}
	\langle \mathcal{\tilde{L}}^{-1}1, \phi^2 \rangle=\langle \mathcal{\tilde{L}}^{-1}\phi^2, 1 \rangle=\frac{4\pi A(c)}{c}d^{-1},
	\end{equation}
and 	\begin{equation}\label{third-term1}\langle\mathcal{\tilde{L}}^{-1}1, 1 \rangle=\displaystyle
	\displaystyle -\frac{2 \pi }{c}d^{-1}.
	\end{equation}
In $(\ref{op1})$, $(\ref{op2})$ and $(\ref{third-term1})$, $\mathcal{\tilde{L}}^{-1}\phi^2$ and $\mathcal{\tilde{L}}^{-1}1$ are calculated by combining \eqref{eq15} and \eqref{eq16} to obtain
	\begin{equation}\label{eqmatrix}
	\mathcal{\tilde{L}} \left(d^{-1} \left( \frac{d}{dc}\phi - \frac{1}{c} - \frac{1}{c}\phi \right) \right) = 1,
	\end{equation}
	and
	\begin{equation}\label{eqmatrix2}
	\mathcal{\tilde{L}} \left(-2A(c) \left( d^{-1}\left( \frac{d}{dc}\phi -\frac{1}{c} -\frac{1}{c}\phi\right) \right) - 2c\phi \right) = \phi^2.
	\end{equation}
	
	The determinant of $R(0)$  is now given by
	\begin{equation}\label{detM}
	\displaystyle \det R(0) = 8\pi d^{-1} \left( c \int_{-\pi}^{\pi} (D^{\frac{\alpha}{2}}\phi)^2 dx  + (c-1)\int_{-\pi}^{\pi}\phi^2dx\right)=16\pi d^{-1}\mathcal{B}_c(\phi).
	\end{equation}
	Since $c>\frac{1}{2}$, we obtain that $\mathcal{B}_c(\phi)>0$ and the sign of $R(0)$ is obtained only by $d^{-1}$.\\
	\indent On the other hand, by \cite[Theorem 4.1]{pel-book} we have the following identities
	\begin{equation}\label{indexformula}
	n(\mathcal{{\tilde{L}}}|_{\{1,\phi^2\}\bot}) = n(\mathcal{\tilde{L}}) - n_0 - z_0,
	\end{equation}
	and
	\begin{equation}\label{indexformula1}
	z(\mathcal{{\tilde{L}}}|_{\{1,\phi^2\}\bot}) = z(\mathcal{\tilde{L}}) +z_0-z_{\infty}
	\end{equation}
	where $n_0$ is the number of negative eigenvalues of $R(0)$, $z_0$ denotes the dimension of the kernel of $R(0)$ and $z_{\infty}$ corresponds the number of diverging eigenvalues, that is, $z_{\infty}=1$ if $d=0$, and $z_{\infty}=0$ otherwise. We then deduce the result \eqref{indexformula} from the equation \eqref{propr1}. \\
	\indent Finally, since $\phi$ is a single-lobe solution, we see that $\phi'\in \ker(\tilde{\mathcal{L}})$ has only two zeroes in the interval $[-\pi,\pi)$. This means, from the Oscillation Theorem in \cite{hur}, that $z(\tilde{\mathcal{L}})\in \{1,2\}$. If $d\neq 0$, we deduce from the second equality in $(\ref{eq15})$, $(\ref{eqmatrix})$ and $(\ref{eqmatrix2})$ that $\{1,\phi,\phi^2\}\subset{\rm range}(\tilde{\mathcal{L}})$, so that $z(\tilde{\mathcal{L}})=1$ by using \cite[Proposition 3.1]{hur}. If $d=0$, we have $z_{\infty}=1$ and since $\phi'\in\{1,\phi^2\}^{\bot}$, one has $z(\mathcal{{\tilde{L}}}|_{\{1,\phi^2\}\bot})\geq1$. Thus, $z(\tilde{\mathcal{L}})=2$ as requested.
\end{proof}

\begin{obs}\label{obs25}
	Let $\alpha\in\left(\frac{1}{3},2\right]$ be fixed. If $\ker(\mathcal{L}|_{X_0})=[\phi']$, by \cite[Lemma 3.8]{NPL} we obtain the existence of a smooth curve $\omega\in I\mapsto \psi_{\omega}\in H_{per}^{\infty}(\mathbb{T})\cap X_{0,e}$ of periodic waves for the equation $(\ref{ode2})$, where $I\subset(\frac{1}{2},+\infty)$ is a convenient open interval containing $c$. Important to mention that we can not assure if $\psi_{\omega}$ solves the minimization problem $(\ref{minB})$, neither that it is a single-lobe, except at $\omega=c$, where $\psi_{\omega}=\phi$.
\end{obs}

\subsection{Uniqueness of periodic minimizers}

We give sufficient conditions to determine a result of uniqueness for the periodic minimizer $\phi$ obtained in Lemma $\ref{minlema}$. To do so, we will follow similar arguments  in \cite[Section 5]{FL}. \\
\indent Let us consider $\mathcal{L}|_{X_0}=\mathcal{L}+\frac{1}{2\pi}\langle\phi,\cdot\rangle$. Throughout this section we assume that
\begin{equation}\label{hyp-uniq}
z(\mathcal{L}|_{X_0})=1,
\end{equation}
for every $c>\frac{1}{2}$ and $\alpha\in\left(\frac{1}{2},2\right]$. Using that $\phi\in X_0$ we obtain $\langle\mathcal{L}|_{X_0}\phi,\phi\rangle=\langle\mathcal{L}\phi,\phi\rangle=-\frac{1}{2}\int_{-\pi}^{\pi}\phi^3dx=-2\mathcal{B}_c(\phi)<0$, so that $n(\mathcal{L}|_{X_0})\geq1$. By \cite[Corollary 4.5]{NPL}, we can obtain automatically the condition $(\ref{hyp-uniq})$ by assuming the additional hypothesis $n(\mathcal{L})=1$
for every $c>\frac{1}{2}$ and $\alpha\in\left(\frac{1}{2},2\right]$. Equality in $(\ref{hyp-uniq})$ enables us to conclude an important property for $\mathcal{L}|_{X_0}$ as
\begin{equation}\label{hyp-uniq2}
\ker(\mathcal{L} |_{X_0}) =[ \phi'].
\end{equation}
\indent To simplify the notation, we define the real Banach space
\begin{equation}\label{banach-uniq}
V=\{f\in L_{per}^3(\mathbb{T})\cap X_0;\  f\ \mbox{is even and real-valued}\}
\end{equation}
whose norm, since $L_{per}^3(\mathbb{T})\hookrightarrow L_{per}^2(\mathbb{T})$, is given by $||f||_{V}:=||f||_{L_{per}^3(\mathbb{T})}.$ For a fixed $\alpha_0\in (\frac{1}{2},2]$, our first goal is to construct a local branch of solutions $(\phi_{\alpha}, c_{\alpha})\in V\times \left(\frac{1}{2},+\infty\right)$ of $(\ref{ode2})$ parameterized by the index $\alpha$ in some interval $[\alpha_0,\alpha_0+\delta)$, for some $\delta>0$ small enough.\\
\begin{prop}\label{prop1-uniq}
	Let $\alpha_0\in(\frac{1}{2}, 2]$. Suppose that
	$(\phi_{0}, c_{0})\in V\times \left(\frac{1}{2},+\infty\right)$ with $\phi_0$ being a nonzero solution of $(\ref{ode2})$ and
	satisfying assumption $(\ref{hyp-uniq})$  with $\alpha=\alpha_0$ and $c=c_{0}$. Then, for some $\delta>0$, there exists a $C^1-$map
	$\alpha\in I\rightarrow (\phi_{\alpha}, c_{\alpha})\in V\times \left(\frac{1}{2},+\infty\right)$, defined in the interval $I=[\alpha_0, \alpha_0+\delta)$, such that the following
	holds:\\
	\begin{itemize}
		\item[(i)] $(\phi_{\alpha}, c_{\alpha})$ solves equation $(\ref{ode2})$ with $c=c_{\alpha}$, for all $\alpha\in I$ and the pair $(\phi_{\alpha}, c_{\alpha})$ satisfies $(\ref{hyp-uniq})$.
		\item[(ii)] There exists $\varepsilon>0$ such that $(\phi_{\alpha}, c_{\alpha})$ is the unique solution of $(\ref{ode2})$ for $\alpha\in I$ in
		the neighborhood $\{(\phi,c)\in V\times \left(\frac{1}{2},+\infty\right);\ ||\phi-\phi_{0}||_V+|c-c_{0}|<\varepsilon\}$.
		\item[(iii)] For all $\alpha\in I$, we have
		$\int_{-\pi}^{\pi}\phi_{\alpha}^3dx=\int_{-\pi}^{\pi}\phi_{0}^3dx.$
	\end{itemize}
\end{prop}

\begin{proof}
	The proof of this result relies on the implicit function theorem. Since $c>\frac{1}{2}$, we obtain that $(\ref{ode2})$ can be written as
	$\phi=\left(cD^{\alpha}+\left(c-1\right)\right)^{-1}\frac{1}{2}\Pi_0\phi^2.$
	For some $\delta>0$ to be chosen later, let us define the mapping
	\begin{equation}\label{map-ift-uniq1}
	\mathcal{F}: V\times \left(\frac{1}{2},+\infty\right)\times I\rightarrow V\times \mathbb{R},
	\end{equation}
	given by
	\begin{equation}\label{map-ift-uniq2}
	\mathcal{F}(\phi,c,\alpha)=\left[\begin{array}{ccccc} \displaystyle\phi-\left(cD^{\alpha}+\left(c-1\right)\right)^{-1}\frac{1}{2}\Pi_0\phi^2\\\\
\displaystyle	\int_{-\pi}^{\pi}\phi^3dx-\int_{-\pi}^{\pi}\phi_{0}^3dx\end{array}\right].
	\end{equation}

	We see that $\mathcal{F}$ is a well defined $C^1$ map and $\mathcal{F}(\phi_{0},c_{0},\alpha_0)=0$. Our intention is to show the invertibility of the Fr\'echet derivative
	of $\mathcal{F}$ with respect $(\phi, c)$ at $(\phi_{0}, c_{0},\alpha_0)$. In fact, we see that

	$$
	D_{\phi,c}\mathcal{F}(\phi_0,c_0,\alpha_0)=\left[\begin{array}{ccccc} 1-(c_0D^{\alpha_0}+(c_0-1))^{-1}\Pi_0\phi_0 & & \displaystyle (c_0D^{\alpha_0}+(c_0-1))^{-2}(D^{\alpha}+1)\frac{1}{2}\Pi_0\phi_0^2\\\\
\displaystyle	3\langle\phi_0^2,\cdot\rangle & & 0\end{array}\right].
	$$
	
	We claim that $D_{\phi,c}\mathcal{F}$ is invertible at $(\phi_0,c_0,\alpha_0)$. To do so, for every $g\in V$ and $\lambda\in \mathbb{R}$ given, we need to show the existence of a unique pair $(f,\beta)\in V\times\mathbb{R}$ such that
	
	\begin{equation}\label{fr-der-uniq2}
\left[\begin{array}{ccccc} \displaystyle f-(c_0D^{\alpha_0}+(c_0-1))^{-1}\Pi_0(\phi_0 f) + \frac{\beta}{2} (c_0D^{\alpha_0}+(c_0-1))^{-2}(D^{\alpha}+1)\Pi_0\phi_0^2\\\\
\displaystyle	3\langle\phi_0^2,f\rangle \end{array}\right]=\left[\begin{array}{llll}g\\\\
	\lambda\end{array}\right].
	\end{equation}
	In order to simplify the notation, we define $\mathcal{P}=-(c_0D^{\alpha_0}+(c_0-1))^{-1}\Pi_0\phi_0$ and $h_0=\frac{1}{2}(c_0D^{\alpha_0}+(c_0-1))^{-2}(D^{\alpha}+1)\Pi_0\phi_0^2$. By $(\ref{fr-der-uniq2})$ we see that
	\begin{equation}\label{fr-der-uniq3}
	(1+\mathcal{P})f+\beta h_0=g,
	\end{equation}
	
	\begin{equation}\label{fr-der-uniq4}
	3\int_{-\pi}^{\pi}\phi_0^2 fdx=\lambda.
	\end{equation}
	\indent Now, $\mathcal{P}$ is a compact operator on $X_{0,e}=\{f\in X_0;\  f\ \mbox{is even}\}$ and from $(\ref{hyp-uniq})$ we obtain that $-1$ is not an element of the spectrum of $\mathcal{P}$. Thus, $1+\mathcal{P}$ is an invertible operator on $X_{0,e}$ and since $\mathcal{P}: V\rightarrow V$ (this fact follows similarly from \cite[Lemma E.1]{FL}), we obtain that $(1+\mathcal{P})^{-1}$ exists on $V$. Thus, we can express $f$ uniquely as
	\begin{equation}\label{f-uniq1}
	f=-\beta(1+\mathcal{P})^{-1}h_0+(1+\mathcal{P})^{-1}g.
	\end{equation}
	Gathering results in $(\ref{fr-der-uniq4})$ and $(\ref{f-uniq1})$, we deduce
	\begin{equation}\label{inp-uniq1}
	 3\beta\int_{-\pi}^{\pi}\phi_0^2(1+\mathcal{P})^{-1}h_0dx=-\lambda+3\int_{-\pi}^{\pi}\phi_0^2(1+\mathcal{P})^{-1}gdx.
	\end{equation}
	\indent By $(\ref{hyp-uniq2})$, we see that $(\mathcal{L}|_{X_0})^{-1}$ exists on $X_{0,e}$ and therefore, on $H_{per}^{\alpha_0}(\mathbb{T})\cap X_{0,e}$
	\begin{equation}\label{eq-inv-uniq}
	(1+\mathcal{P})^{-1}=(\mathcal{L}|_{X_0})^{-1}(c_0D^{\alpha_0}+(c_0-1))
	\end{equation}
	is well defined. To show that $\beta$ can be explicitly expressed in terms of $\phi_0$, $\mathcal{P}$ and $g$, we need to prove that $\langle \phi_0^2,(1+\mathcal{P})^{-1}h_0\rangle\neq0$. Indeed, using $(\ref{eq-inv-uniq})$, the fact that $\mathcal{L}|_{X_0}\phi_0=-\frac{1}{2}\Pi_0\phi_0^2$ and since $h_0\in V$, we see from $(\ref{inp-uniq1})$ that
\begin{eqnarray*}
	\langle \phi_0^2, (1+\mathcal{P})^{-1}h_0\rangle  &=& \langle A(c_0), (1+\mathcal{P})^{-1}h_0 \rangle- 2 \langle \mathcal{L}|_{X_0}\phi_0, (1+\mathcal{P})^{-1}h_0 \rangle \\ &=& -2 \langle \mathcal{L}|_{X_0}\phi_0, (\mathcal{L}|_{X_0})^{-1}(c_0D^{\alpha_0}+(c_0-1))h_0 \rangle \\ &=& - \langle (c_0D^{\alpha_0}+(c_0-1))\phi_0,  (c_0D^{\alpha_0}+(c_0-1))^{-2}(D^{\alpha}+1)\Pi_0\phi_0^2 \rangle \\ &=& - \frac{1}{2} \langle \Pi_0 \phi_0^2, (c_0D^{\alpha_0}+(c_0-1))^{-2}(D^{\alpha}+1)\Pi_0\phi_0^2 \rangle \\ &=& - \pi \sum_{k \neq 0}\frac{|k|^{\alpha_0} +1}{ (c_0|k|^{\alpha_0}+(c_0-1))^2} | \widehat{\Pi_0 \phi_0^2}(k)|^2 \neq 0,
\end{eqnarray*}
where $\widehat{f}$ indicates the periodic Fourier transform for a function $f\in L_{per}^2(\mathbb{T})$. The rest of the proof can be done by a direct application of the implicit function theorem.
\end{proof}

The next step is to follow similar arguments as found in \cite[Subsection 5.2]{FL}. In fact, we consider the corresponding maximal extension of the branch $(\phi_{\alpha}, c_{\alpha})$ for
$\alpha\in [\alpha_0, \alpha_{*})$, where $\alpha_{*}$ is given by
$$\begin{array}{rrrr}\alpha_{*}:=\sup\{q;\ \alpha_0<q<2,\  (\phi_{\alpha},c_{\alpha})\in C^1([\alpha_0, q);V\times \left(\frac{1}{2},+\infty\right))\ \mbox{given by Proposition \ref{prop1-uniq}}\\\\
\mbox{and}\ (\phi_{\alpha}, c_{\alpha})\ \mbox{satisfies (\ref{ode2}) for}\ \alpha\in [\alpha_0, q)\}.\end{array}$$
It is clear that $\alpha_{*}\leq2$ and we prove that $\alpha_{*}=2$.

\begin{prop}\label{prop-seq-ext}
	Let $\{\alpha_n\}_{n=1}^{n=+\infty}\subset (\frac{1}{2}, \alpha_{*})$ be a sequence such that $\alpha_n\rightarrow \alpha_{*}$. Furthermore,
	we assume that $\phi_{\alpha_n}\in X_{0,e}$ are the corresponding solutions obtained in Proposition $\ref{prop1-uniq}$ with wave speed $c_{\alpha_{n}}$. Up to a subsequence, it follows that
	$$\phi_{\alpha_n}\rightarrow \phi_{*}\ \mbox{in}\ V\cap L_{per}^4(\mathbb{T})\ \ \mbox{and}\ \ c_{\alpha_n}\rightarrow c_{*}$$
	where $c_{*}\in \left(\frac{1}{2},+\infty\right)$ and $\phi_{*}$ satisfy $(\ref{ode2})$. Moreover, the corresponding maximal branch $(\phi_{\alpha}, c_{\alpha})\in C^1([\alpha_0, \alpha_{*});V\times\left(\frac{1}{2},+\infty\right))$ extends to $\alpha_{*}=2$.
\end{prop}
\begin{proof}
	First, let us suppose that $\{c_{\alpha_n}\}_{n=1}^{+\infty}\subset \left(\frac{1}{2},+\infty\right)$ is not a bounded sequence. For all $M>0$ large enough, there exists an index $m:=n_{M}\in\mathbb{N}$ such that $c_{\alpha_{m}}>M$. Thus, from $(\ref{ode2})$ and the fact that $\int_{-\pi}^{\pi}\phi_{\alpha_n}^3dx=\int_{-\pi}^{\pi}\phi_0^3>0$ for all $n\in\mathbb{N}$, we obtain
	$$
	 0\leq\int_{-\pi}^{\pi}(D^{\alpha_m/2}\phi_{\alpha_m})^2dx\leq\frac{1}{2M}\int_{-\pi}^{\pi}\phi_{\alpha_{n}}^3dx\leq\int_{-\pi}^{\pi}\phi_{0}^3dx.
$$
	\indent Since $\alpha_m\geq \alpha_0$, the Sobolev embedding $H_{per}^{\alpha_m/2}(\mathbb{T})\cap X_{0,e}\hookrightarrow H_{per}^{\alpha_0/2}(\mathbb{T})\cap X_{0,e}$ for all $\alpha_0>\frac{1}{2}$ is valid. Now,
	$f\mapsto ||D^{s}f||_{L_{per}^2(\mathbb{T})}$ is an equivalent norm on the space $H_{per}^{s}(\mathbb{T})\cap X_{0,e}$ for all $s>0$, and thus $\{\phi_{\alpha_m}\}_n$ is a bounded sequence on $H_{per}^{\alpha_0/2}(\mathbb{T})\cap X_{0,e}$. Since $H_{per}^{\alpha_0/2}(\mathbb{T})\cap X_{0,e}$ is compactly embedded into $V\hookrightarrow X_{0,e}$ and into $L_{per}^4(\mathbb{T})$, there exists $\phi_{0}\in V\cap L_{per}^4(\mathbb{T})$ such that, up to a subsequence $\int_{-\pi}^{\pi}\phi_{\alpha_m}^3dx\rightarrow\int_{-\pi}^{\pi}\phi_{0}^3dx>0$, $\int_{-\pi}^{\pi}\phi_{\alpha_m}^4dx\rightarrow\int_{-\pi}^{\pi}\phi_{0}^4dx>0$ and $\int_{-\pi}^{\pi}\phi_{\alpha_m}^2dx\rightarrow \int_{-\pi}^{\pi}\phi_{0}^2dx>0$, as $m\rightarrow+\infty$. By $(\ref{ode2})$ and the Poincar\'e-Wirtinger inequality, we deduce
	\begin{equation}\label{est-bound-seq-uniq1}
	0\leq(2M-1)\int_{-\pi}^{\pi}\phi_{\alpha_m}^2dx\leq (2c_{\alpha_m}-1)\int_{-\pi}^{\pi}\phi_{\alpha_m}^2dx\leq\frac{1}{2} \int_{-\pi}^{\pi}\phi_0^3dx.
	\end{equation}
	When $m\rightarrow+\infty$, we obtain from $(\ref{est-bound-seq-uniq1})$ a contradiction, so that $\{c_{\alpha_n}\}_{n=1}^{+\infty}$ is a bounded sequence on $\mathbb{R}$. Therefore, there exists a $c_{*}\in \left[\frac{1}{2},+\infty\right)$ such that, up to a subsequence
	$
	c_{\alpha_n}\rightarrow c_{*}.
$ On the other hand, the fact that $c_{\alpha_n}\rightarrow c_{*}\geq\frac{1}{2}$ enables us to obtain,
	$$0\leq (2c_{\alpha_n}-1)||\phi_{\alpha_n}||_{L_{per}^2}^2\leq 2\mathcal{B}_{c_{\alpha_{n}}}(\phi_{\alpha_{n}})=\frac{1}{2}\int_{-\pi}^{\pi}\phi_{\alpha_{n}}^3dx=\frac{1}{2}\int_{-\pi}^{\pi}\phi_{0}^3dx,$$
using again the Poincar\'e-Wirtinger inequality. Similarly as determined above, we guarantee the existence of $\phi_{*}\in V\cap L_{per}^4(\mathbb{T})$ such that
	\begin{equation}\label{conv1}
	\int_{-\pi}^{\pi}\phi_{\alpha_n}^3dx\rightarrow\int_{-\pi}^{\pi}\phi_{*}^3dx,\ \ \  \int_{-\pi}^{\pi}\phi_{\alpha_n}^2dx\rightarrow\int_{-\pi}^{\pi}\phi_{*}^2dx,\ \ \ \mbox{and}\ \ \
	 \int_{-\pi}^{\pi}\phi_{\alpha_n}^4dx\rightarrow \int_{-\pi}^{\pi}\phi_{*}^4dx.
	\end{equation}
	
	\indent Next, from the equation $(\ref{ode2})$, we obtain
	$$
	 D^{\alpha_n}\phi_{\alpha_n}=-\left(1-\frac{1}{c_{\alpha_{n}}}\right)\phi_{\alpha_{n}}+\frac{1}{2c_{\alpha_n}}\phi_{\alpha_{n}}^2-\frac{A(c_{\alpha_{n}})}{c_{\alpha_n}},
	$$
	so that $\{\phi_{\alpha_{n}}\}_n$ is bounded in $H_{per}^{\alpha_{n}}(\mathbb{T})$. Since $\alpha_{n}>\frac{1}{2}$, it follows  that $H_{per}^{\alpha_{n}}(\mathbb{T})$ is a Banach algebra and thus $\{\phi_{\alpha_{n}}^2\}_{n}$ is bounded in $H_{per}^{\alpha_{n}}(\mathbb{T})$. From
		\begin{equation}\label{equalitywave1}
	D^{2\alpha_n}\phi_{\alpha_n}=-\left(1-\frac{1}{c_{\alpha_{n}}}\right)	 D^{\alpha_n}\phi_{\alpha_{n}}+\frac{1}{2c_{\alpha_n}}	D^{\alpha_n}\phi_{\alpha_{n}}^2,
	\end{equation}
	one has that $\{\phi_{\alpha_{n}}\}_n$ is bounded in $H_{per}^{2\alpha_{n}}(\mathbb{T})$. Following with this inductive process, there is $r\in\mathbb{N}$ large enough such that $r\alpha_{n}>\alpha_{*}$ and $\{\phi_{\alpha_{n}}\}_n$ is bounded in $H_{per}^{r\alpha_{n}}(\mathbb{T})$. By the compact embedding $H_{per}^{r\alpha_{n}}(\mathbb{T})\hookrightarrow H_{per}^{\alpha_{*}}(\mathbb{T})$, one has (modulus to a subsequence)
	\begin{equation}\label{equalitywave2}
	\phi_{\alpha_n}\rightarrow\phi_{*}\ \mbox{in}\ H_{per}^{\alpha_*}(\mathbb{T}),\ \mbox{as}\ n\rightarrow+\infty.
	\end{equation}
	Therefore,
	\begin{equation}\label{equalitywave3}\begin{array}{llll}
	||D^{\alpha_n}\phi_{\alpha_n}-D^{\alpha_{*}}\phi_{{*}}||_{L_{per}^2}
	\leq
	||D^{\alpha_*}(\phi_{\alpha_n}-\phi_{{*}})||_{L_{per}^2}+
	||D^{\alpha_n}\phi_{*}-D^{\alpha_{*}}\phi_{{*}}||_{L_{per}^2}
	\end{array}
	\end{equation}
	We obtain from $(\ref{equalitywave2})$, $(\ref{equalitywave3})$ and the fact that $\alpha_{n}\rightarrow\alpha_{*}$,
	
\begin{equation}\label{equalitywave4}
D^{\alpha_n}\phi_{\alpha_n}\rightarrow D^{\alpha_{*}}\phi_{{*}}\ \mbox{in}\ L_{per}^{2}(\mathbb{T}),\ \mbox{as}\ n\rightarrow+\infty.
\end{equation}
	\indent Since $c_{\alpha_n}\rightarrow c_*$, the second and third convergences in $(\ref{conv1})$ give us by $(\ref{equalitywave4})$ that $\phi_{*}$ satisfies the equation
	\begin{equation}\label{ode123}
	c_{*}D^{\alpha_{*}}\phi_{*}+(c_{*}-1)\phi_{*}-\frac{1}{2}\phi_*^2+A(c_*)=0.
	\end{equation}
\indent To prove $\alpha_*=2$, since $\phi_{*}$ satisfies $(\ref{ode123})$, Proposition $\ref{prop1-uniq}$ allows to extend the branch $(\phi_{\alpha}, c_{\alpha})$ beyond $\alpha_{*}$. This fact contradicts the maximality property of $\alpha_{*}$. This fact proves that $\alpha_{*}=2$.\\
	\indent It remains to establish $c_{*}>\frac{1}{2}$. First of all, we obtain the smoothness of $\phi_*$ by a standard bootstrapping argument. If $c_{*}=\frac{1}{2}$, the only possibility for the smooth solution $\phi_{*}\in H_{per}^{\infty}(\mathbb{T})\cap X_{0,e}$ of $(\ref{ode123})$ with $\alpha_{*}=2$ is that $\phi_{*}\equiv0$ (see Remark $\ref{rem-uniq1}$ below). This fact generates a contradiction since $\int_{-\pi}^{\pi}\phi_{*}^3dx>0$.\\
	
\end{proof}

\begin{remark}\label{rem-uniq1}
	
	The explicit solution for the equation $(\ref{ode123})$ with $\alpha_{*}=2$ depends on the Jacobi elliptic function of \textit{dnoidal} type and it is given by
	
	\begin{equation}\label{dnsol}
	\phi_*(x)=a_*\left({\rm dn}^2\left(\frac{K(\kappa)}{\pi}x,\kappa\right)-\frac{E( \kappa)}{K( \kappa)}\right),
	\end{equation}
	where
	$$a_*= 12c_{*}\frac{K(\kappa)^2}{\pi^2}.$$
	The value of $c_{*}$ can be expressed by
	\begin{equation}\label{cvalue}	c_{*}={\frac {{\pi }^{2}}{ \left( 4\,{\kappa}^{2}-8 \right)   {\it
				K}(\kappa)^{2}+12\,{\it K} ( \kappa)
			{\it E} \left( \kappa \right) +{\pi }^{2}}},\end{equation}
	where $K$ and $E$ are, respectively, the complete elliptic integral of first and second kind. Both functions depending on the modulus of the Jacobi elliptic function $\kappa\in(0,1)$.
 For this solution, the  integration constant
 \begin{equation}\label{constant}
 	A(c_*)=\frac{1}{16}\frac{(24\kappa^2-24)K(\kappa)^4+(96-48k^2)E(\kappa)K(\kappa)^3-72K(\kappa)^2E(\kappa)^2}{\left((\kappa^2-2)K(\kappa)^2+3E(\kappa)K(\kappa)+\frac{\pi^2}{4}\right)^2}
 	 \end{equation}
 is obtained by using symbolic computations in Maple. By the explicit expressions given by $(\ref{cvalue})$ and $(\ref{constant})$, respectively, it is easy to see that $c_{*}$ and $A(c_*)$ are strictly monotonic functions in terms of $\kappa\in(0,\kappa_0)$, where $\kappa_0\approx0.994$ is the unique zero of the function $p(\kappa)=\left( 4\,{\kappa}^{2}-8 \right)   {\it K}(\kappa)^{2}+12\,{\it K} ( \kappa){\it E} \left( \kappa \right) +{\pi }^{2}$. Furthermore, we have $$c_*\rightarrow \frac{1}{2}^{+}\  \mbox{and}\ A(c_*)\rightarrow 0^{+},\ \mbox{as}\ \kappa\rightarrow 0^{+}.$$
	\indent From $(\ref{cvalue})$, we see that the limit $c_{*}>\frac{1}{2}$ is independent of the sequence $\{\alpha_n\}_{n=1}^{+\infty}$ presented by the Proposition $\ref{prop-seq-ext}$. Furthermore, since $\phi_{*}$ is unique for each fixed $c_{*}$ and satisfying $c_{*}>\frac{1}{2}$, we conclude that the limit $\phi_{*}$ is also independent of $\{\alpha_n\}_{n=1}^{+\infty}$.
\end{remark}

\indent Collecting all the results enunciated above, we can prove our uniqueness result to establish the precise statement in Theorem $\ref{maintheorem}$-iii).

\begin{proposition}\label{theo-uniq}
	Let $\alpha\in(\frac{1}{2},2]$ be fixed. If $(\ref{hyp-uniq})$ is valid, then the solution obtained in Lemma $\ref{minlema}$  is unique.
\end{proposition}
\begin{proof}
	The proof of this result has the same spirit as determined in \cite[Theorem 2.4]{FL}. Suppose that $\phi_0$ and $\tilde{\phi}_0$ are solutions of the problem $(\ref{minp})$ satisfying $\phi_0\not\equiv\tilde{\phi}_0$. Since both $\phi_0$ and $\tilde{\phi}_0$ are in $Y_0$, we obtain that $\int_{-\pi}^{\pi}\phi_0dx=\int_{-\pi}^{\pi}\tilde{\phi}_0dx=0$ and  \begin{equation}\int_{-\pi}^{\pi}\phi_0^3dx=\int_{-\pi}^{\pi}\tilde{\phi}_0^3dx
		\label{L3uniq}\end{equation} For the case of solitary waves, equality $(\ref{L3uniq})$ is one of the crucial parts in the uniqueness proof  contained in \cite{FL} since they use the classical Pohozaev equality to reach the result. According to our best knowledge, it is well known that we do not have a similar equality in the periodic context, so that equality $(\ref{L3uniq})$ is essential for our purpose. By Proposition \ref{prop-seq-ext}, there exist two smooth branch of solutions $(\phi_{\alpha}, c_{\alpha}), \ (\tilde{\phi}_{\alpha}, \tilde{c}_{\alpha})\in C^1([\alpha_0, 2);V\times\left(\frac{1}{2},+\infty\right))$ satisfying \eqref{ode2} with $c_{\alpha}, \ \tilde{c}_{\alpha} \in \left(\frac{1}{2},+\infty\right)$, $(\phi_{\alpha_0}, c_{\alpha_0}) := (\phi_0, c_0)$ and $(\tilde{\phi}_{\alpha_0}, \tilde{c}_{\alpha_0}) := (\tilde{\phi}_0, c_0)$. Notice that by the local uniqueness obtained from Proposition \ref{prop1-uniq}, the smooth branches $(\phi_{\alpha}, c_{\alpha})$ and $(\tilde{\phi}_{\alpha}, \tilde{c}_{\alpha})$ do not have a common point. Thus, for each branch $(\phi_{\alpha}, c_{\alpha})$ and $(\tilde{\phi}_{\alpha}, \tilde{c}_{\alpha})$, we have by Proposition \ref{prop-seq-ext} the existence of $(\phi_{*},c_{*})$ and $(\tilde{\phi}_{*}, \tilde{c}_{*})$ satisfying the following ordinary differential equations
			\begin{equation}\label{eq-sol}
				-c_{*}\phi''_{*}+(c_{*}-1)\phi_{*}-\frac{1}{2}\phi_{*}^2+A(c_{*})=0
			\end{equation}
			and
			\begin{equation}\label{eq-sol2}	-\tilde{c}_{*}\tilde{\phi}''_{*}+(\tilde{c}_{*}-1)\tilde{\phi}_{*}-\frac{1}{2}\tilde{\phi}_{*}^2+A(\tilde{c}_{*})=0,	
			\end{equation}
			respectively. Moreover, by Proposition \ref{prop1-uniq}-(iii), one has	\begin{equation}\label{relac}\displaystyle	\int_{-\pi}^{\pi}\phi_{\alpha}^3dx = \int_{-\pi}^{\pi}\phi_0^3dx=\int_{-\pi}^{\pi}\tilde{\phi}_0^3dx =\int_{-\pi}^{\pi}\tilde{\phi}_{\alpha}^3dx.\end{equation}
			
			Again, using compactness arguments as done in the proof of  Proposition $\ref{prop-seq-ext}$, it follows that $
			\int_{-\pi}^{\pi}\phi_{\alpha}^3dx \rightarrow \int_{-\pi}^{\pi}\phi_{*}^3dx $ and $\int_{-\pi}^{\pi}\tilde{\phi}_{\alpha}^3dx \rightarrow \int_{-\pi}^{\pi}\tilde{\phi}_{*}^3dx$, as $\alpha\rightarrow 2$. By \eqref{relac} we have
			\begin{equation}\label{igualdadecub}
				\int_{-\pi}^{\pi}\phi_{*}^3dx = \int_{-\pi}^{\pi}\tilde{\phi}_{*}^3dx.
			\end{equation}
			\indent We claim that $c_{*}= \tilde{c}_{*}$. Indeed, for $\alpha_{*} =2$, let us consider $\varphi:=\varphi_{\omega}$ a general zero mean solution for the equation
			$(\ref{eq-sol})$ with wave speed $$\omega:=\omega_{\kappa}={\frac {{\pi }^{2}}{ \left( 4\,{\kappa}^{2}-8 \right)   {\it
						K}(\kappa)^{2}+12\,{\it K} ( \kappa)
					{\it E} \left( \kappa \right) +{\pi }^{2}}},$$ instead of $c_{*}$. Deriving equation \eqref{eq-sol} with respect to $\omega$, we obtain
			\begin{equation}\label{eqsol2}
				\displaystyle
				- \omega\frac{d}{d\omega}\varphi'' -\omega\varphi''+ (\omega-1)\frac{d}{d\omega}\varphi +\varphi-\varphi\frac{d}{d\omega}\varphi+\frac{d}{d\omega}A(\omega)=0.
			\end{equation}
			
			Next, multiplying equation \eqref{eqsol2} by $\varphi$, integrating the result over $[-\pi,\pi]$, we have after an integration by parts that
			\begin{equation}\label{eqsol3}
				\displaystyle
				\frac{\omega}{2}\frac{d}{d\omega}\left(\int_{-\pi}^{\pi}(\varphi')^2  + \varphi^2dx \right) -\frac{1}{2}\frac{d}{d\omega}\int_{-\pi}^{\pi}\varphi^2dx + \int_{-\pi}^{\pi}(\varphi')^2 + \varphi^2dx -\frac{1}{3}\frac{d}{d\omega}\int_{-\pi}^{\pi}\varphi^3dx =0.
			\end{equation}
			
			On the other hand, multiplying \eqref{eq-sol} by $\varphi$ and integrating the final result over  $[-\pi,\pi]$, it follows that
			\begin{equation}\label{eqsol4}
				\displaystyle
				\omega\left(\int_{-\pi}^{\pi}(\varphi')^2  + \varphi^2dx \right) - \int_{-\pi}^{\pi}\varphi^2 dx -  \frac{1}{2}\int_{-\pi}^{\pi}\varphi^3 dx =0.
			\end{equation}
			Deriving expression \eqref{eqsol4} with respect to $\omega$ and combining the final result with $(\ref{eqsol3})$, we obtain
			
			\begin{equation}\label{eqsol6}
				\displaystyle
				\frac{d}{d\omega}\int_{-\pi}^{\pi}\varphi^3dx = 6\left(\int_{-\pi}^{\pi}(\varphi')^2+\varphi^2 dx\right) >0,
			\end{equation}
		so that, $\omega\mapsto\int_{-\pi}^{\pi}\varphi^3dx$ is one-to-one, and consequently $c_*=\tilde{c}_*$.\\
			\indent Finally, by  Remark \ref{rem-uniq1}, we have that $\omega$  given explicitly by
			\begin{equation}\label{cvalue2}	
				\omega={\frac {{\pi }^{2}}{ \left( 4\,{\kappa}^{2}-8 \right)   {\it K}(\kappa)^{2}+12\,{\it K} ( \kappa){\it E} \left( \kappa \right) +{\pi }^{2}}}
			\end{equation}
			is positive and strictly increasing in terms of $\kappa\in(0,\kappa_0)$. We have that $\phi_{*}$ and $\tilde{\phi}_{*}$ are $2\pi-$periodic solutions given explicitly by $(\ref{dnsol})$ with the same wave speed $c_*=\tilde{c}_*$. Since $\omega$ is one-to-one according with $(\ref{cvalue2})$, it follows that the corresponding modulus of $c_*$ and $\tilde{c}_*$ must satisfy $\kappa_*=\tilde{\kappa}_*$, so that $\tilde{\phi}\equiv\phi_*$. The remainder of the proof is similar to \cite[Theorem 2.4]{FL}.

\end{proof}

\begin{remark}
	Suppose that $n(\mathcal{L})=1$ for every $c>\frac{1}{2}$ and $\alpha\in\left(\frac{1}{2},2\right]$. Since $(\ref{hyp-uniq})$ is valid, let $\{\alpha_n\}_{n=1}^{n=+\infty}\subset (\frac{1}{2}, \alpha_{*})$ be a sequence such that $\alpha_n\rightarrow \alpha_{*}$. The additional assumption on the number of negative eigenvalues gives us immediately the following equality:
	$$ n(\mathcal{L}_{\alpha_n})= n(\mathcal{L}_{\alpha_*})=1 ,$$
	where $\mathcal{L}_{\alpha_n}=c_{\alpha_n}D^{\alpha_n}+(c_{\alpha_n}-1)-
	\phi_{\alpha_n}$
	and $\mathcal{L}_{\alpha_*}=c_{*}D^{\alpha_{*}}+(c_{*}-1)-\phi_{*}$
	indicate the linearized operators around the periodic waves $\phi_{\alpha_{n}}$ and $\phi_{\alpha_{*}}$, respectively. It is important to mention that our uniqueness result remains valid if the quantity of negative eigenvalues for $\mathcal{L}_{\alpha_n}$ is not stable in the sense that we may have $2=n(\mathcal{L}_{\alpha_{n_0}})> n(\mathcal{L}_{\alpha_*})=1$ for some $n_0\in\mathbb{N}$. However, our numerical approach in Section 5 attests that $d<0$ for all $c>\frac{1}{2}$, so that the quantity of negative eigenvalues of $\mathcal{L}_{\alpha_n}$ remains stable along $c\in\left(\frac{1}{2},+\infty\right)$ for all $n\in\mathbb{N}$.
\end{remark}
\begin{remark}
	It is worth mentioning that Proposition $\ref{theo-uniq}$ agrees in some sense with the uniqueness result for positive and periodic waves associated to the fractional Korteweg-de Vries equation as in \cite{LP}. The authors employed  Krasnoselskii's fixed-point theorem and three additional facts:  \\
	i) $n(\mathcal{L})=1$.\\
	ii) $\ker(\mathcal{L})=[\phi']$.\\
	iii) $\alpha\in (0.5849,2]$.
	\end{remark}

\begin{remark}\label{rem12}
	Similarly as reported on page 1962 of \cite{NPL}, if $\phi$ is a solution of $(\ref{ode1})$ obtained from Lemma $\ref{minlema}$, we also conjecture that $(\ref{hyp-uniq})$ is always valid for all $c>\frac{1}{2}$. Our conjecture and the results in this section allow to conclude that the uniqueness of minimizers is always expected.
\end{remark}

\subsection{Existence of the small amplitude solutions}
Let $0<\alpha\leq 2$. In this subsection, we determine the existence of small amplitude periodic wave solutions for the boundary-value problem \eqref{ode2} which is   an even profile. We show that this kind of solutions are given by the corresponding expansion of the wave for $c$ near $\frac{1}{2}$.

For obtaining the small amplitude periodic waves of the boundary-value problem \eqref{ode2}, we use the arguments contained in \cite[Chapter 8.4]{buffoni}, where the authors have determined the existence of the referred waves via bifurcation theory and the Lyapunov-Schmidt reduction.

 For $\alpha \in (0, 2]$, let $F: H^{\alpha}_{per}(\mathbb{T})\cap X_{0,e} \times  (0, + \infty) \rightarrow  X_{0,e}$ be the map defined by
	\begin{equation}
	F(g, \omega):= \omega D^{\alpha}g + (\omega -1)g - \frac{1}{2}\Pi_0g^2.
	\end{equation}
We see that $F$ is smooth in all variables. Moreover, $F(g, \omega) =0$ if and only if, $g$ is a solution of the boundary-value problem \eqref{ode2} corresponding to the wave speed $\omega \in (0,+\infty)$. In particular, it is clear that $F(0, \omega)=0$ for all $\omega \in (0, +\infty)$. \\
\indent The Fr\'echet derivative associated to the function $F$ with respect to $g$ is given by
	\begin{equation}\label{frechet}
	D_{g}F(g, \omega): H^{\alpha}_{per}(\mathbb{T})\cap X_{0,e} \rightarrow  X_{0,e}, \ \ f \mapsto (\omega D^{\alpha} + \omega-1 - \Pi_0 g)f.
	\end{equation}

Let $c_0>0$ be fixed and consider the corresponding linearized operator $\mathcal{L}$ around the wave $\phi$ defined as in \eqref{operator}. The derivative \eqref{frechet} at the point $(\phi, c_0)$ becomes the well known self-adjoint projector operator $\mathcal{L}|_{X_0}$. Moreover, at the point $(0, c_0)$,  we have
	\begin{equation}
	D_{g}F(0,c_0)= c_0D^{\alpha} + c_0 -1 .
	\end{equation}

Notice that the nontrivial kernel of $D_{g}F(0,c_0)$ is determined by  functions $h \in  H^{\alpha}_{per}(\mathbb{T}) \cap X_{0,e}$ satisfying
	\begin{equation}
	\widehat{h}(k)((c_0-1)+c_0|k|^{\alpha}) =0, \ \ k \neq 0.
	\end{equation}
Then, $D_{g}F(0, c_0)$ has the one-dimensional kernel if and only  if, $c_0 = (1+ |k|^{\alpha})^{-1}$ for some $k\in \mathbb{Z}$, in which case it is given by
	\begin{equation}
	\textrm{ker} D_{g}F(0,c_0) = [\tilde{\phi}_k(x)],
	\end{equation}
where $\tilde{\phi}_k(x) = \cos(xk)$.\\
\indent The local bifurcation theory in \cite[Chapter 8.4]{buffoni} enables us to guarantee the existence of an open interval $\mathcal{I}$ containing $c_0 = (1+ |k|^{\alpha})^{-1}$, an open ball $B_r \in H^{\alpha}_{per}(\mathbb{T}) \cap X_{0,e}$ of radius $r>0$ centered at $g=0$ and a unique smooth mapping $c\in \mathcal{I} \mapsto \phi (c) \in H^{\alpha}_{per}(\mathbb{T}) \cap X_{0,e}$ such that $F(\phi(c), c) = 0$ for all $c \in \mathcal{I}$ and $\phi(c_0)=0$. In other words, every solution of $F(g, \omega)=0$ has the form $(\phi(c),c)$.

 For each integer $k\geq 1$, the point $(0, \tilde{c}_k)$, where $\tilde{c}_k := (1+ k^{\alpha})^{-1}$, is a bifurcation point. Furthermore, there exists $a_0>0$ and a local bifurcation curve $a \in (0,a_0)\mapsto (\phi_k(a), c_k(a)) $ through $(0,\tilde{c}_k)$ constituted by even $\frac{2\pi}{k}$-periodic solutions of the boundary-value problem \eqref{ode2}. Additionally, one has $c_k(0)=\tilde{c}_k$ and $D_a \phi_k(0)= \tilde{\phi}_k(x)= \cos(xk)$ and all solutions of $F(g, \omega)=0$ in a neighbourhood of $(0, \tilde{c}_k)$ belongs to the above curve depending on $a$.

We apply the method of Lyapunov-Schmidt reduction to obtain the explicit  expansion of the wave $\phi$ for $c$ near $\frac{1}{2}$. In addition, the local bifurcation curve extends to a global smooth curve $c \in ((1+k^{\alpha})^{-1}, +\infty) \mapsto (\phi_k(c), c_k(c))$ of solutions for the equation $F(g, \omega)=0$. For the simple case $k=1$, both results can be summarized in the next proposition.

	\begin{prop}\label{prop-stokes}
	Let  $\alpha \in (0,2]$ be fixed. There exists $a_0>0$ such that for all $a \in (0, a_0)$ there is a unique even local periodic solution $\phi$ for the boundary value problem \eqref{ode2} given by the following expansion:
\begin{equation}\label{exp-sol}
	\phi(x)= a\cos(x) + \frac{a^2}{2(2^{\alpha} -1)} \cos(2x) + O(a^3)
\end{equation}
and
\begin{equation}\label{exp-speed}
	c= \frac{1}{2} + \frac{a^2}{4(2^\alpha +1)(2^\alpha -1)} + O(a^4).
\end{equation}
In addition, constant $A$ is given by
\begin{equation}\label{exp-ci}
	A(c) = \frac{a^2}{4} + O(a^4).
\end{equation} The pair $(\phi, c) \in H^{\alpha}_{per}(\mathbb{T})\cap X_{0,e} \times (\frac{1}{2}, + \infty)$ is global in terms of the parameter $c>\frac{1}{2}$ and it satisfies $(\ref{ode2})$.
	\end{prop}
\begin{proof}

The proof of this result relies on slight modifications of \cite[Theorem 5.4 and Theorem 5.6]{BD}.
\end{proof}

\indent In the next result, we see that $q_c$ in $(\ref{infB})$ is continuous in $c$ and for $c>\frac{1}{2}$ and that $q_c\rightarrow 0^+$ as $c\rightarrow \frac{1}{2}^+$. This fact gives us that the small amplitude periodic waves obtained in the Proposition $\ref{prop-stokes}$ satisfy the minimization problem $(\ref{minB})$ for $c\rightarrow \frac{1}{2}^{+}$.
\begin{prop}
	\label{coro-continuity}
	Let $\phi \in Y_0$ be the solution of the constrained minimization problem obtained in Lemma \ref{minlema} and $q_c = \mathcal{B}_c(\phi)$. Then $q_c$ is continuous in
	$c$ for $c > \frac{1}{2}$ and $q_c \to 0$ as $c \to \frac{1}{2}$.
\end{prop}

\begin{proof}
	This result is similar to \cite[Lemma 2.3]{NPL}.
\end{proof}

\indent Next result establishes good spectral properties for the linearized operator $\mathcal{L}$ in $(\ref{operator})$ around the single-lobe solution $\phi$ obtained in Lemma $\ref{minlema}$ by knowing the spectral information for the same operator around the small amplitude periodic waves obtained by Proposition $\ref{prop-stokes}$.
\begin{prop}\label{propspec}
	Let $\alpha\in\left(\frac{1}{2},2\right]$ be fixed. For every $c >\frac{1}{2}$, we have $n(\mathcal{L})=1$ if and only if $z(\mathcal{L})=1$. Moreover, $\mathcal{L}$ consists in a countable sequence of positive eigenvalues bounded away from zero.
	\end{prop}
\begin{proof}
	
	The proof has the same spirit as in \cite[Lemma 2.3]{LP}. According to the Proposition $\ref{propL}$, the number of negative eigenvalues
	of $\mathcal{L}$ may change in the parameter continuations in $c$ if and only if, the eigenvalues pass through zero eigenvalue. If $n(\mathcal{L})=1$ for every $c >\frac{1}{2}$, we see that $\ker(\mathcal{L}|_{X_0})=[\phi']$ by \cite[Corollary 4.5]{NPL}. If $z(\mathcal{L})=2$ for some $c_1>\frac{1}{2}$ we see that parameter $d$ defined in $\ref{propL}$ satisfies $d=0$ at $c=c_1$. This means  by Proposition $\ref{propL}$ that $n(\mathcal{L})=2$ at $c=c_1$ which is a contradiction. Suppose that $z(\mathcal{L})=1$ for every $c>\frac{1}{2}$. For the small amplitude periodic waves $\phi$ obtained by the Proposition $\ref{prop-stokes}$, we see that $c\mapsto\phi
	\in H_{per}^{\infty}(\mathbb{T})$ is smooth. Now, from \eqref{exp-speed}, we have the expansion
	\begin{equation}\label{relAder}
	d=1 +2A(c) -c -cA'(c)= \frac{1}{2}\frac {-{16}^{\alpha}+14\,{2}^{-2+2\,\alpha}-3}{{4}^{\alpha}-1}+O(a^3).
	\end{equation}
	For $\alpha\in\left(\frac{1}{2},2\right]$, we have $d<0$ and this means by Proposition $\ref{propL}$ that $n(\mathcal{L})=1$ for the small amplitude periodic waves.  By a continuity argument, we obtain that $n(\mathcal{L})=1$ for the linearized operator $\mathcal{L}$ around the single-lobe solution $\phi$ determined by Lemma $\ref{minlema}$ for all $c > \frac{1}{2}$.
\end{proof}

\section{Spectral Stability}
 In this section, we study the spectral stability of \eqref{spectproblem}. First, we see that the even solution  for the boundary-value problem \eqref{ode2} obtained in Lemma \ref{minlema} is smooth by using similar arguments as in \cite[Proposition 2.4]{NPL}.  Moreover $\phi$ has a single-lobe profile, namely, there exists only one maximum (at $x=0$) and minimum of $\phi$ on $\mathbb{T}$. This  fact is in accordance with the approach in \cite{hur} to decide about the non-degeneracy of the $\ker(\mathcal{\tilde{L}})$.\\
 \indent For $\alpha \in (0, 2]$ and since $\phi$ is a single-lobe solution, we see by Oscillation Theorem in \cite{hur} that the operator $\mathcal{\tilde{L}}$ has at most two negative eigenvalues, that is, $n(\mathcal{\tilde{L}}) \in \{1,2\}$. Consequently, there may be at most two eigenfunctions of $\mathcal{\tilde{L}}$ for the zero eigenvalue, that is, $z(\mathcal{\tilde{L}})\in \{1,2\}$. Recall that $\ker(\tilde{\mathcal{L}}|_{X_0})=[\phi']$ implies $c\in\left(\frac{1}{2},+\infty\right)\mapsto \phi$ is smooth and $d$ in $(\ref{vd1})$ can be used to decide about the non-degeneracy of $\ker(\tilde{\mathcal{L}})$ and the exact quantity of $n(\tilde{\mathcal{L}})$ according to the Proposition $\ref{propL}$. Indeed, if $d\neq0$ we conclude by the first equality in \eqref{eq15}, \eqref{eq16}, and \eqref{vd1} that  $\{1, \phi, \phi^2\} \in$ range$(\mathcal{\tilde{\mathcal{L}}})$. Proposition 3.1 in \cite{hur} is now used to conclude that ker($\mathcal{\tilde{L}})=$ $[\phi']$. On the other hand, if $d=0$, we have $\ker(\mathcal{\tilde{L})}= \left[\phi', \frac{d}{dc} \phi - \frac{1}{c}-\frac{1}{c}	 \phi \right]$. In addition, $d<0$ implies $n(\tilde{\mathcal{L}})=1$ and $d>0$ gives us that $n(\tilde{\mathcal{L}})=2$.\\
\indent First, we prove the spectral stability for the small amplitude periodic waves associated to the case $\alpha\in(0,2]$.

	\begin{proposition}\label{propstability}
 Let $\alpha \in (0, 2]$.  The small amplitude periodic waves $\phi  \in H^{\infty}_{per}(\mathbb{T})$ determined by Proposition $\ref{prop-stokes}$ are spectrally stable in the sense of Definition $\ref{defistab1}$.

	\end{proposition}
\begin{proof}
Since $\mathcal{L}$ is a Hessian operator for $G(u)$ in \eqref{operator} and $\mathcal{\tilde{L}}=\frac{1}{c}\mathcal{L}$ for all $c>\frac{1}{2}$, the spectral stability holds if $\mathcal{\tilde{L}}|_{\{1, (D^\alpha +1)\phi\}^{\bot}} \geq0$, that is, if $n(\mathcal{\tilde{L}}|_{\{1, (D^\alpha +1)\phi\}^{\bot}})=0$. On the other hand, the periodic wave $\phi$ is spectrally unstable if  $n(\mathcal{\tilde{L}}|_{\{1, (D^\alpha +1)\phi\}^{\bot}})=1$.

For $\lambda \notin \sigma(\mathcal{\tilde{L}})$, let us define the following symmetric 2-by-2 matrix given by
	\begin{eqnarray*} S(\lambda)&:=& \left[
	\begin{array}{cc}
	\langle (\mathcal{\tilde{L}}- \lambda I)^{-1}1, 1 \rangle &  						 \langle(\mathcal{\tilde{L}}- \lambda I)^{-1}(D^\alpha +1)\phi, 1 \rangle \\
	& \\
	\langle (\mathcal{\tilde{L}}- \lambda I)^{-1}1, (D^\alpha +1)\phi \rangle & \langle (\mathcal{\tilde{L}}- \lambda 			I)^{-1}(D^\alpha +1)\phi, (D^\alpha +1)\phi \rangle
	\end{array}\right].
	\end{eqnarray*}
Using  \eqref{relAder}, for $\alpha\in (0,\tilde{\alpha})\cup(\frac{1}{2},+\infty)$ we obtain $d<0$, whereas for $\tilde{\alpha}<\alpha<\frac{1}{2}$, we have $d>0$. For $\alpha=\tilde{\alpha}\approx 0.2924$ and $\alpha= \frac{1}{2}$, we have the existence of fold points for which $d=0$.\\
\indent Assume $d \neq 0$, we compute at $\lambda=0$. Indeed,
	\begin{eqnarray}\label{first-term}
	\langle \mathcal{\tilde{L}}^{-1}(D^\alpha +1)\phi, (D^\alpha +1)\phi \rangle & = &\displaystyle - \frac{2 \pi 	 d^{-1}}{c}(A'(c))^2 - 2\pi A'(c) -\frac{1}{6}\gamma'(c) ,
	\end{eqnarray}
	
	\begin{equation}\label{second-term}
	\langle\mathcal{\tilde{L}}^{-1}(D^\alpha +1)\phi, 1 \rangle = \langle \mathcal{\tilde{L}}^{-1}1, (D^\alpha +1)\phi \rangle = \displaystyle \frac{ 2\pi d^{-1}}{c}A'(c),
	\end{equation}
and 	
 \begin{equation}\label{third-term}\langle\mathcal{\tilde{L}}^{-1}1, 1 \rangle=\displaystyle
	 \displaystyle -2 \pi \frac{d^{-1}}{c},
	\end{equation}
where $\gamma(c)= \int_{-\pi}^{\pi} \phi ^3 dx $ and $\frac{1}{3}\gamma'(c) = \frac{8\pi}{c}A(c) - \frac{1}{c}\gamma(c)$. The determinant of $S(0)$, for $d \neq 0$, is given by
	\begin{equation}\label{determinant-D}
	\displaystyle \textrm{det}S(0)= \frac{4\pi^2d^{-1}}{c} \left(A'(c)+ \frac{1}{\pi}\mathcal{B}_c(\phi) \right).
	\end{equation}
 By \eqref{deriv-ci} we have $A'(c)>0$ for all $\alpha\in (0,2]$. Since $\mathcal{B}_c(\phi) >0$, we see that the sign of $\det S(0)$ is determined only by $d^{-1}$ since the terms between parentheses in \eqref{determinant-D} are strictly positive.

We denote by $n_0$ and $z_0$ the number of negative and zero eigenvalues of $S(0)$, respectively. If $d=0$, then $S(0)$ is singular, in which case we denote the number of diverging eigenvalues of $D(\lambda)$ as $\lambda \rightarrow 0$ by $z_{\infty}$. Thus, \cite[Theorem 4.1]{pel-book} gives us the following relations:
\begin{eqnarray*}\displaystyle \left\{
	\begin{array}{ccc}
	  n (\mathcal{\tilde{L}}\mid_{\left\{1, (D^{\alpha}+1)\phi\right\}^{\bot}}) & = & n(\mathcal{\tilde{L}}) - n_0 - z_0,
	\\
	 z ( \mathcal{\tilde{L}}\mid_{\left\{1, (D^{\alpha}+1)\phi\right\}^{\bot}}) & = & z(\mathcal{\tilde{L}}) + z_0 - z_{\infty}
	\end{array}\right..
	\end{eqnarray*}

Assume $d\neq0$, so that $z_{\infty} =0$. If $0<\alpha<\tilde{\alpha}$ or $\alpha\in\left(\frac{1}{2},2\right]$, we have $z_0=0$ and det$S(0) <0 $ which implies that $n_0=1$. On the other hand by Corollary $\ref{coro-continuity}$, we see that the waves in $(\ref{exp-sol})$ solve the minimization problem $(\ref{infB})$ and Proposition $\ref{propL}$ gives us $n(\mathcal{L})=1$. Therefore, we have the spectral stability in this case because $ n (\mathcal{\tilde{L}}\mid_{\left\{1, (D^{\alpha}+1)\phi\right\}^{\bot}}) =1-1-0=0$. For the case $\tilde{\alpha}<\alpha<\frac{1}{2}$, we have $n_0=0$ or $n_0=2$. Since the trace of  $S(0)$ is strictly negative, we conclude that $n_0=2$. Hence, we obtain the spectral stability of $\phi$ since $n(\tilde{\mathcal{L}})=2$ and $ n (\mathcal{\tilde{L}}\mid_{\left\{1, (D^{\alpha}+1)\phi\right\}^{\bot}})=2-2-0=0$.   \\
\indent Finally, for the case $\alpha=\tilde{\alpha}$ or $\alpha=\frac{1}{2}$, we obtain from the relations above and Proposition $\ref{propL}$ that $z_{\infty}=1$, $n(\mathcal{\tilde{L}})=2$, so that $n (\mathcal{\tilde{L}}\mid_{\left\{1, (D^{\alpha}+1)\phi\right\}^{\bot}})=2-2=0$.
\end{proof}

\begin{remark}
	The result contained in Proposition $\ref{propstability}$ can be extended for the case $\alpha>2$.
\end{remark}

\indent The next result guarantees sufficient conditions for the spectral stability of the periodic minimizer $\phi$ and it gives the proof of Theorem $\ref{maintheorem}$-iv).

\begin{proposition}\label{spectmin}
	Let $\alpha\in(\frac{1}{3},2]$ be fixed. If $\ker\left(\mathcal{L}|_{X_0}\right)=[\phi']$, the periodic minimizer $\phi$ obtained in Lemma $\ref{minlema}$ is spectrally stable in the sense of Definition $\ref{defistab1}$ if:\\
	i)  $d\neq0$ and $A'(c)+ \frac{1}{\pi}\mathcal{B}_c(\phi)>0$.\\
	ii) $\alpha\in\left(\frac{1}{2},2\right]$ and $d\neq0$ for all $c>\frac{1}{2}$.
	\end{proposition}
\begin{proof}
	We use the same computations as determined in the proof of Proposition $\ref{propstability}$. To prove i), suppose that $\ker(\mathcal{L}|_{X_0})=[\phi']$, $d>0$ and $A'(c)+ \frac{1}{\pi}\mathcal{B}_c(\phi)>0$. By Proposition $\ref{propL}$, we have $n(\mathcal{L})=2$ and by $(\ref{third-term})$ the first entry of the matrix $S(0)$ in Proposition $\ref{propstability}$ is negative. Thus $n_0=2$ and we have the spectral stability since $n (\mathcal{\tilde{L}}\mid_{\left\{1, (D^{\alpha}+1)\phi\right\}^{\bot}})  =  n(\mathcal{\tilde{L}}) - n_0 - z_0=2-2-0=0$. If $d<0$ and $A'(c)+ \frac{1}{\pi}\mathcal{B}_c(\phi)>0$, we see that $\det S(0)<0$ and we have the spectral stability since $n (\mathcal{\tilde{L}}\mid_{\left\{1, (D^{\alpha}+1)\phi\right\}^{\bot}})  =  n(\mathcal{\tilde{L}}) - n_0 - z_0=1-1-0=0$.  To prove ii), since $d\neq0$ for all $c>\frac{1}{2}$, we obtain by Proposition $\ref{propL}$ that $z(\mathcal{L})=1$, so that $n(\mathcal{L})=1$ by Proposition $\ref{propspec}$. We claim that  $1+2A(c)-c-cA'(c)<0$ for all $c>\frac{1}{2}$. Indeed, if there exists $c_2>\frac{1}{2}$ such that $1+2A(c_2)-c_2-c_2A'(c_2)\geq0$, we deduce by Proposition $\ref{propL}$ that $n(\mathcal{L})=2$ at $c=c_2$. This fact generates a contradiction and the claim is satisfied. Next, we obtain after multiplying this inequality by the factor $\frac{1}{c^3}$ that
	\begin{equation}\label{ineqdiff1}
	\frac{d}{dc}\left(A(c)\frac{1}{c^2}\right)>\frac{1}{c^3}-\frac{1}{c^2}.
	\end{equation}
	\indent Integrating $(\ref{ineqdiff1})$ over the interval $\left(\frac{1}{2},c\right]$ and using $A(c)\rightarrow 0$ when $c\rightarrow\frac{1}{2}$, we obtain the important inequality
	\begin{equation}\label{ineqdiff2}
	A(c)>c-\frac{1}{2}.
	\end{equation}
	Thus, we have from the Poincar\'e-Wirtinger inequality, $(\ref{ineqdiff2})$,  inequality $1+2A(c)-c-cA'(c)<0$, and a simple computation
	
	\begin{equation}\label{ineqdiff3}\begin{array}{lllll}
\displaystyle	A'(c)+\frac{1}{\pi}\mathcal{B}_c(\phi)\geq \displaystyle \left(c-\frac{1}{2}\right)^2>0.
	\end{array}\end{equation}
	 Therefore, if $c>\frac{1}{2}$ one has by $(\ref{ineqdiff3})$ that $A'(c)+\frac{1}{\pi}\mathcal{B}_c(\phi)>0$. From $(\ref{determinant-D})$, we obtain $$\det S(0)=\frac{4\pi^2d^{-1}}{c} \left(A'(c)+ \frac{1}{\pi}\mathcal{B}_c(\phi) \right)<0.$$ This together with $n(\mathcal{L})=1$ give us from the index formula contained in Proposition $\ref{propstability}$ the spectral stability of the single-lobe $\phi$.

\end{proof}

\begin{remark}
	The numerical approach in the next section will attest for a fixed $\alpha\in\left(\frac{1}{2},2\right]$ that $d\neq0$ for all $c>\frac{1}{2}$, so that the periodic waves obtained in Lemma $\ref{minlema}$ are spectrally stable by Proposition $\ref{spectmin}-ii)$. For the case $\alpha\in\left(\frac{1}{3},\frac{1}{2}\right]$ we will have the following situation: there exists a unique $c^{*}$ such that $d>0$ for all $c\in\left(\frac{1}{2},c^{*}\right)$ and $d<0$ for all $c\in(c^*,+\infty)$. Since in both cases we have $A'(c)+\frac{1}{\pi}\mathcal{B}_c(\phi)>0$, we obtain from $(\ref{determinant-D})$ that $\det S(0)>0$ for the case $d>0$ and $\det S(0)<0$ if $d<0$. In both cases, the spectral stability is obtained by Proposition $\ref{spectmin}-i)$. For the case $d=0$, we also have the spectral stability of $\phi$ using a similar procedure as in Proposition $\ref{propstability}$.
\end{remark}

\section{Numerical Experiments}

In this section we propose a Petviashvili's iteration method for the numerical generation of periodic travelling wave solutions of the equation \eqref{rbbm}.
The method is widely used for the generation of travelling wave solutions \cite{aduran,AD,LP,OBM,PS}.  The iteration method for the periodic waves,   a modification of standard Petviashvilli's algorithm \cite{AD, LP},
is based on the following solution steps.
First, we use the transformation
\begin{equation}\label{transformation}
\psi= \displaystyle\frac{1}{2c} \bigg[ \phi-(c-1)+\sqrt{(c-1)^2+2 A} \bigg]
\end{equation}
to convert the equation \eqref{ode1}  into
\begin{equation}\label{ode4}
{D}^{\alpha} \psi +w \psi -\psi^2 =0,
\end{equation}
where $w= \displaystyle\frac{1}{c} \sqrt{(c-1)^2+2 A}$.
Thus, we   obtain the  equation  with a zero integration constant. \mbox{Integrating} \eqref{transformation} on $[-\pi, \pi]$ and using the fact that $\phi: \mathbb{T}\rightarrow \mathbb{R}$ is a periodic function with the zero mean value yield that
\begin{equation}\label{wavespeed}
c=\left( 1-w+\frac{1}{\pi} \int_{-\pi}^{\pi} \psi dx   \right)^{-1}.
\end{equation}

\noindent Next, we use the standard Petviashvilli's method to solve the equation \eqref{ode4}. Employing the Fourier transform to the equation \eqref{ode4} gives
\begin{equation}
\left(  |\xi|^{\alpha} + w  \right) \widehat{\psi}(\xi)-  \widehat{\psi^2}(\xi)=0.
\end{equation}
A simple iterative algorithm for numerical calculation of $\widehat{\psi}(\xi)$ for the above equation can be proposed in the form
\begin{equation}\label{iterative1}
\widehat{\psi}_{n+1}(\xi)=\frac{ \widehat{\psi^2_n}(\xi)  }
{|\xi|^{\alpha} + w},\hspace{30pt} n\in\mathbb{N}
\end{equation}
where $\widehat{\psi}_n(\xi)$ is the Fourier transform of ${\psi}_n$ which is the $n^{th}$ iteration of the numerical solution. Here the solutions are constructed under the assumption
\begin{equation}\label{condition}
 |\xi|^{\alpha} + w \neq 0.
 \end{equation}

\noindent Since the above algorithm is usually divergent, we finally present the Petviashvilli's method as
\begin{equation}\label{iterative2}
\widehat{\psi}_{n+1}(\xi)=\frac{M_n^{\nu}} {|\xi|^{\alpha} + w} \widehat{\psi^2_n}(\xi)
\end{equation}
by introducing the stabilizing factor
\begin{equation}\label{sf}
  M_n=\frac{\langle({D}^{\alpha} +w )\psi_n, \psi_n \rangle}{\langle\psi^2_n, \psi_n\rangle},~~~~~~~~~~ \hspace*{20pt} \psi_n\in H^{\alpha}_{per}(\mathbb{T})
\end{equation}
 where $\langle\cdot ~, \cdot \rangle $ denotes the standard inner product in $L^2_{per} (\mathbb{T})$. Here, the free parameter $\nu$ is chosen as $2$ for the fastest convergence (see \cite{PS} for details).
\noindent
The iterative process is controlled by the error between two consecutive iterations given by
$$
  Error(n)=\|\psi_n-\psi_{n-1}\|_{L_{per}^{\infty}}
$$
and the stabilization factor error given by
$$
|1-M_n|. $$ The residue of the interaction process is determined by
$
{RES(n)}= \|{ \mathcal{S}} \psi_n\|_{L_{per}^{\infty}},
$
where
$$
{ \mathcal{S}}\psi={D}^{\alpha} \psi +w \psi -\psi^2.  $$

\indent The periodic traveling wave solution $\phi$ of the BBM equation  corresponds to boundary value problem \eqref{ode1} with $\alpha=2$  is given in the equation  \eqref{dnsol}.

\begin{figure}[h]
 \begin{minipage}[t]{0.40\linewidth}
   \includegraphics[width=2.9in]{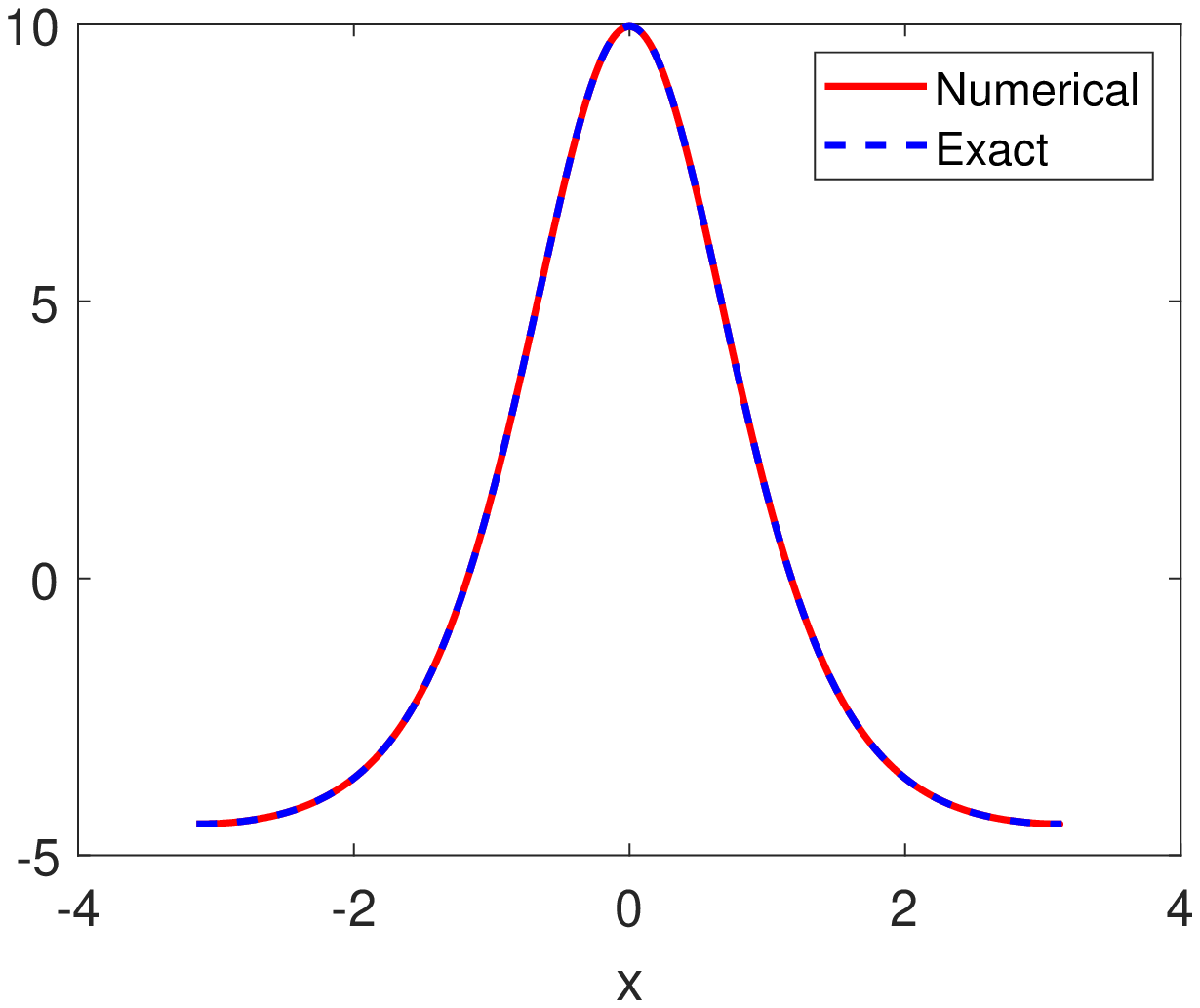}
 \end{minipage}
\hspace{30pt}
\begin{minipage}[t]{0.40\linewidth}
   \includegraphics[width=3in]{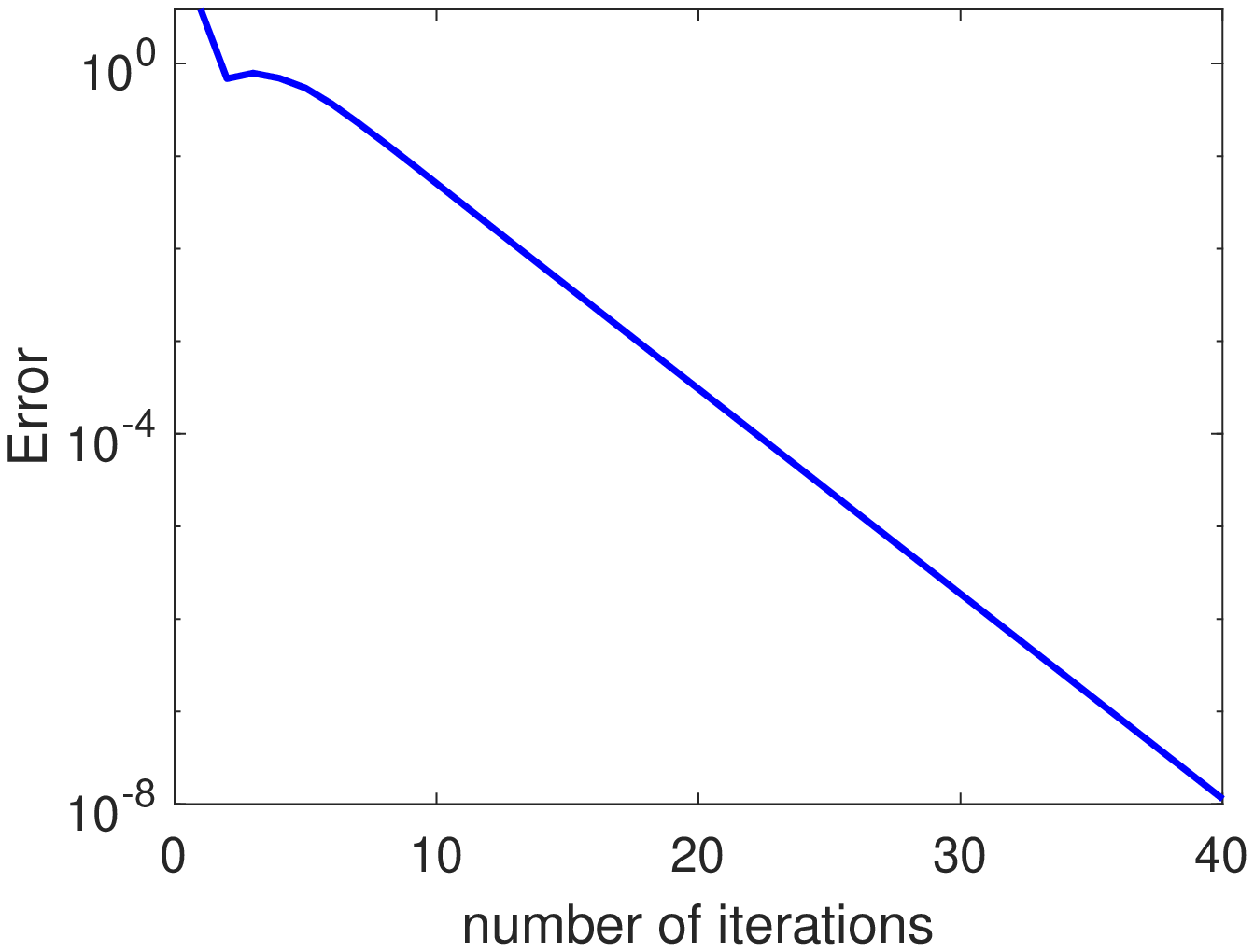}
 \end{minipage}
 \hspace{30pt}
\begin{minipage}[t]{0.40\linewidth}
   \includegraphics[width=3in]{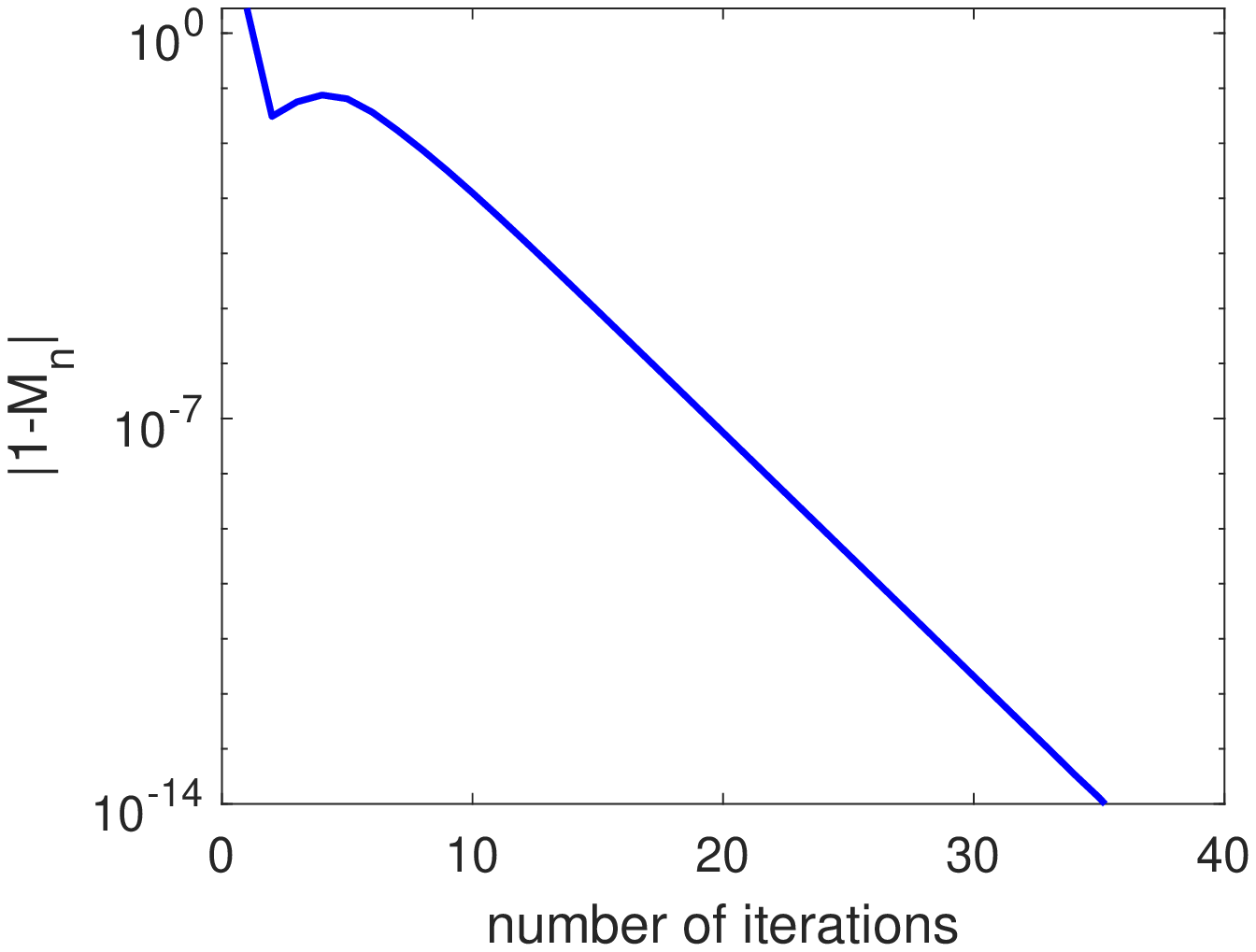}
 \end{minipage}
 \hspace{30pt}
\begin{minipage}[t]{0.40\linewidth}
   \includegraphics[width=3in]{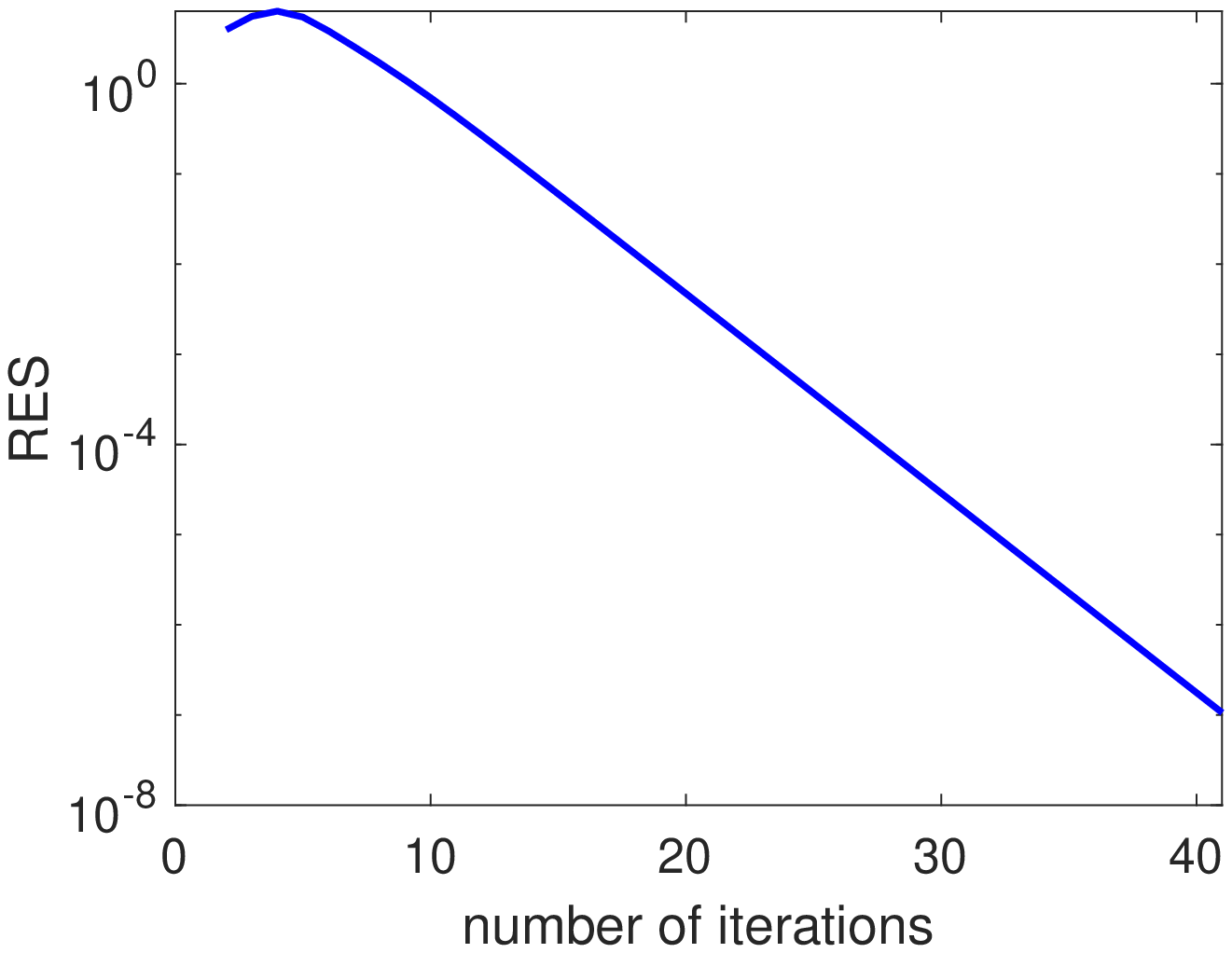}
 \end{minipage}
  \caption{The exact and the numerical solutions of  the BBM equation with the wave speed $c=1.2181$  and  the variation of  \mbox{$Error(n)$}, $|1-M_n|$ and $RES$ with the number of iterations in semi-log scale.}
 \label{dif_er}
\end{figure}

In order to test the accuracy of our scheme, we compare the exact solution \eqref{dnsol} with the numerical solution obtained by using  the expansion corresponding to \eqref{ode4} as the initial guess
\begin{equation} \label{initialguess}
 \psi_0(x)= 1+a\cos(x)+a^2\left[\frac{1}{2}-\frac{1}{2(2^{\alpha}-1)}+\frac{\cos(2x)}{2(2^{\alpha}-1)}\right]+a^3\frac{\cos(3x)}{2(2^{\alpha}-1)(3^{\alpha}-1)},  \end{equation}
where $a=0.2$ and $\alpha=2$.
 In this experiment, the space interval is $ [-\pi,\pi] $ and number of grid points is chosen as $N=2^{12}$. In the first panel of   Figure 5.1, we depict the exact and numerical solutions for the wave speed $c=1.2181$. As it is seen from the figure, the exact and the numerical solutions coincide. In the other panels of Figure 5.1, the variations of three different errors with the number of iteration are presented.  These results show that our numerical scheme captures the solution remarkably well.

 \vspace*{-10pt}
\begin{figure}[h]
\begin{minipage}[t]{0.49\linewidth}
   \includegraphics[width=3in]{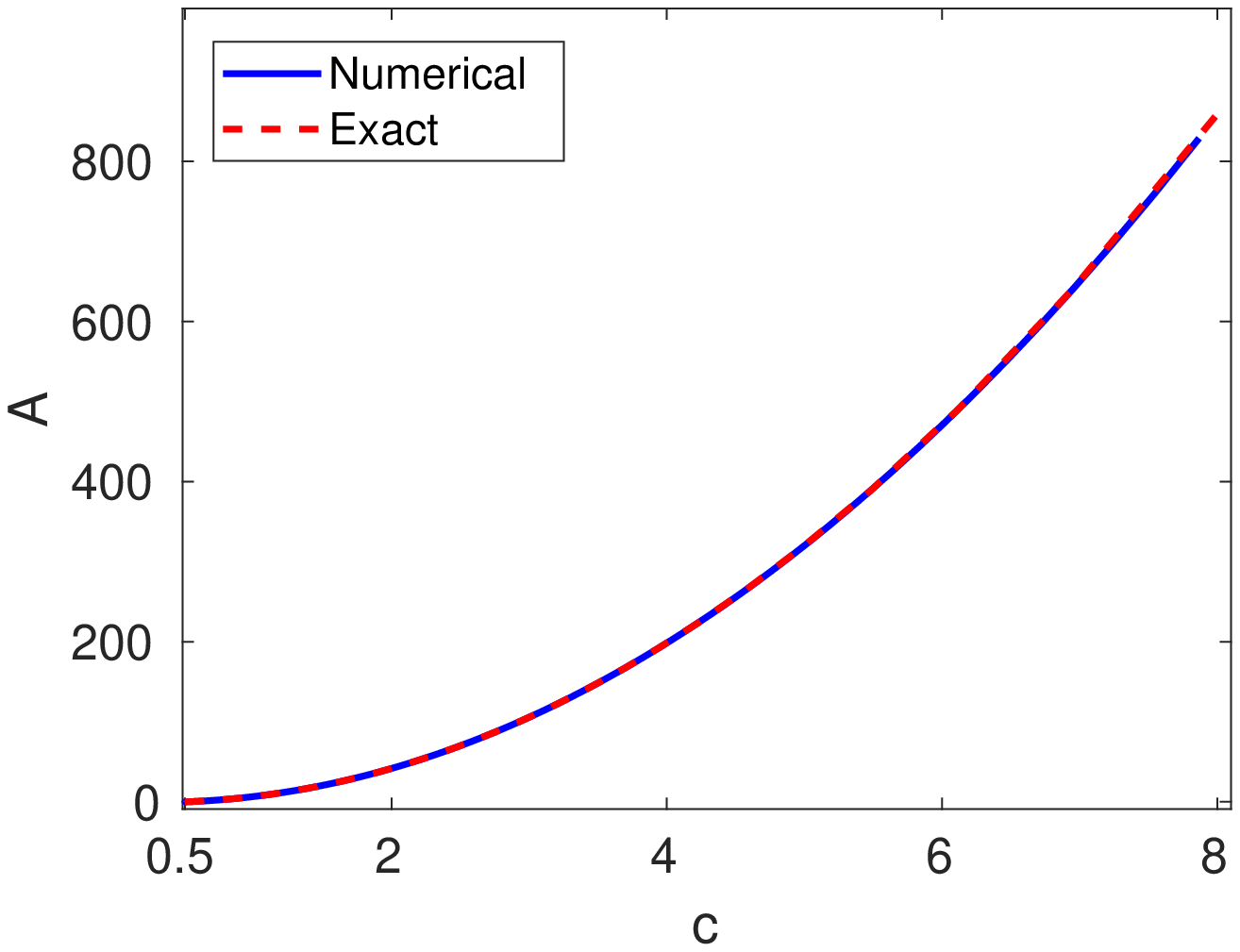}
 \end{minipage}
  \begin{minipage}[t]{0.49\linewidth}
   \includegraphics[width=3in]{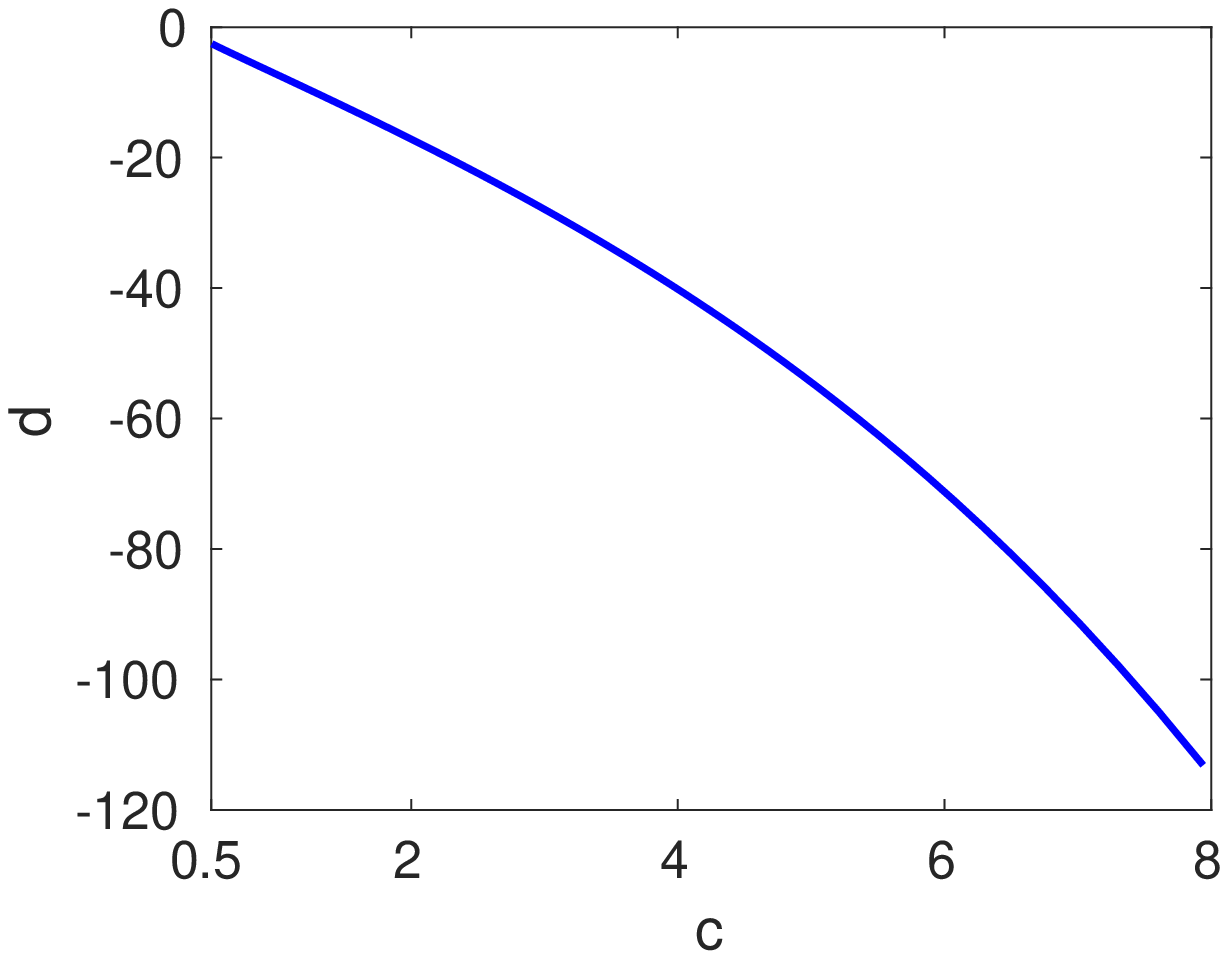}
 \end{minipage}
 \caption{Variation of exact and numerical values of  integration constant $A$ versus $c$ (left panel) and variation of numerical value of $d$ versus $c$ (right panel) for  $\alpha=2$.}
\end{figure}

\indent The left panel of Figure 5.2 shows the exact and numerical variation of the integration constant $A$ with respect to $c$. The exact relation between $A$ and $c$ is illustrated by using the eqs. \eqref{cvalue} and \eqref{constant}. From the numerical point of view, first we obtain the numerical solution by using  Petviashvili's scheme \eqref{iterative2} for every $w\in (1,+\infty)$. Finally,
the numerical value of the integration constant $A(c)$ is evaluated by
\begin{equation}\label{Avalue}
A(c)= \frac{1}{2} \left[ c^2 w^2-(c-1)^2 \right].
\end{equation}
 As it is seen from the figure, the numerical and the exact values of $A(c)$ coincide. In the right panel of the figure, we present the
variation of  $d$  with $c$. We observe that $d$ is always negative.
\begin{figure}[h]
 \centering
 \includegraphics[width=3in]{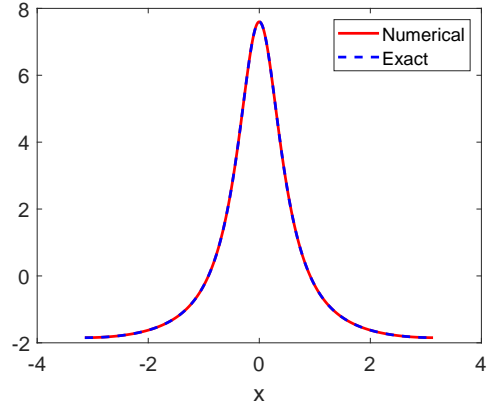}
  \caption{The exact and the numerical solutions of rBO equation for the wave speed $c=1.2192$.}
\end{figure}

\begin{figure}[h]
\begin{minipage}[t]{0.49\linewidth}
   \includegraphics[width=3in]{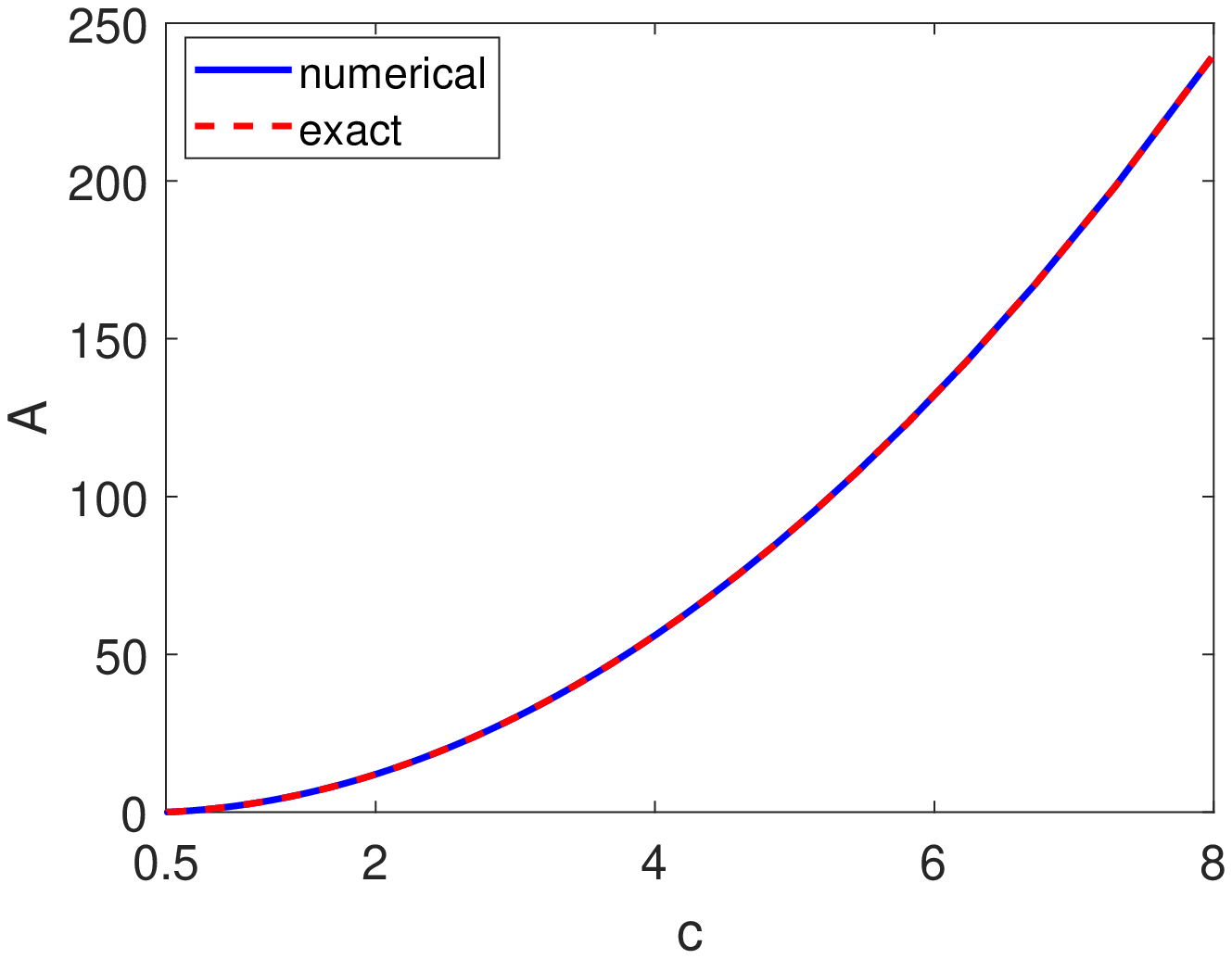}
 \end{minipage}
  \begin{minipage}[t]{0.49\linewidth}
   \includegraphics[width=3in]{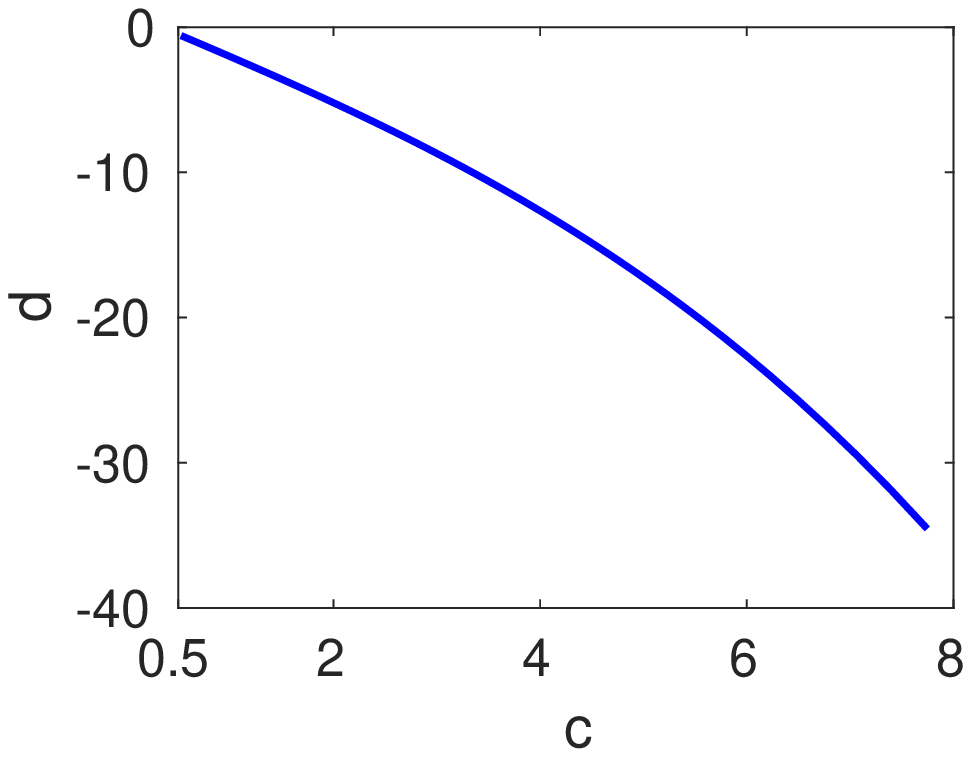}
 \end{minipage}
 \caption{Variation of exact and numerical values of  integration constant $A$ versus $c$ (left panel) and variation of numerical value of $d$ versus $c$ (right panel) for  $\alpha=1$.}
\end{figure}

\noindent
The single-lobe periodic solution to the boundary value problem
\eqref{ode4} for $\alpha=1$ is given by
\begin{equation}\label{rboexact}
\psi (x)= \frac{\sinh \gamma}{ \cosh \gamma -\cos x},
\end{equation}
where  parameter $\gamma\in(0,+\infty)$ is given by $\gamma=\coth^{-1}(w)$.
The transformation \eqref{transformation} allows one to obtain the single-lobe periodic solution for the rBO equation as
\begin{equation} \label{rboexact2}
\phi(x)=2c \bigg( \frac{\sinh \gamma}{ \cosh \gamma -\cos x} -1 \bigg).
\end{equation}
We compare the exact solution \eqref{rboexact2} with the numerical solution obtained by using the expansion \eqref{initialguess} with $a=0.2$ and $\alpha=1$ as the initial guess. The space interval is $ [-\pi,\pi] $ and number of grid points is chosen as $N=2^9$. We present the exact and numerical solutions for the wave speed $c=1.2192$ in Figure 5.3. Since we have $\int_0^{\pi} \psi(x) dx= \pi$ from \eqref{rboexact}, we obtain  $c=(3-w)^{-1}$ by using \eqref{wavespeed}. Eliminating $w$ in \eqref{Avalue}, we compute the integration constant $A(c)=4c^2-2c$ explicitly.  In the left panel of Figure 5.4, we compare the
exact and numerical variation of the integration constant $A$ with respect to $c$
for $\alpha=1$. The right panel shows the variation of  $d$  with $c$ for $\alpha=1$.
 Figures 5.2 and $5.4$  show that $A(c)$ is strictly increasing for all values of $c$ and $d$ is always negative. Therefore, numerical results are compatible with Proposition  \ref{spectmin} stating the spectral stability of single-lobe solution for $\alpha\in (\frac{1}{2},2]$.
\begin{figure}[!htbp]
\begin{minipage}[t]{0.49\linewidth}
   \includegraphics[width=3in]{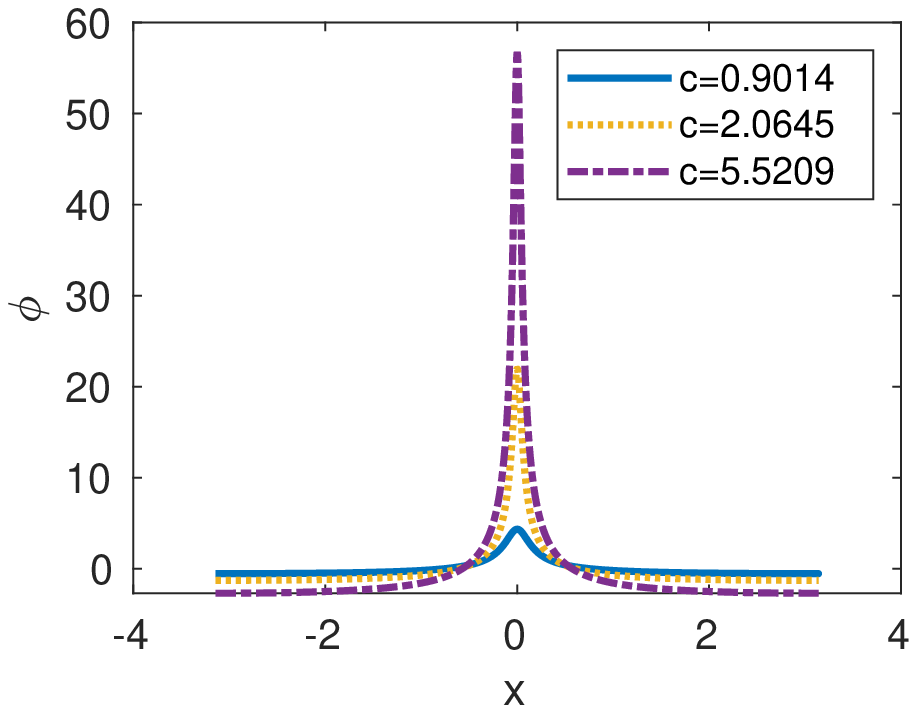}
 \end{minipage}
  \begin{minipage}[t]{0.49\linewidth}
   \includegraphics[width=3in]{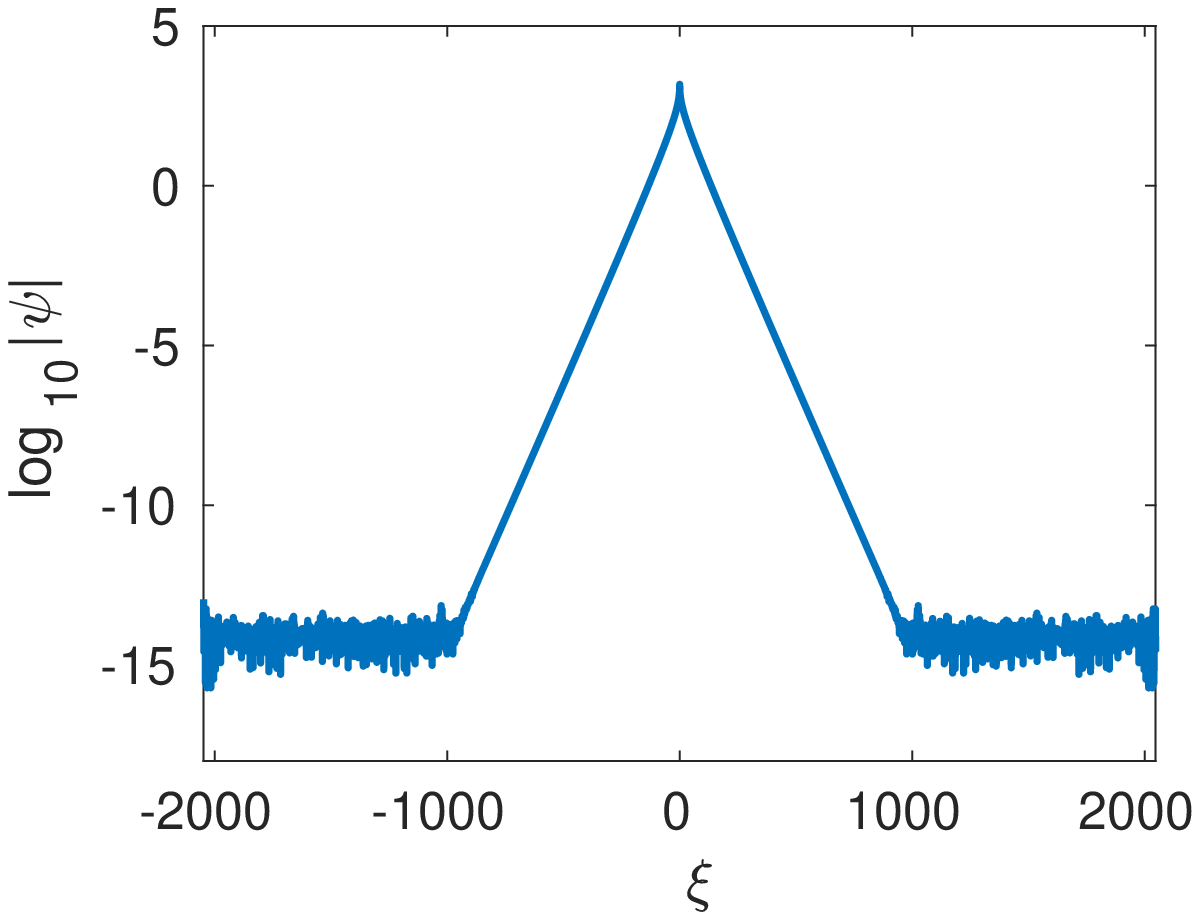}
 \end{minipage}
 \caption{ Various numerical wave profiles  (left panel) and the modulus of the Fourier coefficients of the numerical solution (right panel) for   $\alpha=0.55$.}
\end{figure}

\indent
In the next numerical experiment, we  choose $\alpha=0.55> \alpha_0$   since the fold point $\alpha_0=0.5$ for the fBBM equation.
In the left panel of  Figure 5.5, we illustrate the periodic wave profiles for several  values  of $c$. It can be seen that the amplitude becomes more and more peaked  with the increasing wave speed $c$.  The right panel of Figure 5.5 shows that the modulus of the Fourier coefficients computed via a discrete Fourier transformation  decreases to machine precision for
 $N = 2^{12}$ Fourier modes. {In  Figure 5.6, we depict the
variation of $A(c)$  and $d$  with $c$. The numerical results are again compatible with the Proposition  \ref{spectmin}  for $\alpha > \frac{1}{2}$.
}
\begin{figure}[!htbp]
\begin{minipage}[t]{0.49\linewidth}
   \includegraphics[width=3in]{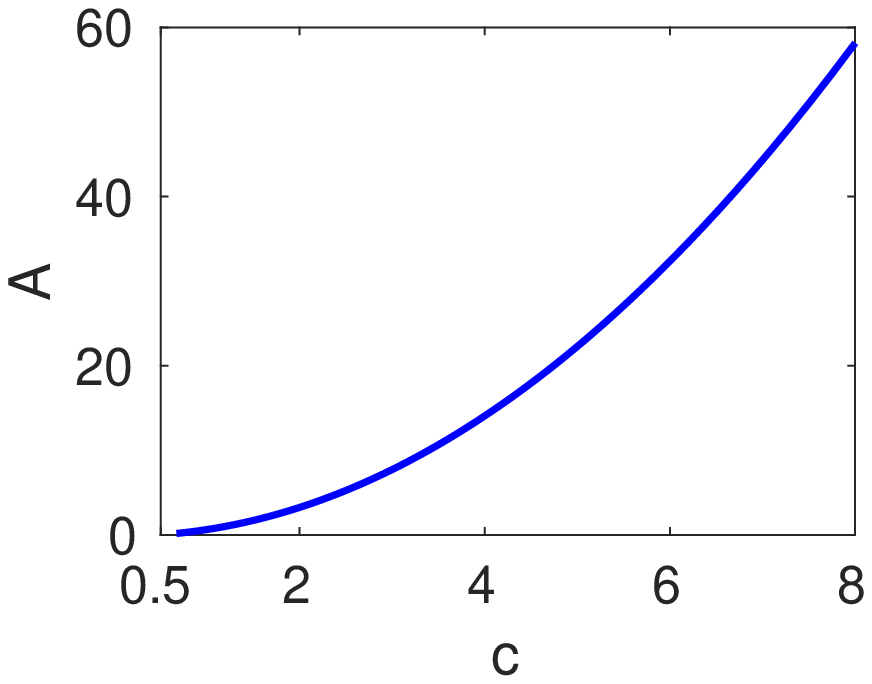}
 \end{minipage}
  \begin{minipage}[t]{0.49\linewidth}
   \includegraphics[width=3in]{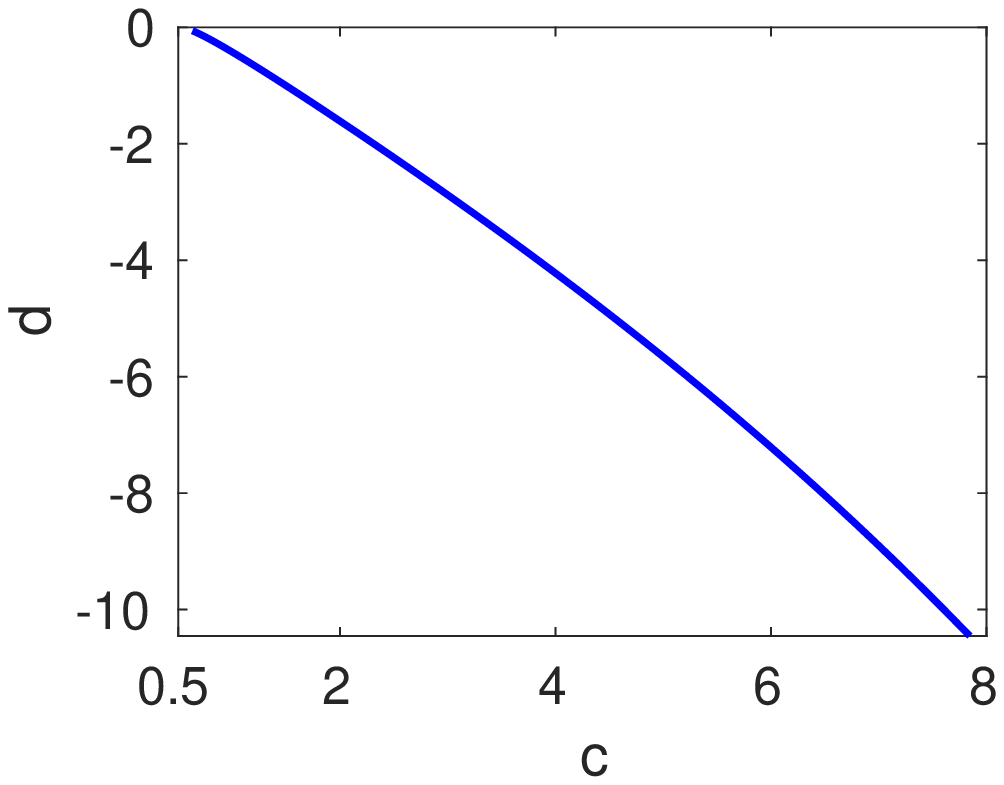}
 \end{minipage}
  \caption{Variation of  $A$ versus $c$ (left panel) and variation of  $d$ versus $c$ (right panel) both evaluated  numerically for  $\alpha=0.55$.}
\end{figure}

\begin{figure}[!htbp]
\begin{minipage}[t]{0.49\linewidth}
   \includegraphics[width=3in]{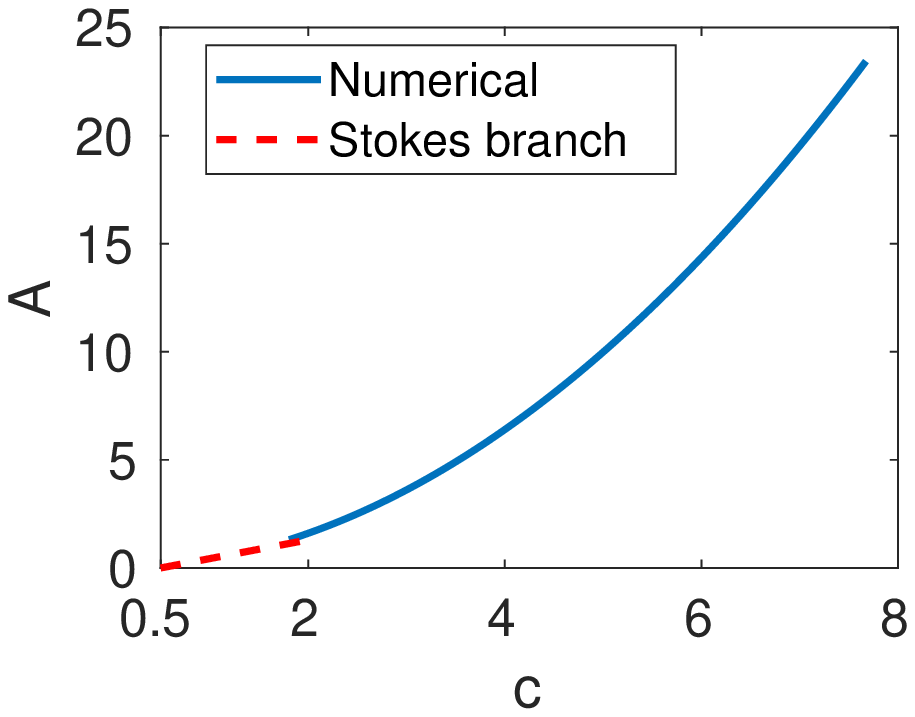}
 \end{minipage}
  \begin{minipage}[t]{0.49\linewidth}
   \includegraphics[width=3in]{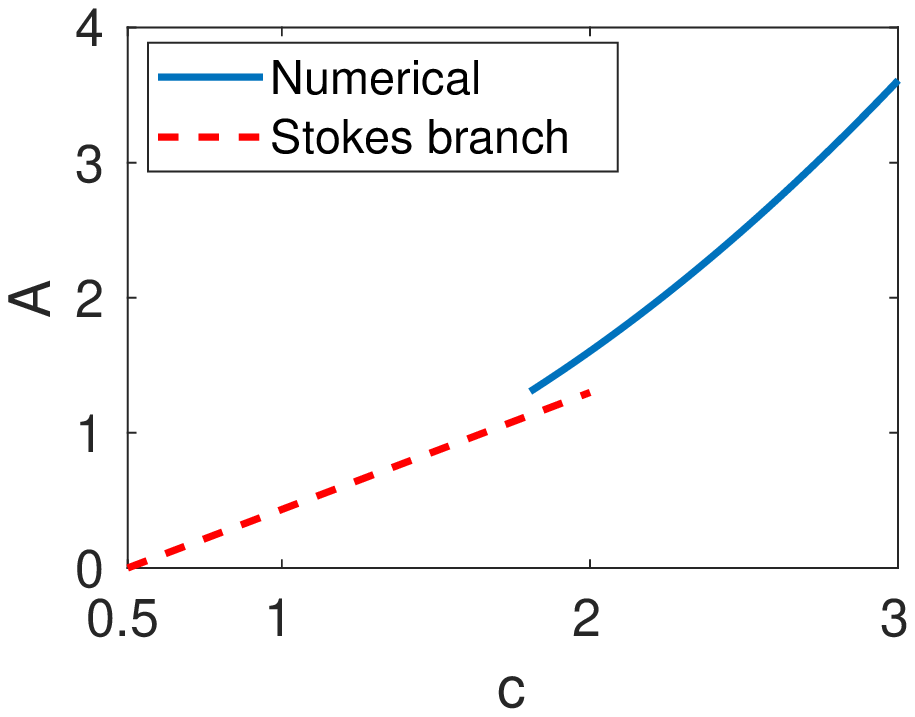}
 \end{minipage}
 \caption{Variation of  values of  $A$ versus $c$ (left panel) and  close-up look to the gap   (right panel)  ($\alpha=0.45$). The solid line shows the numerical result obtained by Petviashvili's method and the dashed line shows the relation \eqref{rel}  from the small amplitude expansion. }
\end{figure}

In Figure 5.7, we show the variation of $A(c)$ with $c$ for $\alpha=0.45 <\alpha_0$. The picture shows a lack of convergence using Petviashvili's scheme in the sense that function $A(c)$ can not reach the bifurcation point $c=\frac{1}{2}$ as determined for the case $\alpha\in\left(\frac{1}{2},2\right]$. To do best of our knowledge, this phenomena gives us (at least numerically) an indication that the number of negative eigenvalues for the linearized operator is $2$ and/or the existence of a fold point, that is, a value of $c$ such that $\dim(\ker(\mathcal{L}))=2$. In both cases, it is well known that the numerical method does not converge  (see \cite{LP, NPL,PS}) and prevent us to show that $\phi$ solves the minimization problem $(\ref{minp})$. Due to the lack of numerical data for $c\in (0.5, 1.8)$, we illustrate the curve (dashed line) by using  the
relation
\begin{equation}
A(c)\approx \left(c-\frac{1}{2}\right) (2^{\alpha}+1) (2^{\alpha}-1),  \label{rel}
\end{equation}
obtained by the expressions of $c$ and $A(c)$ in $(\ref{exp-speed1})$. The right panel gives a closer look to the gap between the branch of the small amplitude periodic waves and numerical result.   In order to fill this gap, we use the Newton's method for $c\in (0.5, 1.8)$ in Figure 5.8.
\begin{figure}
 \centering
 \includegraphics[width=3in]{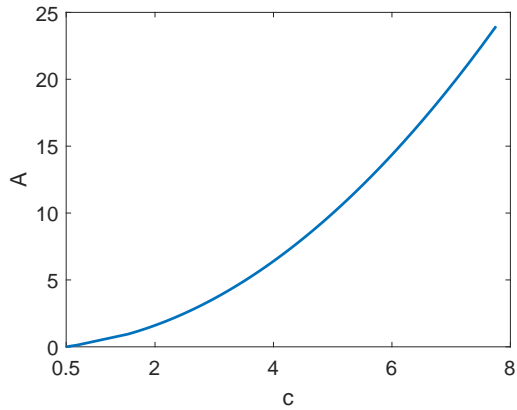}
  \caption{Variation of numerical values of  $A$ versus $c$ with the combination of Newton and Petviashvili method ($\alpha=0.45$)}
\end{figure}

\begin{figure}[]
\begin{minipage}[t]{0.49\linewidth}
   \includegraphics[width=2.9in]{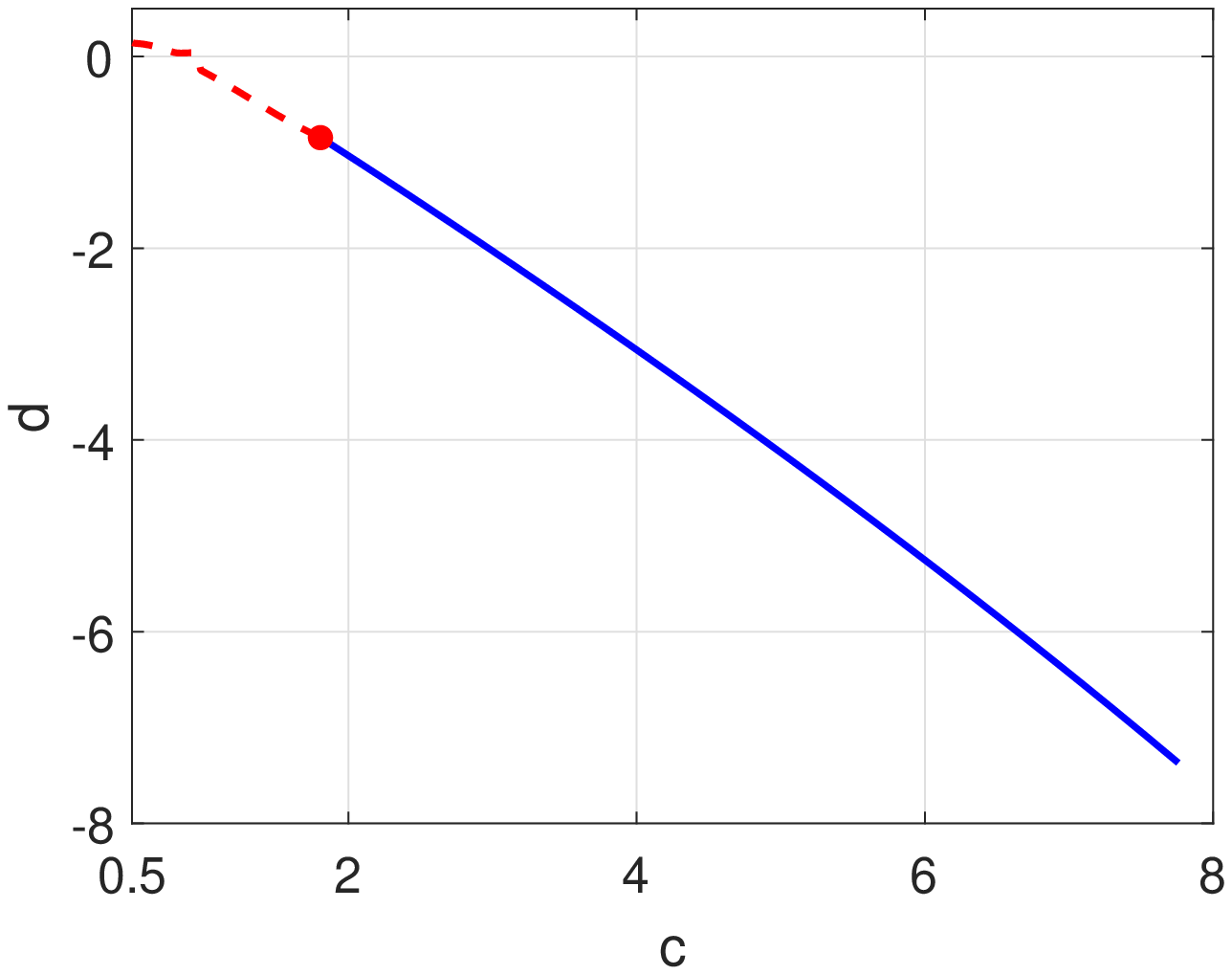}
 \end{minipage}
  \begin{minipage}[t]{0.49\linewidth}
   \includegraphics[width=2.9in]{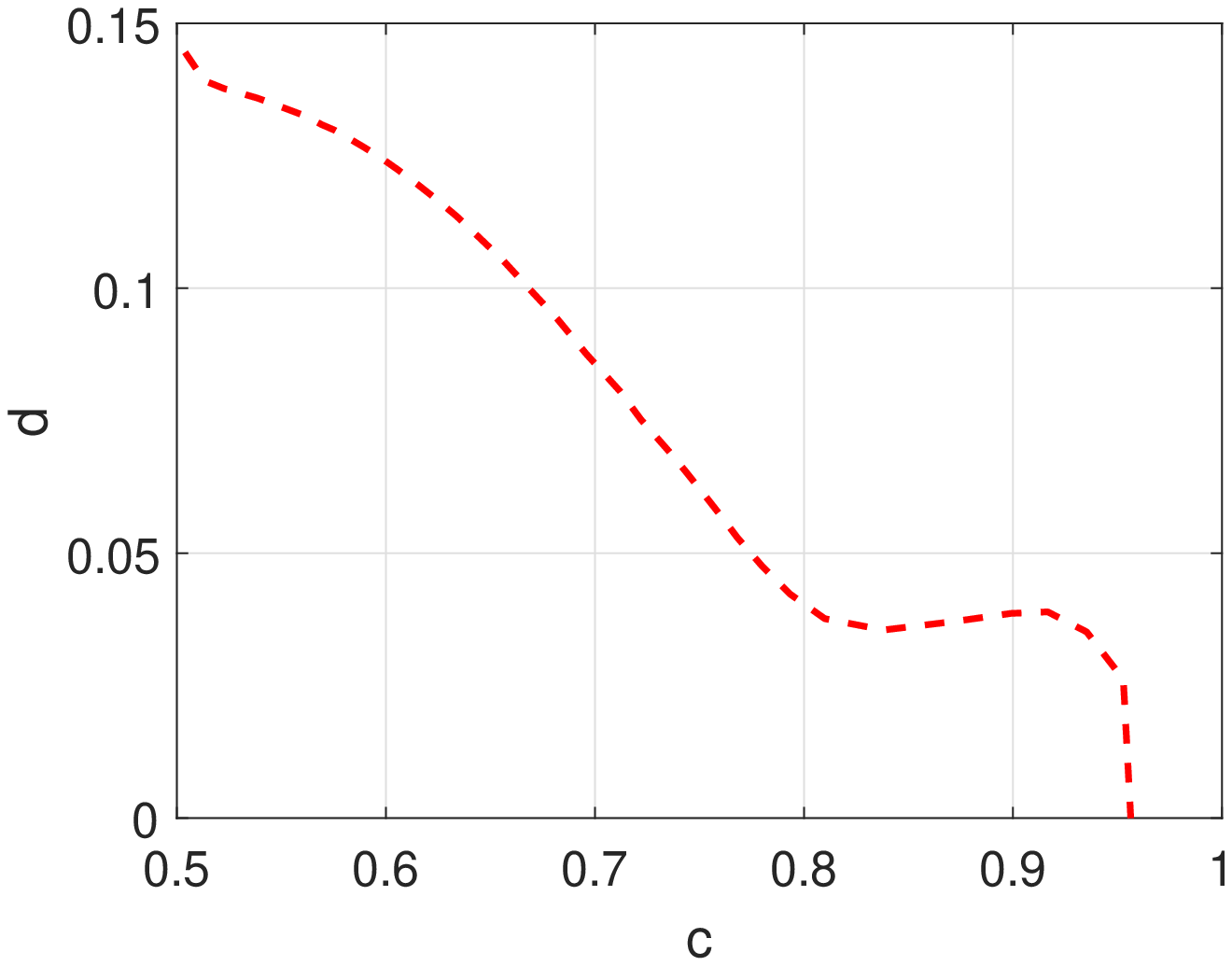}
 \end{minipage}
 \centering
 \includegraphics[width=2.9in]{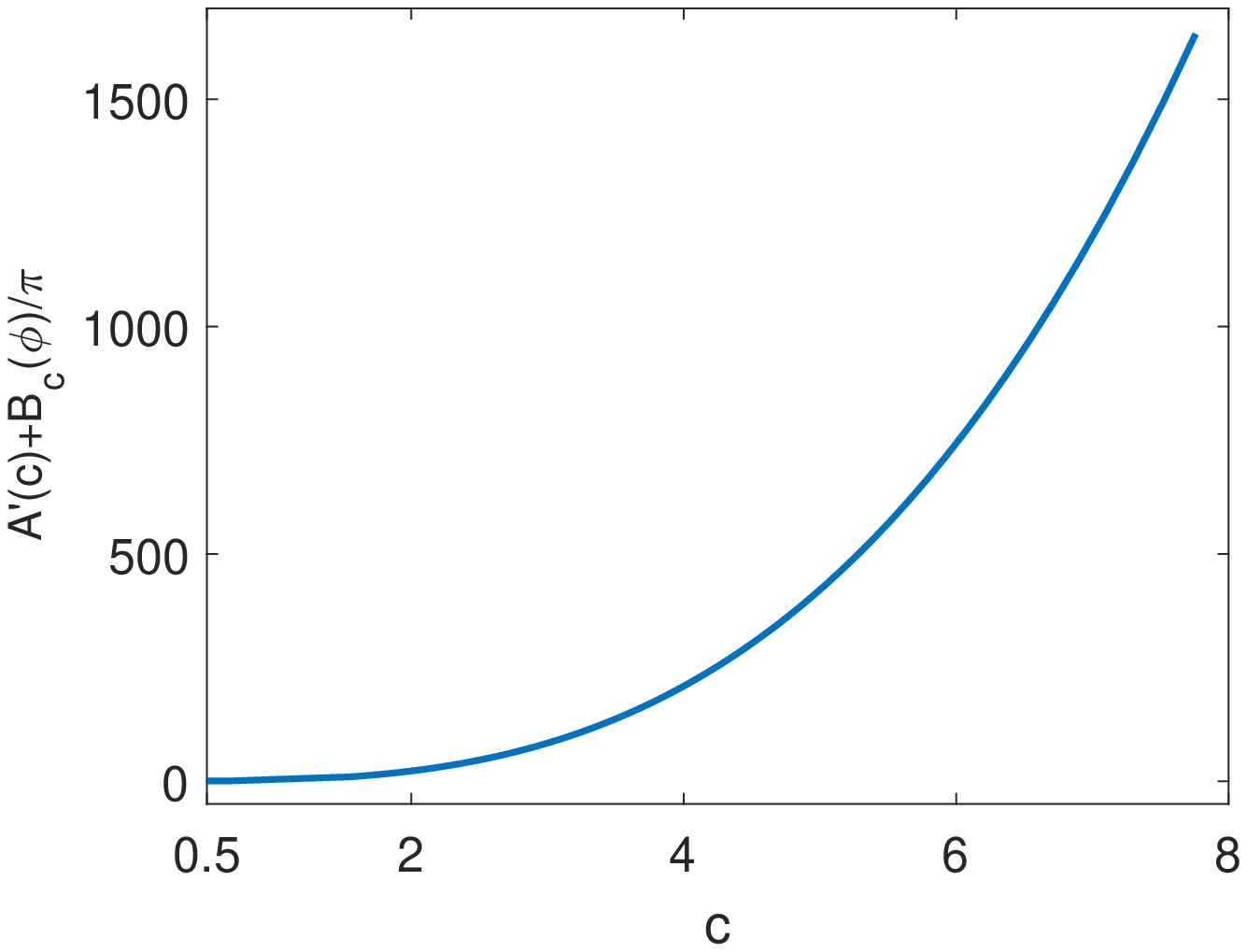}
 \caption{Variation of    $d$ versus $c$ (top left),  close-up look near critical value $c^*$   (top right) and the variation of $A^{\prime}(c)+\frac{1}{\pi} {\mathcal B}_c(\phi)$ with $c$ (bottom)  for $\alpha=0.45$.}
\end{figure}

\begin{figure}[]
\begin{minipage}[t]{0.49\linewidth}
   \includegraphics[width=3in]{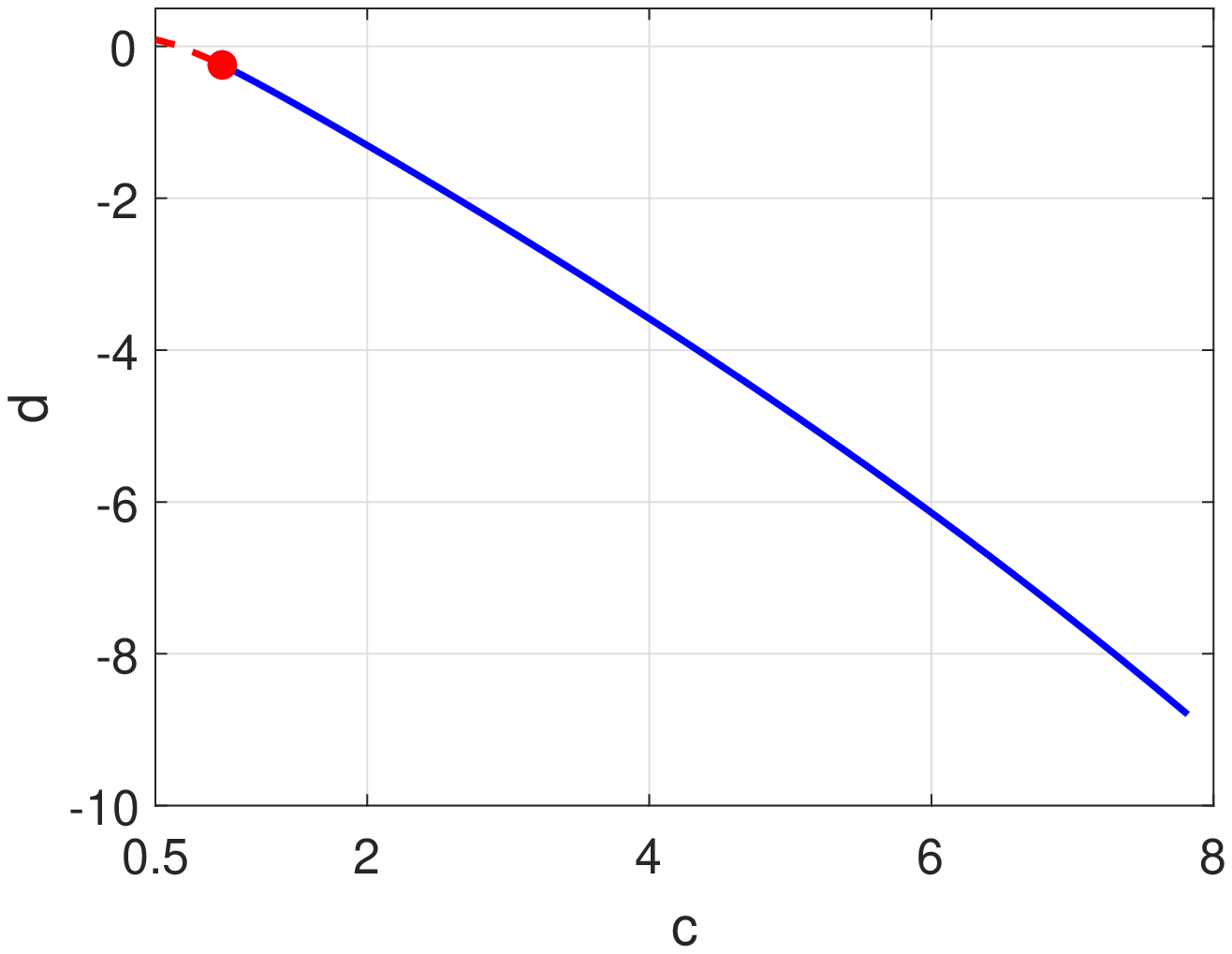}
 \end{minipage}
  \begin{minipage}[t]{0.49\linewidth}
   \includegraphics[width=3in]{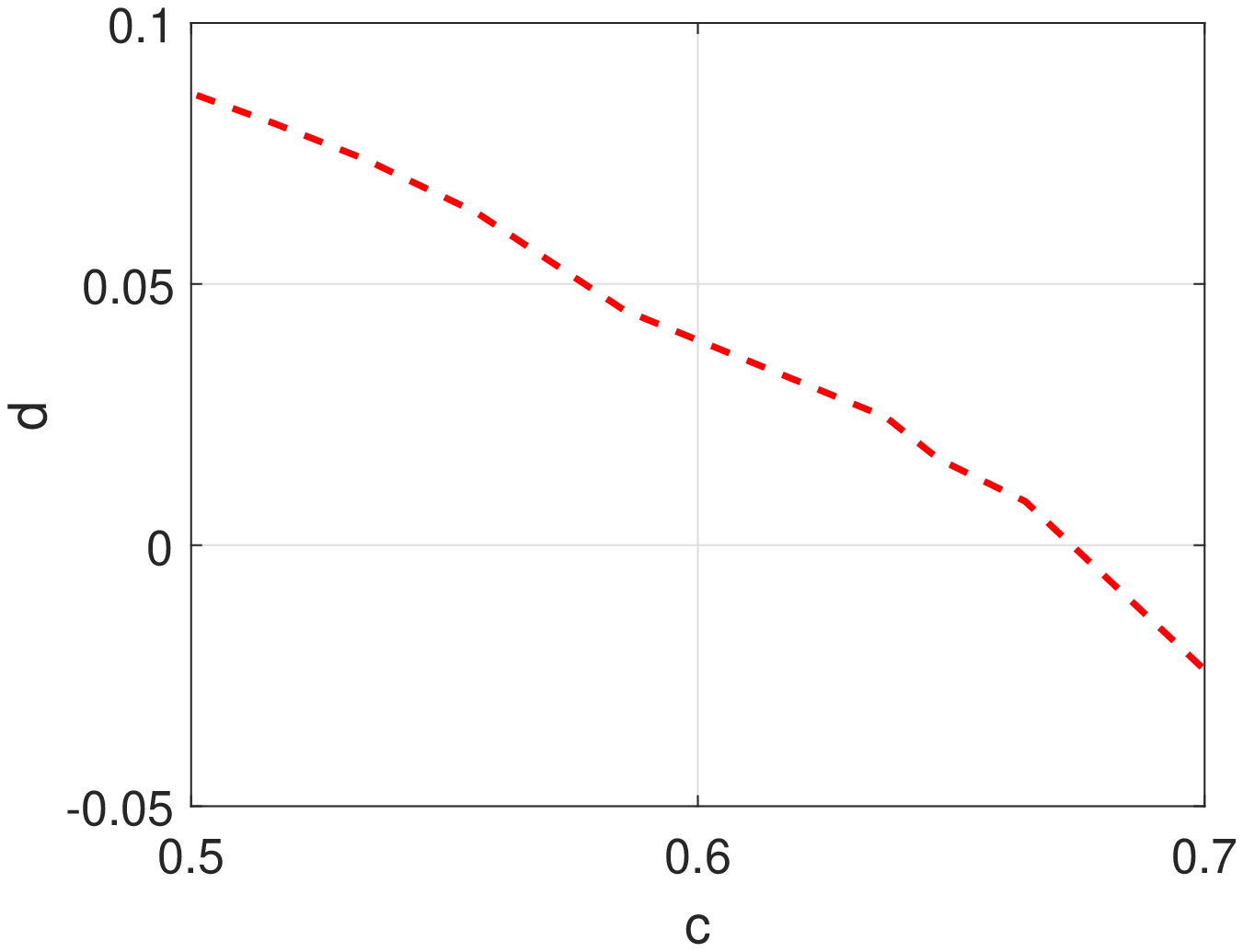}
 \end{minipage}
 \centering
 \includegraphics[width=3in]{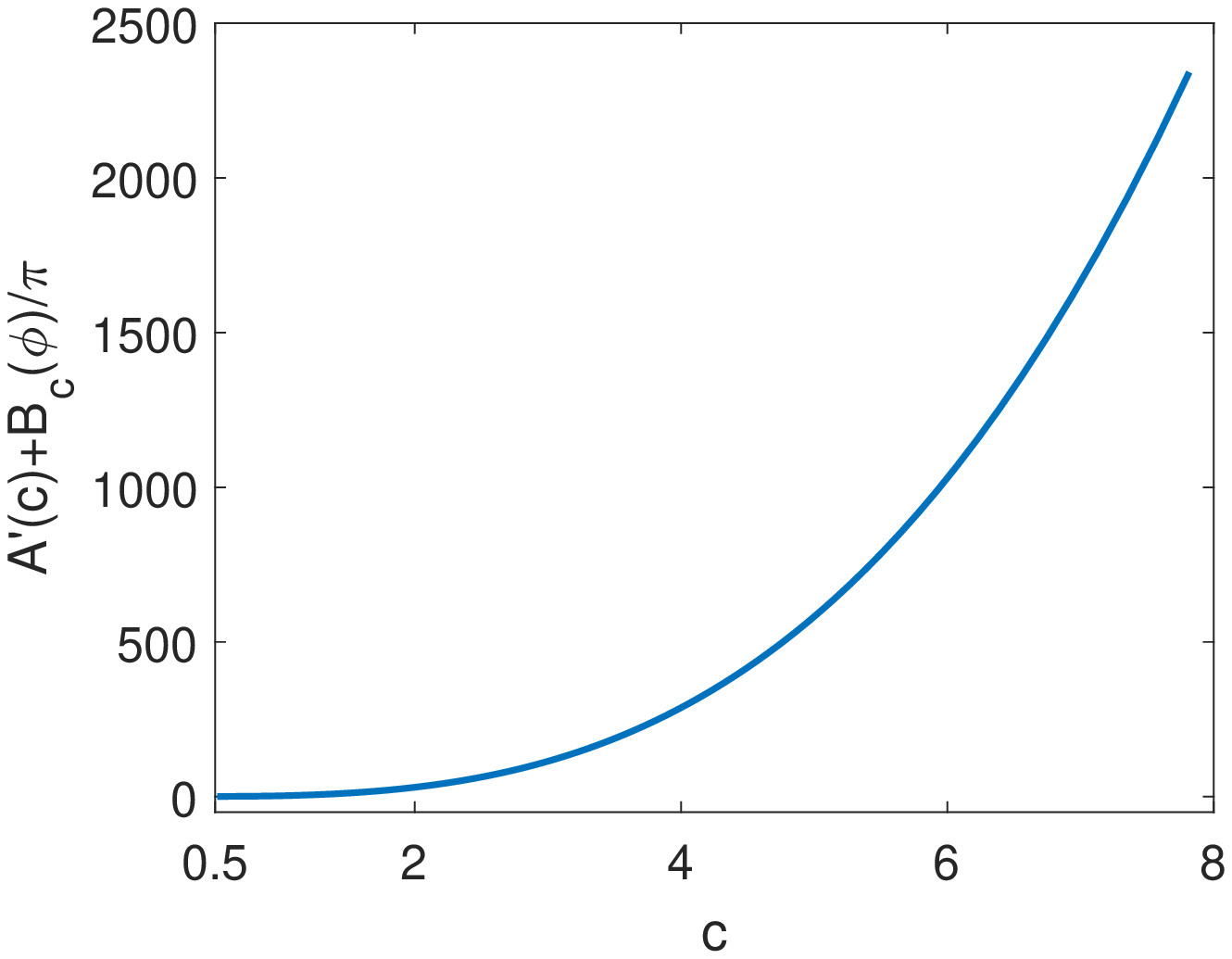}
 \caption{Variation of    $d$ versus $c$ (top left),  close-up look near critical value $c^*$   (top right) and the variation of $A^{\prime}(c)+\frac{1}{\pi} {\mathcal B}_c(\phi)$ with $c$ (bottom)  for $\alpha=0.5$.}
\end{figure}

In Figure 5.9, we present the variation of $d$ and the term  $A^{\prime}(c)+\displaystyle\frac{1}{\pi} {\mathcal B}_c(\phi)$ with $c$ for $\alpha=0.45$.  We observe that
$A^{\prime}(c)+\displaystyle\frac{1}{\pi} {\mathcal B}_c(\phi)$ is  positive for all values of $c$. However, $d$ is positive up to a critical speed $c^* \approx 0.953$ and then becomes negative. Therefore, $\mbox{det} S(0)$ is positive for $c\in (\frac{1}{2}, c^*)$ and negative for $c\in (c^*, +\infty)$.   We  also illustrate  the variation of $d$ and the term  $A^{\prime}(c)+\displaystyle\frac{1}{\pi} {\mathcal B}_c(\phi)$ with $c$ in Figure 5.10 for the fold point $\alpha_0=0.5$. The numerical results are very similar to the ones for $\alpha=0.45$ with $c^* \approx 0.67$.  As it is seen from the figures,  $\det S(0)<0$  for $c\in (c^*, +\infty)$ which indicates the spectral stability.  However, $\det S(0)>0$  for $c\in (\frac{1}{2},c^*)$    which yields that $n_0=0$ or $n_0=2$. Since in this case $d>0$, we have by Proposition $\ref{teoest}$ that $\phi$ is spectrally stable for all $c\in (\frac{1}{2}, c^*)$.

\section{Remarks on the Orbital Stability of Periodic Waves}

In this section, we present a brief discussion concerning the orbital stability of the periodic wave $\phi$ obtained by Lemma $\ref{minlema}$.\\
\indent Before stating the result, we need some preliminary tools. For functions $u$ and $v$ in $H_{per}^\frac{\alpha}{2}(\mathbb{T})$, we define $\rho$ as the ``distance'' between $u$ and $v$  given  by
\begin{equation*}
	\rho(u,v)=\inf_{y\in\mathbb{R}}||u-v(\cdot+y)||_{\Hn}.
\end{equation*}

\indent Our precise definition of orbital stability is given below.\\

\begin{definition}\label{defstab}
	We say that the periodic solution $\phi$ is orbitally stable in $\Hper$, by the periodic flow of \eqref{rbbm},  if for any $\ve>0$ there exists $\delta>0$ such that for any $u_0\in H_{per}^{\frac{\alpha}{2}}(\mathbb{T})$ satisfying $\|u_0-\phi\|_{\Hn}<\delta$, the solution $u(t)$ of \eqref{rbbm} with initial data $u_0$ exists globally and satisfies
	$$
	\rho(u(t),\phi)<\ve,
	$$
	for all $t\geq0$.
\end{definition}
\begin{remark}
	Our notion of orbital stability prescribes the existence of global solutions. Thus, from Section 2 we need to assume that $\alpha>1$.
\end{remark}
\indent Next, let us introduce the conserved quantity
\begin{equation}\label{V}
	V(u)=\frac{1}{2}\int_{-\pi}^{\pi}\left(u^2+\frac{u^3}{3}\right)dx=P(u)-E(u)
\end{equation}
and the auxiliary functional
\begin{equation*}
	Q(u):=-V(u)+(c-1-A(c))M(u).
\end{equation*}
\indent In what follows, we set
\begin{equation*}
	\Upsilon_0=\{u\in \Hper;\ \langle Q'(\phi),u\rangle=0\}.
\end{equation*}
Note that $\Upsilon_0$ is nothing but the tangent space to $\{u\in \Hper;  Q(u)=Q(\phi)\}$ at $\phi$. With these notations, the result in \cite[Theorem 2.1]{CNP} (see also \cite{ANP}) reads as follows.
\medskip
\begin{prop}\label{teoest} Suppose that $n(\mathcal{L})=1$ and $z(\mathcal{L})=1$. If there exists $\Phi\in H_{per}^{\alpha}(\mathbb{T})$ such that $\langle\mathcal{L}\Phi,\Psi\rangle=0$ for all $\Psi\in \Upsilon_0$ and $\langle\mathcal{L}\Phi,\Phi\rangle<0$, then $\phi$ is orbitally stable in $\Hper$ by the periodic flow of   $(\ref{rbbm})$.
\end{prop}

\begin{flushright}
	$\square$
\end{flushright}

\indent Finally, we can use the estimate $(\ref{ineqdiff2})$ to obtain the orbital stability of $\phi$. This last result gives us the proof of Theorem $\ref{maintheorem}$-v).

\begin{prop}
	Let $\alpha\in(1,2]$ be fixed. If $\ker\left(\mathcal{L}|_{X_0}\right)=[\phi']$ and $d\neq0$ for all $c>\frac{1}{2}$, the periodic wave $\phi$ obtained in Lemma $\ref{minlema}$ is orbitally stable in $\Hper$ in the sense of Definition \ref{defstab}.
\end{prop}
\begin{proof}
	Since $n(\mathcal{L})=z(\mathcal{L})=1$ by Proposition $\ref{propspec}$, we only need to use Proposition $\ref{teoest}$ for a convenient $\Phi$. Indeed, let us consider $\Phi=1+\phi$. For $$Q(u)=-V(u)+(c-1-A(c))M(u),$$ we obtain $\langle Q'(\phi),\phi'\rangle=0$ and $\langle\mathcal{L}(1+\phi),\Psi\rangle=0$ for all $\Psi\in \Upsilon_0$. If  $\langle \mathcal{L}(1+\phi),1+\phi\rangle<0$, one has the orbital stability of $\phi$ in the energy space $\Hper$. To calculate $\langle \mathcal{L}(1+\phi),1+\phi\rangle$, we employ the Poincar\'e-Wirtinger inequality and $(\ref{ineqdiff2})$ to obtain
	\begin{equation}\label{simpfor2}
		\begin{array}{lll}\langle\mathcal{L}(1+\phi),1+\phi\rangle&=&\displaystyle2\pi(c-1)- \displaystyle \int_{-\pi}^{\pi}c(D^{\frac{\alpha}{2}}\phi)^2+\left(c+1\right)\phi^2dx\\\\
			&\leq& 2\pi(c-1)-4\pi(2c+1)A(c)\\\\
			&\leq& 2\pi(c-1)-4\pi(2c+1)\left(c-\frac{1}{2}\right)=-2\pi c(4c-1)<0.
		\end{array}
	\end{equation}
	By Proposition $\ref{teoest}$ one has the orbital stability in the energy space $\Hper$.
\end{proof}

\section*{Acknowledgments} The authors are grateful to the two anonymous referees for their valuable suggestions and comments which greatly improved the presentation of the paper. S. Amaral was supported by the regular doctorate scholarship from CAPES. F. Natali is partially supported by CNPq (grant 304240/2018-4), Funda\c{c}\~ao Arauc\'aria (grant 002/2017) and CAPES MathAmSud (grant 88881.520205/2020-01).


\begin{thebibliography}{99}




\bibitem{ANP} {\sc G. Alves, F. Natali and A. Pastor}, \textit{Sufficient conditions for orbital stability of periodic traveling waves}, J. Diff. Equat., 267 (2019), pp. 879-901.

\bibitem{amb} {\sc V. Ambrosio}, \textit{On some convergence results for fractional periodic Sobolev spaces}, Opuscula Math., 40 (2020), pp. 5-20.




\bibitem{angulo1} {\sc J. Angulo}, \textit{Stability properties of solitary waves for fractional KdV and BBM equations}, Nonlinearity, 31 (2018), pp. 920-956.


\bibitem{ACN} {\sc J. Angulo, J.,  E. Cardoso Jr. and F. Natali}, \textit{Stability properties of periodic traveling waves for the intermediate long wave equation}, Rev. Mat. Iber. 33 (2017), pp. 417--448.


\bibitem{an-ban-scia1} {\sc J. Angulo, C. Banquet and M. Scialom},
\textit{The regularized Benjamin-Ono and BBM equations: Well-posedness and
nonlinear stability}, J. Diff. Equat., 250 (2011), pp. 4011-4036.







\bibitem{natali} {\sc J. Angulo and F. Natali}, \textit{Positivity properties of
the Fourier transform and the stability of
periodic travelling-wave solutions}, SIAM J. Math. Anal., 40 (2008),
pp. 1123--1151.


\bibitem{BBM} {\sc T.B. Benjamin, J.L. Bona and J.J. Mahony}, \textit{Model Equations for Long Waves in Nonlinear Dispersive Systems}, Phil. Trans. Royal Soc. London. Series A, Math. Phys. Sci., 272 (1972),  pp. 47–78




\bibitem{BD} {\sc G. Bruell and R.N. Dhara}, \textit{Waves of maximal height for a class of nonlocal equations with homogeneous symbol}, Indiana Univ. Math. Journal, to appear, (2020).


\bibitem{buffoni} {\sc B. Buffoni and J. Toland}, Analytic Theory of Global Bifurcation. Princeton Series in Applied
Mathematics. Princeton University Press, Princeton, NJ, 2003.

\bibitem{CJ2019} {\sc K. Claasen and M. Johnson}, \textit{Nondegeneracy and stability of antiperiodic bound states for fractional
	nonlinear Schr\"{o}dinger equations}, J. Diff. Eqs., 266 (2019), 5664--5712.


\bibitem{CNP} {\sc F. Crist\'ofani, F. Natali and A. Pastor}, \textit{Periodic Traveling-wave solutions for regularized dispersive equations: Sufficient conditions for orbital stability with applications}, Comm. Math. Sci., 18 (2020), pp. 613-634.


\bibitem{DK} {\sc B. Deconinck and T. Kapitula}
\textit{On the spectral and orbital stability of spatially periodic stationary
	solutions of generalized Korteweg-de Vries equations}, in Hamiltonian Partial Diff. Eq.
Appl., 75 285-322., Fields Inst. Comm., Springer, New York, 2015.

\bibitem{aduran} {\sc A. Duran}, \textit{An efficient method to compute solitary wave solutions of fractional Korteweg–de Vries equations}. Int J Comp Math.,  95 (2018), pp. 1362-1374.


\bibitem{AD} {\sc A Duran}, \textit{Numerical generation of periodic traveling wave solutions of
some nonlinear dispersive wave systems}. J. Comp. App. Math.,  316 (2017), pp. 29-39.


\bibitem{FL} {\sc R.L. Frank and E. Lenzmann}, \textit{Uniqueness of non-linear ground states for fractional Laplacians in $\mathbb{R}$}, Acta Math., 210 (2013), pp. 261–-318.

\bibitem{gallay} {\sc T. Gallay and M. H\u{a}r\u{a}gu\c{s}}, \textit{Stability of small periodic waves for the nonlinear Schrödinger equation}, J. Diff. Equat., 234 (2007), pp. 544-581.

\bibitem{gavage} {\sc S. Benzoni-Gavage, C. Mietka and L.M. Rodrigues}, \textit{Co-periodic stability of periodic waves in some Hamiltonian PDEs}, Nonlinearity, 29 (2016), pp. 3241–3308.

\bibitem{grillakis1} {\sc M. Grillakis, J. Shatah, and W. Strauss},
\textit{Stability theory
of solitary waves in the presence of symmetry I}, J. Funct.
Anal., 74 (1987), pp. 160-197.

\bibitem{hakkaev12} {\sc S. Hakkaev.}, \textit{Nonlinear Stability of Periodic Traveling Waves of the BBM System}, Comm. Math. Anal., 15 (2013), pp. 39-51.


\bibitem{haragus1} {\sc M. H\u{a}r\u{a}gu\c{s} and E. Wahl\'en}, \textit{Transverse instability of periodic and generalized solitary waves for a fifth-order KP model}, J. Diff. Equat., 262 (2017), pp. 3235-3249.


\bibitem{haragus} {\sc M. H\u{a}r\u{a}gu\c{s}}, \textit{Stability of periodic waves for the generalized BBM equation}, Rev. Roumaine Math. Pures Appl., 53 (2008), pp. 445--463




\bibitem{HP} {\sc V.M. Hur and A.K. Pandey}, \textit{Modulational instability in nonlinear non local equations of regularized long wave type}, Phys. D, 325 (2016), pp. 98-112.


\bibitem{hur} {\sc V.M. Hur and M. Johnson}, \textit{Stability of periodic traveling waves
	for nonlinear dispersive equations}, SIAM J. Math. Anal.,  47 (2015), 3528--3554.




\bibitem{johnson13} {\sc M. Johnson}, \textit{ Stability of small periodic waves in fractional KdV-type equations}, SIAM J. Math. Anal., 45 (2013), pp. 3168-3193.


\bibitem{kalisch} {\sc H. Kalisch}, \textit{Error analysis of a spectral projection of the regularized Benjamin–Ono equation}, BIT Numer. Math., 45 (2005), pp. 69–89.



\bibitem{HKi}{\sc H. Kielh\"ofer}, \textit{Bifurcation theory}, Appl. Math. Sci., Springer, New York, 2012.

\bibitem{LP} {\sc U. Le and D.E. Pelinovsky}, \textit{Convergence of Petviashvili's method near periodic periodic waves in the fractional Korteweg-de Vries equation}, SIAM J. Math. Anal., 51 (2019), pp. 2850-2583.

\bibitem{lin} {\sc Z. Lin}, \textit{Instability of nonlinear dispersive solitary waves},
J. Funct. Anal., 255 (2008), pp. 1091-1124.


\bibitem{LPS} {\sc F. Linares, D. Pilod and J-C. Saut} \textit{Dispersive perturbations of Burgers and hyperbolic
equations I: local theory}, SIAM J. Math. Anal., 46 (2015), pp. 1505–1537.

\bibitem{NPU1} {\sc F. Natali, D.E. Pelinovsky and U. Le}, \textit{Periodic waves in the fractional modified Korteweg--de Vries equation}, to appear in J. Dyn. Diff. Equat. (2022).


\bibitem{NPL} {\sc F. Natali, U. Le and D.E. Pelinovsky}, \textit{New variational characterization of periodic waves in the fractional Korteweg-de Vries equation}, Nonlinearity, 33 (2020), pp 1956-1986.

\bibitem{OBM} {\sc G. Oruc, H. Borluk, G. M Muslu}, \textit{The generalized fractional Benjamin-Bona-Mahony equation: Analytical and numerical results},
    Physica D: Nonlinear Phenomena, (2020), Article number:132499.


\bibitem{pel-book} {\sc D.E. Pelinovsky}, \textit{Localization in periodic potentials: from Schr\"{o}dinger operators
	to the Gross--Pitaevskii equation}, LMS Lecture Note Series,  390 Cambridge University Press, Cambridge, 2011.

\bibitem{PS} {\sc D.E.Pelinovski and Y.A Stepanyants}, \textit{Convergence of Petviashvili's iteration method for numerical approximation of stationary solution  of nonlinear wave equations}. SIAM J. Numer. Anal., 42 (2004), pp. 1110-1127.

\end{thebibliography}
\end{document}